\documentclass{amsart}

\usepackage{amsmath}
\usepackage{amsfonts}
\usepackage{amssymb}
\usepackage{graphicx}
\usepackage{mathrsfs}
\usepackage{pb-diagram}
\usepackage{amscd}
\usepackage{hyperref}
\usepackage{url}
\usepackage{caption}
\usepackage{subcaption}
\usepackage{color}
\usepackage[all]{xy}

\usepackage[abbrev,non-sorted-cites]{amsrefs}

\newtheorem{theorem}{Theorem}[section]
\newtheorem{lemma}[theorem]{Lemma}

\newtheorem{prop}[theorem]{Proposition}
\newtheorem{defn}[theorem]{Definition}
\newtheorem{example}[theorem]{Example}
\newtheorem{remark}[theorem]{Remark}

\numberwithin{equation}{section}

\newcommand{\Q}{\mathbb{Q}}
\newcommand{\fq}{\mathfrak{q}}
\newcommand{\fb}{\mathfrak{b}}
\newcommand{\fm}{\mathfrak{m}}
\newcommand{\fc}{\mathfrak{c}}
\newcommand{\N}{\mathbb{N}}
\newcommand{\Z}{\mathbb{Z}}
\newcommand{\R}{\mathbb{R}}
\newcommand{\C}{\mathbb{C}}

\newcommand{\Int}{\mathrm{Int}}

\newcommand{\parallelsum}{\mathbin{\!/\mkern-5mu/\!}}

\DeclareFontFamily{U}{mathx}{\hyphenchar\font45}
\DeclareFontShape{U}{mathx}{m}{n}{
      <5> <6> <7> <8> <9> <10>
      <10.95> <12> <14.4> <17.28> <20.74> <24.88>
      mathx10
      }{}
\DeclareSymbolFont{mathx}{U}{mathx}{m}{n}
\DeclareFontSubstitution{U}{mathx}{m}{n}
\DeclareMathAccent{\widecheck}{0}{mathx}{"71}
\DeclareMathAccent{\wideparen}{0}{mathx}{"75}

\begin{document}

\author[Hong]{Hansol Hong}
\address{Department of Mathematics \\ Yonsei University}
\email{hansolhong@yonsei.ac.kr, hansol84@gmail.com}
\author[Lin]{Yu-Shen Lin}
\address{Department of Mathematics and Statistics\\ Boston University}
\email{yslin@bu.edu}
\author[Zhao]{Jingyu Zhao}
\address{Center of Mathematical Sciences and Applications\\Harvard University}
\email{jzhao@cmsa.fas.harvard.edu, jzhao0105@gmail.com}

\begin{abstract}
We provide an inductive algorithm to compute the bulk-deformed potentials for toric Fano surfaces via wall-crossing techniques and a tropical-holomorphic correspondence theorem for holomorphic discs. As an application of the correspondence theorem, we also prove a big quantum period theorem for toric Fano surfaces which relates the log descendant Gromov-Witten invariants with the oscillatory integrals of the bulk-deformed potentials.

\end{abstract}

\title[Bulk-deformed potentials for toric Fano surfaces]{Bulk-deformed potentials for toric Fano surfaces, wall-crossing and period}

\maketitle
\section{Introduction}

The mirror of a Fano symplectic manifold $X$ is typically given by a Landau-Ginzburg model $W:\check{X}\rightarrow \mathbb{C}$, where $\check{X}$ is a smooth quasi-projective complex manifold and $W$ is a holomorphic function on $\check{X}$ called the superpotential. For toric Fano manifolds, such Landau-Ginzburg mirrors were first written down by Hori-Vafa \cite{HV}. Then it is proved by Cho-Oh \cite{CO} that the Hori-Vafa superpotential $W$ can be derived from the weighted counts of Maslov index two holomorphic discs with boundaries on Lagrangian torus fibers of the moment map. Explicitly, the Hori-Vafa potential is a Laurent polynomial defined on $(\C^*)^n$ which can be described in terms of the combinatorial data of the moment polytope. Later, this was generalized to toric semi-Fano manifolds by the work of Chan-Lau-Leung-Tseng \cite{CLLT} and to general toric manifolds  \cite{FOOO_t} and their bulk-deformations \cite{FOOO_bulk} by Fukaya-Oh-Ohta-Ono.

%
The latter also proved the closed string mirror symmetry conjecture for all toric manifolds in \cite{FOOO_toric} which states as follows. 
For each $\fb \in H^{even}(X),$ the bulk-deformed potential $W^{\fb}$ gives rise to a formal deformation of the Hori-Vafa potential $W$  by counting holomorphic discs with boundaries on the moment map fibers and passing through torus invariant cycles $D$ such that $[D]=PD(\fb)$.
Then there is a Kodaira-Spencer map $ \mathfrak{ks}$ which induces an isomorphism between the big quantum cohomology and the Jacobian ring of the bulk-deformed superpotential
   \begin{align*}
     \mathfrak{ks} \colon (QH^*(X) , *_{\fb}) \cong Jac(W^{\fb}).
   \end{align*} 
More recently, Smith \cite{Sm} found another algebraic method to prove the closed string mirror symmetry for toric varieties.
However, such an isomorphism as well as the bulk-deformed superpotential has not been studied in depth when the the cocycle $\fb$ is not torus invariant.


\indent In this paper, we study bulk deformations of the Hori-Vafa potentials of toric Fano surfaces by non-toric cycles $\fb$ in $H^4(X)$. Prior to our work, Gross \cite{G7} studied the bulk-deformed superpotential $W^{\fb}$ for $X=\mathbb{P}^2$ via counting of tropical discs passing through $k$ generic cycle constraints that represent $\fb$ in $H^4(\mathbb{P}^2)$. Inspired by this, we compute the $k$th order bulk-deformed potential $W_k$ for each $k\geq0$ rather than $W^\fb$ all at once, which counts $J$-holomorphic discs with $k$-interior marked points with boundaries on the Lagrangian toric fibers $L$. 
If we impose the condition that these $k$-interior marked points map to a given $k$-tuple $(q_1, \cdots, q_k)$ of points in generic positions, then the loci of Lagrangian torus fibers that bounds such discs with suitable Maslov index will give $1$-dimensional skeleton in the SYZ-base for $X$. 

More precisely, consider the moduli space $\mathcal{M}_{0,k}(L, \beta, J)$ of $J$-holomorphic discs passing through $k$ generic point constraints in $X$ and bounding $L$.
Its virtual dimension of is given by
\begin{equation}\label{eqn:dimfml}
\dim L + \mu(\beta) - 3 +2k - 4k.
\end{equation}
In the situation of our interest for which $\mu(\beta) = 2k$ and $\dim L =2$, \eqref{eqn:dimfml} becomes negative, and one concludes that generic fibers of the SYZ fibration $\pi$ do not bound such Maslov $2k$ discs passing through $k$ generic marked points in $X$. We will call such a disc \emph{a generalized Maslov index $0$} disc, since it behaves similarly to Maslov zero discs in usual wall-crossing phenomena in Floer theory \cite{Auroux_T, Auroux_survey}. 

In this circumstances, the bulk-deformed potential $W_k$ is no longer a well-defined Laurent polynomial on $(\C^*)^2,$ but it experiences discontinuities. Explicitly, as the Lagrangian torus fibers vary across different chambers separated by walls in the Log base (the Legendre transform of the moment map image), the bulk-deformed potentials $W_k$ will undergo so-called wall-crossing or cluster transformations. In the case of toric Fano surfaces, these cluster transformations are simply of the form
\begin{equation}\label{eqn:cluster}
\begin{array}{lcl}
 z_1 &\mapsto& z_1 \\
 z_2 &\mapsto& f(z_1,z_2)z_2,
\end{array}
\end{equation}
%
 for some holomorphic function $f$ on $(\C^*)^2$ up to a coordinate change.
 
To compute the bulk-deformed potential $W_k$ order by order, we first prove a tropical-holomorphic correspondence theorem which gives a combinatorial description of the locus of the Lagrangian torus fibers that bound generalized Maslov zero discs in the Log base (See section \ref{sec:wall_crossing_Wk} for details) via tropical disc counting. Such descriptions of tropical wall structures agree with Gross' work \cite{G7} in the case of $\mathbb{P}^2$, and our theorem below can be viewed as its generalization to all toric Fano surfaces.

\begin{theorem} \label{thm:correspondence}
For a toric Fano surface $X$, the chamber structure from generalized Maslov index $0$ discs is homeomorphic to the one given by tropical disc counting. 
\end{theorem}

It should be remarked that there is a similar work by Nishinou \cite{N2}, who used algebraic geometric approach to define the counting of tropical discs on toric varieties as certain log Gromov-Witten invariants and establish the tropical-holomorphic correspondence. In contrast, we begin by holomorphic disc counting via Lagrangian Floer theory and establish the tropical-holomorphic correspondence under certain limit. We also provide a counter-example to such a tropical-holomorphic correspondence away from the tropical limit in Section \ref{sec:9999}. After taking the tropical limit, one has the following correspondence theorem.
\begin{theorem}[Theorem \ref{thm:wall} and \ref{9999}] 
Let $X$ be a toric Fano surface.
  Given generic points $q_1,\cdots, q_k$ in $(\mathbb{C}^*)^2\subseteq X$ and a generic point $u$ in $\mathbb{R}^2$, we denote the Lagrangian torus fiber by $L_u:=\pi^{-1}(u)$. Then for any $t\gg 1$, 
   \begin{enumerate}
     \item there exists a neighborhood $\mathcal{U}_t \in \R^2$ depending on $t,$ which contains the holomorphic walls and deformation retracts to the union of the tropical wall structure (described below Theorem \ref{thm:wall}). 
     \item the (rescaled) bulk-deformed superpotential defined in Definition \ref{defn:bulk_potential} as
    \begin{align*}
   W_{k}(u):=W^{H_t^{-1}(q_1),\cdots H_t^{-1}(q_k)}_k(H^{-1}_t(u))
      \end{align*} can be computed tropically if $u$ is in the complement of $\mathcal{U}_t$, 
   \end{enumerate}
   where $H_t \colon (\C^*)^2 \rightarrow (\C^*)^2$ defined by $(x, y) \mapsto (|x|^{\frac{1}{\log{t}}} \frac{x}{|x|}, |y|^{\frac{1}{\log{t}}}\frac{y}{|y|})$ is the diffeomorphism of $(\C^*)^2$ introduced by Mikhalkin \cite{M2} and the superscript of $W_k$ indicates the bulk-deformed potential is computed subject to the generic $k$ point constraints.
\end{theorem}


As the tropical discs can be determined combinatorially, we obtain an inductive algorithm to compute the bulk-deformed potential $W^{\fb}$ for $X$ using the wall-crossing maps. 
They are obtained by Fukaya's pseudo-isotopies \cite{F1} applied to the bulk-deformed Fukaya $A_{\infty}$-algebra $\{ \fm_k^{\fb}\}_{k \geq 0}$ defined in \cite{FOOO_bulk}, and in particular, 
are $A_\infty$-homomorphisms that are compatible with the bulk-deformed $A_\infty$-structures on the Fukaya algebras of $L_u$. As a result, the expressions of the bulk-deformed potentials $W_k$ in adjacent chambers of the Log base differ by the wall-crossing maps. The details of this computation will be given in Section \ref{sec:pseudo_iso} below.


Furthermore, given our algorithmic computation of $W_k$ for each $k$ and the tropical-holomorphic correspondence theorem, one observes that although the non-toric bulk-deformed superpotential may depend on the choice of the moment torus fiber in different chambers of the Log base $\R^2$, the corresponding LG periods, or oscillatory integrals
      \begin{align*}
         \int_{\Gamma} e^{W_k(u)/\hbar}\Omega
      \end{align*} 
are well-defined for $\Gamma\in H_n(\check{X},\mbox{Re}(W_k/\hbar)\ll 0)$ and $\Omega=\frac{dz_1}{z_1} \wedge \frac{dz_2}{z_2}$ is the standard holomorphic volume form on $(\C^*)^2$. In the case when $\Gamma=\frac{1}{(2\pi i)^2}T^2$ in $(\C^*)^2,$ the independence of the chamber structure can be deduced from the following big quantum period theorem, which states that there is a direct relation between such oscillatory integrals of $W_k$ with log Gromov-Witten descendants of the toric Fano surface $X$.

\begin{theorem}[Theorem \ref{thm:period_thm}]\label{thm:big_quantum_period}
Let $X$ be a toric Fano surface and let $L_u$ be any Lagrangian torus fiber of the moment map $\pi.$ The bulk-deformed potential $W_k(u)$ associated to $L_u$ satisfies  for any $k \in \N$
\begin{eqnarray}
&& \ \  \frac{1}{(2\pi i)^2}\int_{T^2}e^{W_k(u)/\hbar}\frac{dx_1}{x_1}\wedge \frac{dx_2}{x_2}\nonumber \\
&& =1+\sum_{I} \sum_{m \geq 2} \sum_{\Delta: |\Delta|=m+n} \frac{1}{|\mbox{Aut}(\Delta)|}\langle p_1,\cdots,p_n, \psi^{m-2}u\rangle^{X(\log{D})}_{\Delta,n}\hbar^{-m}t_I,
\end{eqnarray}
where $t_I=t_{i_1}\cdots t_{i_n}$ and $\langle \cdot ,\cdots, \cdot \rangle^{X(\log{D})}_{\Delta,n}$ denotes the log Gromov-Witten invariants with $n$ affine constraints and for $\Delta$ such that $|\Delta|=m+n.$ The formal parameter $\hbar$ is usually referred to as the \textit{descendant variable} in Gromov-Witten theory.
\end{theorem}


Such relations between descendant Gromov-Witten invariants and oscillatory integrals have been studied in closed-string mirror symmetry by Givental \cite{Givental} and Barannikov \cite{Barannikov} for $X=\mathbb{P}^n$ under the name of quantum periods and semi-infinite variations of Hodge structures. Assuming closed-string mirror symmetry for Fano manifolds are proved, for instance for $X=\mathbb{P}^2$ by Barannikov \cite{Barannikov2}, the work of Gross \cite{G7} defined the tropical bulk-deformed superpotential and studied the flat coordinates on the miniversal deformation space (or the formal deformation space) of the Hori-Vafa potential for $\mathbb{P}^2$ and matched these oscillatory integrals with tropical descendant invariants that he defined (See Remark \ref{rmk:Gross} below for more details). More recently, the work of Chan-Ma \cite{Chan_Ma} obtained Gross's scattering diagrams for $\mathbb{P}^2$ from the asymptotic analysis of the Maurer-Cartan equations associated to the mirror dgla (See  \ref{subsec:reltother} for details).  For all monotone symplectic manifolds, the work of Tonkonog \cite{T5} proved the small quantum period theorem which states that the $1$-pointed descendant Gromov-Witten invariants match with the oscillatory integrals of its Laudu-Ginzburg mirrors which are not necessarily affine varieties. Independently, Mandel \cite{M6} studied the Frobenius conjecture proposed by Gross-Hacking-Keel \cite{GHK}. As a corollary, Mandel also derived a similar small quantum period theorem for cluster varieties.


\subsection*{Outline of the paper} In Section \ref{sec:pseudo_iso}, we first review the construction by Fukaya-Oh-Ohta-Ono of the bulk deformations of the Fukaya's algebras associated to the Lagrangian torus fibers. For toric Fano surfaces, we prove that the weak Maurer-Cartan elements associated to the bulk-deformed Fukaya's algebra is given by $H^1(L).$ In particular, we provide a bulk-deformed version of Fukaya's trick in Proposition \ref{1086}, which is crucial for the computations of the wall-crossing formula later. In Section \ref{sec:trop_sec}, we review the basic definitions of the tropical curve counting and the bulk-deformed counterparts of tropical disc counting. In Section \ref{sec:wall_crossing_Wk}, we explicitly compute the initial wall structures on the Log base for $\mathbb{P}^2.$ Here we demonstrate via a computation using \cite{CO} that the walls on the Log base are not given by straight lines in general, which complicates the determination of the wall structures in the holomorphic setting due to the ambiguities of the shape of the walls. In order to solve this problem, we appeal to the tropical degeneration of the standard complex structure on $(\C^*)^2$ defined by Mikhalkin in \cite{M2} in Section \ref{sec:HTcorres}, where  we prove the tropical-holomorphic correspondence theorem by making use of an anti-symplectic involution of $(\C^*)^2$. 
Finally, in Section \ref{sec:period_thm} we prove the big quantum period theorem using the tropical-holomorphic correspondence Theorem \ref{thm:correspondence} for disc countings and match them with tropical descendant invariants with one higher valency vertex, considered by Gross in \cite{G2}. Based on the recent work by Mandel-Ruddat \cite{MR16}, we obtain that such tropical descendant invariants in fact agree with descendant log Gromov-Witten invariants.

\subsection*{Acknowledgements}We thank Denis Auroux and Shing-Tung Yau for their interests in this work. We are grateful to Mohammed Abouzaid, Cheol-Hyun Cho, Yoosik Kim, Siu-Cheong Lau, Travis Mandel, Kyler Siegel and Hsian-Hua Tseng for useful discussions. The work is substantially supported by the Simons
Foundation grant (\# 385573, Simons Collaboration on Homological Mirror Symmetry). We also thank Harvard CMSA for the hospitality and wonderful research environment.

\subsection*{Notations}
\begin{enumerate}
\item[$\bullet$] Let $N\cong \mathbb{Z}^2$ be a lattice and $M:=\mbox{Hom}(N,\mathbb{Z})$ be the dual lattice. Denote $N_{\mathbb{R}}=N\otimes \mathbb{R}$ and $M_{\mathbb{R}}=M\otimes \mathbb{R}$. Let $\Sigma\in M_{\mathbb{R}}$ be the toric fan and $X:=X_{\Sigma}$ be the associate toric variety.
\item[$\bullet$] Let $\pi \colon X \rightarrow P$ be the moment map fibration and $P$ denotes the moment polytope. For simplicity, we will denote by $\pi$ its Legendre transform $\pi \colon X\backslash D \cong (\C)^* \rightarrow \R^2, \ (x,y) \mapsto (\log(x), \log(y))$.
\item[$\bullet$] Fix a Lagrangian torus fiber $L_u=\pi^{-1}(u),$ we denote by $\mathcal{M}_{k+1, l}(L_u, \beta, J)$ be the moduli space of stable bordered $J$-holomorphic discs representing the class $\beta$ in $H_2(X, L_u)$ with $k+1$ boundary marked points and $l$ interior marked points with respect to the \textit{standard complex structure} $J$ on $X.$
\item[$\bullet$] We denote the $k$ generic marked points by $q_1 ,\cdots,q_k$ in $(\C^*)^2$ and their images under the Log map as $p_i =\log(q_i)$ for $i=1,2,\cdots, k.$
\end{enumerate}

\section{Preliminaries} \label{sec:pseudo_iso}

\subsection{Basic Floer theory}

In this section, we review Fukaya's $A_{\infty}$-algebra associated to each Lagrangian torus fiber $L_u:=\pi^{-1}(u)$ and pseudo-isotopies between such $A_{\infty}$-algebras defined in \cite{F1}. In particular, we will construct the $A_{\infty}$-homomorphism via pseudo-isotopies of $A_{\infty}$-algebras in the bulk-deformed setting. 

\subsubsection{The de Rham model for Fukaya's $A_{\infty}$-algebra}
Let $\Lambda_0$ denote the Novikov ring over the real numbers $\R$,
\begin{equation}
\Lambda_0:=\{\sum_{i=1}^{\infty} a_i T^{\lambda_i} \mathbin{|}  \lambda_i \geq 0, \displaystyle\lim_{i \rightarrow \infty} \lambda_i= \infty \text{ and }a_i \in \R\}. \nonumber
\end{equation}
There is a non-Archimedean valuation $\mathrm{val} \colon \Lambda_0 \rightarrow \R,$  
\begin{equation}
\mathrm{val}(\sum_{i} a_i T^{\lambda_i})= \inf\{ \lambda_i \mathbin{|} a_i \neq 0 \} \text{ and } \mathrm{val}(0)=\infty. \nonumber
\end{equation}
The maximal ideal of $\Lambda_0$ is defined as $\Lambda_+:=\mathrm{val}^{-1}((0, \infty)).$ 

\begin{defn}
Suppose that $G$ is a discrete monoid with a homomorphism $\omega \colon G\rightarrow \R$ such that
\begin{equation}
|\omega^{-1}([0,a))|< \infty, \ \ \forall a \in \R. \nonumber
\end{equation}
Let $C$ be a $\Z/2\Z$-graded vector space over $\Lambda_0.$ A $G$-gapped filtered $A_{\infty}$-structure on $C$ is a sequence of homomorphisms $\{ \fm_{k,\beta}\}_{k \geq 0, \beta \in G}$ of degree $1$
\begin{equation}
\fm_{k,\beta} \colon B_kC[1]:=\underbrace{ C[1] \otimes \cdots\otimes C[1]}_{k \text{ times }  }\rightarrow C[1], \nonumber
\end{equation}
where $(C[n])^d=C^{d+n}$ is a degree shift, and these operations induce a coderivation of degree $1$ \begin{equation}
\hat{m}:=\sum_{k \geq 0} \sum_{\beta \in G} \fm_{k, \beta}T^{\omega(\beta)} \colon BC[1]=\bigoplus_{k=0}^{\infty}B_kC[1] \rightarrow C[1] \nonumber
\end{equation}
satisfying $\hat{m}\circ \hat{m}=0$ on $BC[1]$.
\end{defn}
\begin{defn}
An $A_{\infty}$-homomorphism between two $G$-gapped filtered $A_{\infty}$-algebras $(C, \{ \fm_{k, \beta}\})$ and $(C', \{ \fm_{k, \beta}'\})$ are defined by a sequence of homomorphisms
\begin{equation}
f_{k,\beta} \colon B_kC[1] \rightarrow C'[1] \nonumber
\end{equation} 
of degree $0$ such that the induced map 
\begin{equation}
\hat{f}=\sum_{k \geq 0} \sum_{\beta \in G} f_{k, \beta} \colon BC[1] \rightarrow BC'[1] \nonumber
\end{equation}
satisfies $\hat{m}'\circ \hat{f}=\hat{f} \circ \hat{m},$ where $\hat{m}$ and $\hat{m}'$ are the coderivations on $BC[1]$ and $BC'[1]$ induced by $\{ \fm_{k, \beta}\}$ and $\{ \fm_{k, \beta}'\}$ respectively.
\end{defn}

\indent For a fixed Lagrangian fiber $L$ of toric moment map $\pi \colon X \rightarrow P,$ we denote $H$ as the graded vector space generated by smooth singular cycles $C$ of \textit{even} dimensions in $X$ of degree $n-\dim(C)$ over $\R$ and $\Omega(L)$ as the de Rham cochain complex of $L$.  For each $\fb \in H,$ there is a $G$-gapped filtered $A_{\infty}$-algebra structure on $\Omega(L)\widehat{\otimes}\Lambda_0$ constructed via bulk deformations by Fukaya-Oh-Ohta-Ono in \cite{FOOO_bulk} as follows.  For an $\omega$-tamed almost complex structure $J$ and a fixed relative homology class $\beta \in H_2(X, L;\Z),$ we let $\mathcal{M}_{k+1,l}(L, \beta, J)$ be the moduli space of stable $J$-holomorphic discs with $k+1$ boundary marked points and $l$ interior marked points representing the class $\beta$. An element of $\mathcal{M}_{k+1,l}(L, \beta, J)$ is of the form
\begin{equation}
(\Sigma, \psi, \{ z_i^+\mathbin{|} i=1,\cdots, l\}, \{z_j\mathbin{|} j=0,1,\cdots, k\}), \nonumber
\end{equation} 
where $\Sigma$ is a connected tree of discs, and $\psi \colon (\Sigma, \partial \Sigma) \rightarrow (X, L)$ is a $J$-holomorphic curve with $z_i^+ \in \psi(\Int(\Sigma))$ and $z_j \in \psi(\partial \Sigma)$ and $\psi$ has only finite automorphisms. One further imposes the condition that the ordering of the boundary marked points $z_j$ agrees with the chosen orientation of $\partial \Sigma.$ It is shown in \cite[Section 7.1]{FOOO} that it has a Kuranishi structure with boundaries and corners of dimension $n+2l+k+1-3+\mu(\beta),$ where $\mu(\beta)$ is the Maslov index of $\beta.$ There are interior evaluation maps
\begin{equation}
ev^{int} \colon \mathcal{M}_{k+1,l}(L, \beta, J) \rightarrow X^l, \ \ (\Sigma, \psi, \{ z_i^+\}, \{ z_j\}) \mapsto (\psi(z_1^+), \cdots, \psi(z_l^+ )).\nonumber
\end{equation}
One fixes a basis $\{\textbf{f}_i\}_{i=1}^N$ of $H,$ then  each singular smooth cycle $\fb$ in $H$ can be written in terms of the basis as
\begin{equation}
\fb=\sum_i t_i \textbf{f}_i \nonumber.
\end{equation}
Since the ordering of the interior marked points $z_i^+$ is not fixed a priori, we denote by $\mathrm{Maps(l, \underline{N})}$ the set of all maps $\textbf{q} \colon \{1,\cdots, l \} \rightarrow  \underline{N}=\{1,\cdots, N\}$ and define the fiber product as
\begin{equation} \label{eqn:fibre_prod}
\mathcal{M}_{k+1,l}(L,\beta, J; \textbf{q}):=\mathcal{M}_{k+1,l}(L, \beta, J)_{ ev^{int}} \times_{X^l} \prod_{i=1}^l f_{\textbf{q}(i)}.  
\end{equation}
Its virtual dimension is $n+2l+k+1-3+\mu(\beta)-\sum_i (2n-\dim(f_{\textbf{q}(i)})).$ We include several properties of the moduli space in the case of toric Fano surfaces in Lemma \ref{799}, \ref{798} and \ref{lem:no_bubbles} for later use.
\begin{lemma} \label{799}
   Let $\beta' \in H_2(X, L_u)$ be a relative homology class. If it satisfies $\beta'=k\beta$ for some integer $k>1$ and some primitive homology class $\beta$ with $\mu (\beta')=k \mu (\beta)=2l+2$, then the moduli space $\mathcal{M}_{k+1,l}(L_u, \beta', J;\textbf{q})$ for the bulk-deformed potential contains no holomorphic discs with multiple covers for generic point constraints.
\end{lemma}
\begin{proof}
 
Given that $\beta'=k\beta$ for some $k >1$ and primitive homology class $\beta$ that satisfy the Maslov index relation $\mu(\beta')=k \mu(\beta)=2l+2$. One can consider the moduli space $\mathcal{M}_{0,l} (L_u,\beta)$ is a smooth manifold without stable discs compactification in \cite{CO} for the primitive relative class $\beta.$ Since for $k>1$ the virtual dimension of this moduli space is $\mu(\beta)+2-3+2l \leq 3l$, which implies that the domain of the interior evaluation map
        \begin{align*}
    \mathcal{M}_{0,l}(L_u,\beta, J)  \rightarrow X^{l}.
        \end{align*}
         has dimension less than $4l,$ while the target $X$ has the dimension $4l$. One concludes from Sard's theorem that every point in $X^l$ is a critical value, and the generic fiber of the boundary evaluation map is empty.  
\end{proof}
As a consequence, the moduli space $\mathcal{M}_{1,l}(L_u, \beta, J;\textbf{q})$ admits no automorphisms. 

\begin{defn} \label{1034} 
Fix $\textbf{q}=(q_1,\cdots, q_k)\in (\mathbb{C}^*)^2$ and let $L$ be a Lagrangian in $X$ which is isotopic to a moment map fiber. We define the holomorphic walls as follows
    \begin{align*}
      \mathcal{W}^{holo}=\bigcup_{0 \leq n\leq k}\{u\in \Int(P)| \partial_1\mathcal{M}_{0,n}(L_u, \beta, J; \textbf{q}_n) \neq \emptyset \text{ for } \beta \in H_2(X,L_u), \ \mu(\beta)=2n\},
    \end{align*} where $\partial_1\mathcal{M}_{0,n}(L, \beta, J)$ denotes the real codimension one boundary of the moduli space $\mathcal{M}_{0,n}(L, \beta, J)$ and $\textbf{q}_n=(q_{i_1}, \cdots, q_{i_n})$ is a $n$-tuple of generic marked points for some ordered subset $\{i_1, \cdots, i_n \} \subset \{1,2, \cdots, k\}$.
   
\end{defn}

\begin{lemma} \label{798} For generic point constraints,
   the holomorphic wall $\mathcal{W}^{holo}$ has dimension at most one and there exists no moment fibre bounding negative Maslov index discs. 
\end{lemma}
\begin{proof}
   The proof is similar to Lemma \ref{799}. Consider the moduli space 
     \begin{align*}
        \mathcal{M}=\bigcup_{u\in  \Int(P)} \bigcup_{\beta; \mu(\beta)=2l} \mathcal{M}_{0,l}(L_u,\beta,J).
     \end{align*} 
 Then the fiber of the interior evaluation map 
    \begin{align*}
     \mathcal{M}\rightarrow X^{l}
    \end{align*} has generic fiber of dimension less equal to $1$. The holomorphic wall $\mathcal{W}^{holo}$ is the projection of the generic fiber to the base of the moment map fibration, which is also of dimension at most one. The last part of the lemma follows similarly.
 \end{proof}
 
\begin{lemma} \label{lem:no_bubbles}
Assume that $\mu (\beta)=2l+2$.
   Then the interior of $\mathcal{M}_{1,l}(L_u,\beta, J)$ has no sphere bubbles for generic point constraints $\textbf{q}$. 
\end{lemma}
\begin{proof} This is again a dimension count argument.
   Suppose that the stable map 
   $$f:(\Sigma,\partial\Sigma) \to (X,L_u) \in \mathcal{M}_{1,l}(L_u,\beta, J)$$
    admits a sphere bubble, that is, $\Sigma=\Sigma_1\cup \Sigma_2$ for some stable disc $\Sigma_1$ and $\Sigma_2\cong \mathbb{P}^1$. Let $l_1$and $l_2$ be the number of interior marked points on $\Sigma_1$ and $\Sigma_2$, respectively, and let $\beta_1=f_*(\Sigma_1)$ and $\beta_2=f_*(\Sigma_2)$ so that $\beta=\beta_1+\beta_2$. In particular, we have $\mu(\beta_1) + \beta_2 \cdot K_X = \mu(\beta) = 2l+2$.
   
   From the similar arguments in Lemma \ref{798}, one observes that the moduli space $\mathcal{M}_{0,l_1}(L_u,\beta_1, J)$ is empty only if
   \begin{equation} \label{eqn:ineq_1}
   \mu (\beta_1)+2-3-2l_1<0.
   \end{equation}
   On the other hand, it is known that the moduli space of stable rational curves in class $\beta_2$ passing through $l_2$ points is empty only if 
   \begin{equation} \label{eqn:ineq_2}
   l_2>\beta_2\cdot K_{X}-1.
   \end{equation}
   Since $l=l_1+l_2$ and $\beta=\beta_1+\beta_2$, only one of inequalities \eqref{eqn:ineq_1} and \eqref{eqn:ineq_2} can hold, which gives a contradiction. 
\end{proof}

There are also boundary evaluation maps from the fiber product \eqref{eqn:fibre_prod} to $L^{k+1}$
\begin{equation}
ev=(ev_0, ev_1,\cdots, ev_k) \colon \mathcal{M}_{k+1,l}(L, \beta, J; \textbf{q}) \rightarrow L^{k+1}. \nonumber
\end{equation}
It is proved  in \cite[Corollary 3.1, Theorem 5.1]{F1} that there is a consistent choice of Kuranishi structures on $\mathcal{M}_{k+1,l}(L, \beta, J; \textbf{q})$ for each $k ,l \geq 0$ such that the evaluation map $ev_0$ is weakly submersive. Let $h_1, \cdots, h_k$ in $\Omega(L)$, one can then define an operator as follows
\begin{eqnarray}
&& \fq_{l , k, \beta} \colon B_l(H[2])\otimes B_k(\Omega(L)[1]) \rightarrow \Omega(L)[1], \nonumber \\
&& \fq_{l , k, \beta}(f_{\textbf{q}}; h_1, \cdots, h_k)=\frac{1}{l !}(ev_0)_!(ev_1, \cdots, ev_k)^*(h_1\wedge \cdots \wedge h_k). \nonumber 
\end{eqnarray}
Suppose that the degrees of the differential forms $h_i$ satisfy 
\begin{equation}
\sum_{j}((\deg(h_j)-1)-\mu(\beta)+\sum_i (2n- \dim(f_{\textbf{q}(i)})) +2=d \text{ for } d \in \mathbb{N,}
\end{equation}
the operation $ \fq_{l , k, \beta}$ defines a degree $d$ differential form in $\Omega(L).$ 
There is a symmetric group action of $\mathfrak{S}_l$ of order $l!$ on $B_l(H[2])$ given by
\begin{equation}
\sigma \cdot (x_1 \otimes \cdots \otimes x_l)=(-1)^{\dagger} x_{\sigma(1)} \otimes \cdots \otimes x_{\sigma(l)}, \nonumber
\end{equation}
where $\dagger=\sum_{i<j; \sigma(i) > \sigma(j)} \deg x_i \deg x_j$. We denote by $E_l(H[2])$ the subspace of $\mathfrak{S}_l$-invariant elements in $B_l(H[2])$ and denote the restriction of the operator $\fq_{l, k, \beta}$ to the symmetric tensors of the form
$$Sf_{\textbf{q}}=\frac{1}{l !} \sum_{\sigma \in \mathfrak{S}_l} f_{\textbf{q}(\sigma(1))} \otimes \cdots \otimes f_{\textbf{q}(\sigma(l))}$$ 
by the same notation  $\fq_{l , k, \beta} \colon E_l(H[2])\otimes B_k(\Omega(L)[1]) \rightarrow \Omega(L)[1]$. Then one obtains a family of operators
\begin{equation}
\fq_{l,k}:=\sum_{\beta}T^{\omega( \beta)} q_{l, k, \beta} \colon E_l(H[2])\otimes B_k(\Omega(L) \widehat{\otimes} \Lambda_0[1]) \rightarrow \Omega(L) \widehat{\otimes} \Lambda_0[1].\nonumber
\end{equation}
We remark that this operation $\fq_{l,k}$ is well-defined of degree one as the Maslov index $\mu(\beta)$ is even for any Lagrangian torus fiber $L$ and our vector spaces are taken to be $\Z/2$-graded.
For each $\fb \in H,$ we can define an operation 
\begin{eqnarray}
&& \fm_k^{\fb} \colon B_k(\Omega(L)[1]) \rightarrow \Omega(L)[1],\nonumber \\
&& \fm_k^{\fb}(h_1,\cdots, h_k)=\sum_{\beta}\fm_{\beta,k}^{\fb}(h_1, \cdots ,h_k)=\sum_{\beta}\sum_{l \geq 0}\fq_{l,k, \beta}(\fb, \cdots, \fb, h_1,\cdots, h_k). \nonumber 
\end{eqnarray}
 It is proved in \cite[Lemma 2.2]{FOOO_bulk} that this defines a $G$-gapped filtered $A_{\infty}$-algebra structure on $\Omega(L) \widehat{\otimes} \Lambda_0.$  
 \subsection{The canonical model and its Maurer-Cartan elements}\label{subsec:canmcelt}
\indent To define the canonical model on $H^*(L; \Lambda_0),$ we choose a Riemannian metric on $L$ and represent $H(L;\R)$ as the subspace of harmonic forms in $\Omega(L).$ By applying a version of the homological perturbation lemma for filtered $A_{\infty}$-algebra, Fukaya constructed a $G$-gapped filtered $A_{\infty}$-structure in \cite[Section 10]{F1} given by
\begin{equation}
\fm_{k}^{\fb,can} \colon B_k(H^*(L;\Lambda_0)[1]) \rightarrow H^*(L;\Lambda_0)[1], \nonumber
\end{equation}
which is quasi-isomorphic to the de Rham model $(\Omega(L)\widehat{\otimes}\Lambda_0, \{\fm_{k}^{\fb}\}_{k=1}^{\infty}).$  Given the canonical model of $L$, one can define the space of weak Maurer-Cartan elements of positive valuation as
\begin{equation}\label{eqn:weak_MC}
\widehat{\mathcal{MC}}^{\fb}_+(L):=\{ b \in H^{odd}(L;\Lambda_+) \mathbin{|} \sum_{k=0}^{\infty} \fm_k^{\fb, can}(b, \cdots, b)=\lambda  PD[L], \ \lambda \in \Lambda_+ \}.
\end{equation}
Here the infinite sum $ \sum_{k=0}^{\infty} \fm_k^{\fb,can}(b, \cdots, b)$ converges in the non-Archimedean topology defined by the norm $| \cdot |=e^{-val(\cdot)}$ for elements $b \in H^{odd}(L;\Lambda_+)$ with positive valuations.
We say that two such weak Maurer-Cartan elements $b_0$ and $b_1$ are gauge equivalent if there exist $b(t)$ of degree $1$ and $c(t)$ of degree $0$ for $t \in [0,1]$ satisfying the following conditions
\begin{eqnarray} \label{eqn:gauge}
&& b(0)=b_0 \text{ and } b(1)=b_1, \nonumber  \\
&& \frac{d}{dt} b(t)+\sum_k \fm_k^{\fb, can}(b(t),\cdots, b(t), c(t), b(t),\cdots, b(t))=0,\  \forall t \in (0,1). 
\end{eqnarray}
The moduli space of weak Maurer-Cartan elements of positive valuations in $\Omega(L) \widehat{\otimes} \Lambda_0$ is defined as
\begin{equation}
\mathcal{MC}^{\fb}_+(L):=\widehat{\mathcal{MC}}^{\fb}_+(L)/\sim, \nonumber
\end{equation}
where $\sim$ denotes the gauge equivalence relation. In the specific case that we consider, we prove the following Lemma.
\begin{lemma}\label{lem:weak_MC}
Let $L$ be a Lagrangian torus fiber of the moment map $\pi \colon X \rightarrow P$ for a toric Fano surface and $\fb=\textbf{t}_0\textbf{f}_0:=\textbf{t}_0PD[pt] \in H$. The weak Maurer-Cartan space associated to the $G$-gapped filtered $A_{\infty}$-algebra $(H^*(L; \Lambda_0), \{ \fm_{k}^{\fb, can}\}_{k=1}^{\infty})$ is given by
\begin{equation}
 \mathcal{MC}^{\fb}_+(L) = H^1(L; \Lambda_+). \nonumber
\end{equation}
\end{lemma}

\begin{proof}
The lemma follows from a simple dimension argument. One needs to show that 
\begin{equation}
\sum_{k,\beta} \fm_{k, \beta}^{\fb, can}(b,\cdots, b), \nonumber
\end{equation}
is a multiple of the unit class in $H^0(L; \Lambda)$ for any $b \in H^1 (L;\Lambda_+)$. Because of the compatibility of the forgetful maps, the term $ \fm_{k, \beta}^{\fb, can}(b,\cdots, b)$ can be non-zero only if the moduli space $\mathcal{M}_{0,l}(L,\beta,J; \mathbf{p})$ has the virtual dimension
$$\dim (L)+2l-3+\mu(\beta)-l\deg(\fb)= \mu(\beta) -2l -1 \geq 0,$$
where $l$ is the number of interior marked points and we have used the fact that $\dim(L)=2$ and $\deg(\fb) = 4$ here. Therefore, we have $\mu(\beta)\geq 2l+1$.
This implies that $\mu (\beta) -2l \geq  2$ since $\mu(\beta)$ is even.  On the other hand, the degree of the output of $\fm_{k, \beta}^{\fb, can}(b,\cdots, b)$ is given by
$$2-(\mu(\beta) - 2l) \leq 0,$$ 
as $\mu(\beta) - 2l \geq 2$. Therefore the only possible degree of the output is zero. Since we are working on the canonical model, such a degree zero output must be a multiple of the unit class.\\
\indent To see that any two elements $b_0, b_1 \in  \widehat{\mathcal{MC}}^{\fb}_+(L)$ are not gauge equivalent, we assume by contradiction that there are $b(t)$ of degree $1$  and $c(t)$ of degree $0$ satisfying equation \eqref{eqn:gauge}. In the canonical model $(H^*(L;\Lambda_0), \{ \fm_{k}^{\fb, can}\}_k)$ the only element of degree zero is $\textbf{e}=PD[L],$ so one has $c(t)=f(t)\textbf{e}$ for some smooth function $f(t).$ As the canonical model is shown to be unital \cite[Lemma 6.14]{FOOO_bulk}, we have that 
\begin{eqnarray}
&& \fm_k^{\fb, can}(b(t),\cdots, b(t), \textbf{e}, b(t),\cdots, b(t))=0 \text{ for } k \neq 2,\nonumber \\
&& \fm_2^{\fb, can}(\textbf{e}, b(t))=-\fm_2^{\fb,can}(b(t),\textbf{e})= b(t). \nonumber
\end{eqnarray}
Then equation \eqref{eqn:gauge} becomes $\frac{d}{dt}b(t)=0$ and hence $b(t)=b_0=b_1$ is the constant class, which completes the proof.
\end{proof}
%

\subsubsection{The $A_{\infty}$-homomorphism between Fukaya $A_{\infty}$-algebras from the pseudo-isotopy}
For two different Lagrangian torus fibers of the moment map $\pi \colon X \rightarrow P,$ denoted as $L_u: = \pi^{-1}(u)$ and $L_{u'}:=\pi^{-1}(u')$, there is a notion of pseudo-isotopies between $G$-gapped filtered $A_{\infty}$-algebras introduced in \cite[Definition 8.5]{F2}.

\begin{defn}
A pseudo-isotopy of $G$-gapped filtered $A_{\infty}$-algebras consists of the data $(A, \{\fm_{k, \beta}^t\}_{k=1}^{\infty}, \{ \fc_{k, \beta}^t\}_{k=1}^{\infty} ),$ where
\begin{enumerate}
\item[(1)] The operations $\{\fm_{k, \beta}^t\}_{k=1}^{\infty}$ and $\{ \fc_{\beta;k}^t \}_{k=1}^{\infty}$ are smooth in $t$;
\item[(2)] For each fixed $t,$ the operations $\{\fm_{k, \beta}^t\}_{k=1}^{\infty}$ defines a $G$-gapped filtered $A_{\infty}$-structure on $A$;
\item[(3)] The following equation is satisfied for $x_i \in BA[1],$
\begin{eqnarray}
&& \ \ \ \ \frac{d}{dt} \fm_{k, \beta}^t(x_1,\cdots, x_k) \nonumber \\
&& =\sum_{k_1+k_2=k} \sum_{\beta_1+\beta_2=\beta} \sum_{i=1}^{k-k_2+1} \fm^t_{\beta_1;k_1}(x_1,\cdots, \fc_{\beta_2; k_2}^t(x_i,\cdots),\cdots, x_k)\nonumber \\
&& -\sum_{k_1+k_2=k} \sum_{\beta_1+\beta_2=\beta} \sum_{i=1}^{k-k_2+1}(-1)^{\dagger} \fc_{\beta_1;k_1}^t(x_1,\cdots, \fm_{\beta_2; k_2}^t(x_i,\cdots),\cdots, x_k), \nonumber 
\end{eqnarray}
where $\dagger=\deg(x_1)+1+\cdots +\deg(x_{i-1})+1.$
\item[(4)] When $\beta=0,$ we have that $\fm_{\beta;k}^t$ is independent of $t$ and $\fc_{\beta;k}^t=0$ for all $k.$
\end{enumerate}
\end{defn}

In our geometric situation, for any two interior points $u$ and $u'$ in $\Int(P)$ of the moment polytope, one chooses a smooth path $\phi \colon [0,1] \rightarrow \Int(P)$ such that $\phi(0)=u$ and $\phi(1)=u'$ and a family of diffeomorphisms 
\begin{equation}
\phi_t \colon X \rightarrow X, \text{ for each }t \in  [0,1].  \nonumber
\end{equation}
such that $\phi_1(L_u)=L_{u'}.$
One can construct a pseudo-isotopy between the two $A_{\infty}$-structures $\fm_{k}^{\fb}$ and $\fm_{k}^{\fb '}$ on $\Omega(L_u)\widehat{\otimes }\Lambda_0$ associated to the different choices of almost complex structures $J$ and $(\phi^{-1}_1)_*(J)$ that are used in the definition \eqref{eqn:fibre_prod}. To define the pseudo-isotopy, one considers the the parametrized moduli space
\begin{equation}
 \mathcal{M}_{k+1,l}(L_u, \beta, \mathcal{J};\textbf{q})= \bigcup_{t\in [0,1]} \{ t\} \times \mathcal{M}_{k+1,l}(L_u, \beta, J_t;\textbf{q}), 
\end{equation} 
where $J_t$ is a family of $\omega$-tamed almost complex structures defined by $J_t=(\phi_t^{-1})_* J$ and $\mathcal{M}_{k+1,l}(L_u, \beta, J_t;\textbf{q})$ is the fiber product 
\begin{equation}\label{eqn:J_t}
\mathcal{M}_{k+1,l}(L_u, \beta, J_t;\textbf{q}):=\mathcal{M}_{k+1,l}(L_u, \beta, J_t)_{ev^{int}} \times_{X^l} \prod_{i=1}^l f_{p(i)}. \nonumber
\end{equation}
There are boundary evaluation maps
\begin{eqnarray}
&& ev=(ev_1, \cdots, ev_k) \colon \mathcal{M}_{k+1,l}(L_u, \beta, \mathcal{J}) \rightarrow L_u^{k}, \nonumber \\
&& ev_{0,t} \colon  \mathcal{M}_{k+1,l}(L_u, \beta, \mathcal{J};\textbf{q}) \rightarrow L_u\times [0,1], \{ t\} \times \mathcal{M}_{k+1,l}(L_u, \beta, J_t;\textbf{q}) \mapsto (\psi(z_0), t). \nonumber
\end{eqnarray}
It is shown in \cite[Lemma 11.3]{F1} that the Kuranishi structures on $\mathcal{M}_{k+1,l}(L_u, \beta, \mathcal{J})$ can be chosen such that the evaluation map $ev_{0,t}$ is weakly submersive. We define the operations $\fq_{l,k, \beta}^{1}$ and $\fq_{l,k, \beta}^{0}$ by the formula
\begin{eqnarray} \label{eqn:pseudo-iso}
&& \ \ \ (ev_{0,t})_!(ev_1,\cdots, ev_k)^*(h_1 \wedge \cdots \wedge h_k) \nonumber \\
&&=\fq_{l,k, \beta}^{1}(f_{\textbf{q}(i)}, h_1,\cdots, h_k)+\fq_{l,k, \beta}^{0}(f_{\textbf{q}(i)}, h_1,\cdots, h_k) dt 
\end{eqnarray} 
for differential forms $h_i$ in $\Omega(L_u).$ The operations $ \{\fm_{k, \beta}^{\fb,t}\}_{k=1}^{\infty}$ and $ \{ \fc_{k, \beta}^{\fb,t}\}_{k=1}^{\infty}$ defining the pseudo-isotopy are given by
\begin{eqnarray}
&& \fm_{k, \beta}^{\fb,t}(h_1, \cdots, h_k)=\sum_{\beta}\sum_{l \geq 0}\fq_{l,k, \beta}^1(\fb, \cdots, \fb, h_1,\cdots, h_k)T^{\omega(\beta)}. \nonumber \\
&& \fc_{k, \beta}^{\fb,t} (h_1, \cdots, h_k)=\sum_{\beta}\sum_{l \geq 0}\fq_{l,k, \beta}^0(\fb, \cdots, \fb, h_1,\cdots, h_k)T^{\omega(\beta)}, \nonumber 
\end{eqnarray}
which are of degree $1-\mu(\beta) \text{ mod } 2=1$ and $-\mu(\beta) \text{ mod } 2=0$ respectively.

By construction, the triple $(\Omega^*(L_{u})\hat{\otimes}\Lambda_0, \{\fm_{k, \beta}^{\fb,t}\} , \{\fc_{k, \beta}^{\fb,t}\} )$ gives rise to a pseudo-isotopy. This gives rise to a pseudo-isotopy on the corresponding canonical model $(H^*(L_{u};\Lambda_0), \{\fm_{k, \beta}^{\fb, t, can}\} , \{\fc_{k, \beta}^{\fb, t, can}\} )$ by Theorem 8.4 in \cite{F1} and \cite[Section 2.10]{T4}. From such a pseudo-isotopy, one can construct an $A_{\infty}$-homomorphism 
\begin{eqnarray} \label{eqn:compatiblity}
&& \hat{f}^{\fb,can}:=\sum_{k \geq 0} \sum_{\beta \in H_2(X, L_u)} f^{\fb,can}_{k, \beta}  \colon B(H^*(L_{u};\Lambda_0)[1]) \rightarrow H^*(L_{u};\Lambda_0)[1], \nonumber\\
&&\hat{f}^{\fb, can} \circ \hat{\fm}^{\fb, can} =\hat{\fm}^{\fb ', can}\circ \hat{f}^{\fb, can}.
\end{eqnarray}
The explicit definitions of the operations
\begin{equation}
f_{k, \beta}^{\fb, can} \colon B_k(H^*(L_{u};\Lambda_0)[1]) \rightarrow H^*(L_{u};\Lambda_0)[1] \nonumber
\end{equation}
of degree $-\mu(\beta)=0 \text{ mod } 2$ are given in \cite[Section 11]{F1} and \cite[Section 2.7]{T4}.\\
\indent The $A_{\infty}$-structure $\{ \fm^{\fb, can}_k\}_{k=1}^{\infty}$ and the $A_{\infty}$-homomorphism $\{f^{\fb, can}_k \}_{k=1}^{\infty}$ that we defined on the canonical model of Fukaya algebra associated to any Lagrangian torus fiber satisfy the following property.

\begin{lemma}[Divisor axioms]\label{lem:divisor}
For each fixed $k \geq 0$, $b \in H^1(L; \Lambda_+),$ $x_1, \cdots, x_k \in H^*(L; \Lambda_0)$ and $\fb \in H$, the following equations hold
\begin{equation}\label{eqn:divisor1}
\resizebox{.9\hsize}{!}{$
\ \  \sum_{n_0+\cdots+n_k=n}\fm^{\fb, can}_{k+n,\beta}(b^{\otimes n_0}, x_1, b^{\otimes n_2}, \cdots, b^{\otimes n_{k-1}}, x_k, b^{\otimes n_k})=\frac{1}{n!} \big(\langle \partial \beta, b \rangle\big)^n \fm_k^{\fb, can}(x_1, \cdots, x_k)$}; 
\end{equation}
\begin{equation}\label{eqn:divisor2}
\resizebox{.9\hsize}{!}{$
\ \ \sum_{n_0+\cdots+n_k=n}f^{\fb,can}_{k+n,\beta}(b^{\otimes n_0}, x_1, b^{\otimes n_2}, \cdots, b^{\otimes n_{k-1}}, x_k, b^{\otimes n_k})=\frac{1}{n!}\big( \langle \partial \beta, b \rangle \big)^n f_k^{\fb, can}(x_1, \cdots, x_k).$} 
\end{equation}
\end{lemma}

\begin{proof}
The case when $\fb=0$ is proved by Fukaya in \cite[Lemma 13.2]{F1} and \cite[Lemma 4.4]{T4}. The proof of the Lemma for general $\fb$ is similar. Let $\vec{n}=(n_0, n_1, \cdots, n_k)$ be a tuple of non-negative integers such that $\sum_i n_i=n.$ There is a forgetful map
\begin{equation}\label{eqn:forgetfulmap}
\mathfrak{forget}_{\vec{n},1} \colon  \mathcal{M}_{k+n,l}(L, \beta, J) \rightarrow  \mathcal{M}_{k,l}(L, \beta, J) \text{ for each }l \geq 1, 
\end{equation}
which forgets the marked points labeled by $1,\cdots, n_1,$ $n_1+2, \cdots, n_1+n_2+1, \cdots$ on the boundary.
It is shown in Corollary 5.1 in \cite{F1} that there is a Kuranishi structure on the moduli spaces $ \mathcal{M}_{k+1,l}(L, \beta, J)$ that are compatible with forgetful map $\mathfrak{forget}_{\vec{n},1}$ and permutation group action on the interior marked points. One chooses a system of multisections $\mathfrak{s}, \mathfrak{s}'$ on $ \mathcal{M}_{k+n,l}(L, \beta, J)$ and $\mathcal{M}_{k,l}(L, \beta, J)$ compatible with \eqref{eqn:forgetfulmap}, and denote their zero-sets of $\mathfrak{s}, \mathfrak{s}'$ by $ \mathcal{M}_{k+n,l}(L, \beta, J)^{\mathfrak{s}}$ and $\mathcal{M}_{k,l}(L, \beta, J)^{\mathfrak{s}'}$. One can assume that the interior evaluation maps
$ev^{int}$ is weakly submersive by increasing the dimensions of the obstruction bundles in each Kuranishi neighborhood of $p \in \mathcal{M}_{k,l}(L, \beta, J)$ if necessary.  So there is an induced forgetful maps on the fiber products
\begin{equation}
\mathfrak{forget}_{\vec{n},1}^{\mathfrak{s}} \colon  \mathcal{M}_{k+n,l}(L, \beta, J; \textbf{q})^{\mathfrak{s}} \rightarrow  \mathcal{M}_{k,l}(L, \beta, J; \textbf{q})^{\mathfrak{s}'}. \nonumber
\end{equation}
The preimage of any point $q$ in $\mathcal{M}_{1,l}(L, \beta, J; \textbf{q})^{\mathfrak{s}'}$ is isomorphic the standard $n$-simplex $\Delta^n.$ This implies that
\begin{equation}\label{eqn:forget}
\int_{(\mathfrak{forget^{\mathfrak{s}}_{\vec{n}}})^{-1}(p)} ev^*_{\vec{n}}(b \times \cdots \times b)=\frac{1}{n!}\big(\langle \partial \beta \cap b \rangle \big)^n \text{ for } b \in \Omega^1(L),
\end{equation}
where $ev_{\vec{n}} \colon  \mathcal{M}_{k+n,l}(L, \beta, J) \rightarrow L^n$ is the evaluation map at the boundary marked points $\vec{n}$ and $\frac{1}{n!}$ is the volume of the fiber $(\mathfrak{forget^{\mathfrak{s}}_{\vec{n}}})^{-1}(p)=\Delta^n$. We remark that since $\fb$ was assumed to have even degree and $b \in \Omega^*(L)\hat{\otimes} \Lambda_0 [1]$ is of degree two, there is no extra sign in equation \eqref{eqn:forget} for all $\vec{n}$. This complets the proof of \eqref{eqn:divisor1} for the de Rham model. By Lemma 13.2 in \cite{F1}, equation \eqref{eqn:divisor1} holds in the canonical model as well.\\
\indent For equation \eqref{eqn:divisor2}, one first applies the above arguments to the moduli spaces $ \mathcal{M}_{k+1,l}(L, \beta, \mathcal{J}; \textbf{q})$ defining the pseudo-isotopy and concludes that the operators $\{ \fm^{\fb,t}_{k, \beta}\}$ and $\{ \fc^{\fb,t}_{k,\beta} \}$ satisfy the divisor axiom \eqref{eqn:divisor1} as well. Using formulae (2.6) and (2.7) in \cite{T4}, we see that the induced operation $\{ \fm^{\fb,t, can}_{k, \beta}\}$ and $\{ \fc^{\fb,t, can}_{k,\beta} \}$ satisfy \eqref{eqn:divisor1} in the canonical model $H^*(L_u;\Lambda_0).$ Equation \eqref{eqn:divisor2} then follows from a similar argument as in Lemma 4.4 in \cite{T4}.

\end{proof}

Given the $A_{\infty}$-homomorphism $\hat{f}^{\fb}$, there is an induced map of degree zero on the corresponding weak Maurer-Cartan spaces of positive valuations
\begin{equation}
(F^{\fb, can})_* \colon \mathcal{MC}_+(L_u; \{\fm_k^{\fb} \}_k) \cong H^1(L_u;\Lambda_+) \rightarrow \mathcal{MC}_+(L_{u}; \{\fm_k^{\fb '} \}_k) \cong H^1(L_{u};\Lambda_+), \nonumber \\
\end{equation}
\begin{equation}\label{eqn:map_MC}
(F^{\fb, can})_*(b):=\hat{f}^{\fb, can}(e^b)=f_0^{\fb, can}(1)+f_1^{\fb, can}(b)+f_2^{\fb,can}(b\otimes b)+\cdots, 
\end{equation}

where $e^b:=1 +b +b \otimes b+ b \otimes b \otimes b +\cdots.$ 
We prove next that this map $(F^{\fb, can})_*$ depends only on the homotopy class of the path $\phi$ relative to the end points.  In the subsequent discussions, we will refer the following important consequence of Fukaya's pseudo-isotopies associate to $\phi$ as the \textit{Fukaya's trick}.

\begin{prop} [Fukaya's trick]\label{1086}
Given two Lagrangian torus fibers $L_u$ and $L_{u'}$ of the moment map $\pi \colon X \rightarrow P$ if a toric Fano surface $X$ with $u, u' \in \Int(P)\backslash\mathcal{W}^{holo}$.  If $\phi$ is homotopic relative to end points to $\phi'$ in $\Int{P} \backslash \{p_1,\cdots, p_k\}$, then
\begin{equation} \label{eqn:induced_on_MC}
(F_{\phi}^{\fb, can})_*=(F_{\phi'}^{\fb, can})_* \colon  H^1(L_u;\Lambda_+) \rightarrow H^1(L_{u};\Lambda_+) .
\end{equation}
Equivalently, if $\phi$ is a contractible loop in a small open neighborhood of $u \in \Int(P)\backslash \{p_1,\cdots, p_k\}$ such that $\phi(0)=\phi(1)=u$, then 
\begin{equation} 
(F_{\phi}^{\fb, can})_*=id \colon  H^1(L_u;\Lambda_+)\rightarrow H^1(L_{u};\Lambda_+). \nonumber
\end{equation}
\end{prop}

\begin{proof} Since the above two statements are equivalent by obvious reasons, it suffices to prove the first statement \eqref{eqn:induced_on_MC}. Given two pseudo-isotopies on the $G$-gapped filtered $A_{\infty}$-structures on the cohomology $H^*(L_u; \Lambda)$, there is a notion of pseudo-isotopy between pseudo-isotopies defined in \cite[Definition 14.1]{F1}. One important property of such pseudo-isotopies between pseudo-isotopies, shown in \cite[Section 14]{F1}, is that the $A_{\infty}$-functors that they define induce the same map on the weak Maurer-Cartan moduli spaces
\begin{equation}
(F_{\phi}^{\fb, can})_*=(F_{\phi'}^{\fb, can})_* \colon  \mathcal{MC}_+^{\fb}(L_u;  \{ \fm_k^{\fb }\}) \rightarrow   \mathcal{MC}_+^{\fb}(L_{u}; \{ \fm_k^{\fb '}\}).
\end{equation}
By Lemma \ref{lem:weak_MC},  it suffices to construct a pseudo-isotopy between the pseudo-isotopies defined by two different paths $\phi$ and $\phi'.$ First, we choose a homotopy $\phi_s$ relative to the end points between the two paths $\phi$ and $\phi'$, and choose  a smooth map $\mathcal{J}$ from $[0,1]^2$ to the space of $\omega$-compatible almost complex structures $ \mathcal{J}(M, \omega)$ such that $\mathcal{J}([t,0])=J_{t}$, $\mathcal{J}([t,1])=J_t',$ where $J_t$ and $J_t'$ are the almost complex structures that define the pseudo-isotopy in \eqref{eqn:J_t} respectively. It is proved in \cite[Section 14]{F1} that the moduli space
\begin{equation}\label{978}
 \mathcal{M}_{k+1,l}(L_u, \beta, \mathcal{J};\textbf{q})= \bigcup_{(t,s) \in [0,1]^2} \{ (t,s)\} \times \mathcal{M}_{k+1,l}(L_u, \beta, J_{t,s};\textbf{q}), 
\end{equation} 
defines a pseudo-isotopy of pseudo-isotopies by a construction which    is analogous to equation \eqref{eqn:pseudo-iso}. From Lemma \ref{798}, there are no holomorphic discs of negative generalized Maslov discs for generic point constraints and there exists no extra boundary terms in the right hand side of (\ref{978}). 

\end{proof}

Now we define the $k$th order bulk-deformed potential in our setup as follows.
\begin{defn}  \label{defn:bulk_potential}
Let $\fb=t_1q_1+\cdot+t_k q_k$ and $R_k=\C[t_i]/(t_i^2=0\mathbb{|}i=1,\cdots, k).$
\begin{enumerate}
 \item[(1)] The $k$th order bulk-deformed potential associated to $L_u$ is given by
\begin{equation}
W^{q_1,\cdots, q_k}_k (u)=\sum_{l \leq k} \sum_{|I|=l}\sum_{\beta}\mathfrak{m}_{0, \beta}^{\fb, can} (1) z^{\partial \beta}t_I \ \in \Lambda[ z_1{^\pm}, z_2^{\pm} ]\otimes R_k, \nonumber
\end{equation}
where  $t_I:=t_{i_1}\cdots  t_{i_l}$ for some $I=\{i_1, \cdots, i_l \} \subset \{ 1, 2, \cdots, k\}$. 
\item[(2)] We define the generalized Maslov index $\mu' (\beta)$ of a disc in the class $\beta$ with $l$ interior marked points by the formula
$$ \mu' (\beta) = \mu(\beta) - 2l.$$
Then the $k$th order bulk-deformed potential can be viewed as the count of generalized Maslov $2$ discs with $l$ point constraints for all $l \leq k$.
%
%
\end{enumerate}

\end{defn}

\begin{remark} $\quad$
\begin{enumerate}
\item[(1)]
By Gromov-Ulenbeck compactness theorem, the bulk-deformed superpotential over $R_k$ is a finite sum when $X$ is toric Fano. Therefore,  the evaluation $W_k^{q_1,\cdots, q_k}(u)|_{T=1}$ is well-defined. We will mostly omit the Novikov variable, or set $T=1,$ for notational simplicity in Section \ref{sec:trop_sec}.
%
\item[(2)]
Let $p: \Omega^*(L_{u}) \to H^* (L_{u})$ be the harmonic projection used to take canonical model. Then we have
\begin{equation}\label{eqn:harmonic1}
p (\mathfrak{m}^{\fb}_0 (1)) = \mathfrak{m}_0^{\fb, can} (1)
\end{equation}
from the construction of the canonical model.
\end{enumerate}
\end{remark}

From the $A_{\infty}$-homomorphism relation \eqref{eqn:compatiblity}, one can deduce that for each $k$, the $k$th order bulk-deformed potential satisfies the following equation,
   \begin{equation}\label{eqn:comptwf}
 F_* \circ W_k = W_k \circ F_* \colon H^1(L_u;\Lambda_+) \otimes_{\C}R_k \rightarrow H^1(L_{u};\Lambda_+) \otimes_{\C}R_k, 
   \end{equation}
   where $F_*:=(F^{\fb, can})_*$ is the $k$th order induced map on the weak Maurer-Cartan space defined in equation \eqref{eqn:map_MC} for $\fb=t_1q_1+\cdots+t_k q_k$.
   By setting $T=1,$ the corresponding statement over $R_k$ also holds. This implies that the bulk-deformed potential is preserved under wall-crossing or cluster transformations as desired. Our next goal is to compute $W_k(u)$ inductively via the wall-crossing techniques. To illustrate the inductive algorithm and the necessity of our tropical-holomorphic correspondence theorem \ref{thm:correspondence}, we will compute the low order contributions to $W_n(u)$ and the holomorphic wall structure $\mathcal{W}^{holo}$ explicitly in Section \ref{sec:wall_crossing_Wk}.

\section{Tropical curves in Toric Fano Surfaces} \label{sec:trop_sec}
 In this section, we recall the definitions of tropical curves as well as tropical counting invariants in toric Fano surfaces. Let $X$ be a toric Fano surface associated with a moment polytope $\Sigma$. We first recall the notation of tropical curves as follows.
  \begin{defn} 
   A tropical rational curve in $\mathbb{R}^2$ is a triple $(h,T,w)$ satisfying the following properties. 
       \begin{enumerate}
          \item[(i)] $T$ is a tree that possibly contains unbounded edges. The set of vertices is denoted by $T^{[0]}$, and the set of edges are denoted by $T^{[1]}$. 
          \item[(ii)] Every vertex is trivalent. 
          \item[(iii)] $h:T\rightarrow M_{\mathbb{R}}$ is a map such that $h(e)$ is an embedding of affine line segment or ray if $e \in T^{[1]}$ is bounded or unbounded respectively. 
          \item[(iv)] There is a map $w:T^{[1]}\rightarrow \mathbb{N}$ assigning weights to edges such that the balancing condition holds: for every vertex with adjacent edges $e_1,e_2,e_3,$ one has that 
              \begin{align*}
                w(e_1)v(e_1)+w(e_2)v(e_2)+w(e_3)v(e_3)=0,
              \end{align*}
          where $v(e_i)$ is the primitive vector tangent to $h(e_i)$ which is pointed away from $v$.
       \end{enumerate}
Furthermore, the triple $(h,T,w)$ is called a tropical curve of $X$ if every unbounded edges is parallel to a $1$-cone in $\Sigma$.
\end{defn}


\begin{defn} Given a tropical rational curve $(h, T, w).$ We say that
  \begin{enumerate}
      \item A marked tropical rational curve in $X$ is a tropical rational curve $(h,T,w)$ with a marking 
           \begin{align*}
            \epsilon \colon  I \longrightarrow T^{[1]}_{\infty},          \end{align*} where $I$ is some index set and  $T^{[1]}_{\infty}$ is the set of the unbounded edges. 
      \item The degree $\Delta$ of a marked tropical rational curve is a map 
        \begin{align*}
         \Delta \colon  I\longrightarrow M_{\mathbb{R}}\\
        \ \ \ \ \ \      i \mapsto w(e_i)u_i,
        \end{align*} where $u_i$ is the primitive vector of a $1$-cone of $\Sigma$.
     \item Denote $[\Delta]\in H_2(X,\mathbb{Z})$ to be the unique curve class such that 
     $$[\Delta]\cdot D_v=\sum_{e_i\in T^{[1]}_{\infty}:h(e_i) \parallelsum  v}w(e_i)$$
      for every vector $v$ generating a $1$-cone in $\Sigma$.
     \item We denote by $\mbox{Aut}(\Delta)$ the automorphism group of the marking $\epsilon$.  
      \end{enumerate}
\end{defn}
The degree $\Delta$ can be viewed as a refinement of a curve class. 
Algebraically, one can understand a tropical curve as the corner locus of a Laurent polynomial in the tropical semi-ring. Geometrically, tropical curves are the Gromov-Hausdorff limit of holomorphic curves at certain adiabatic limit in the following sense. Mikhalkin \cite{M2} introduced the $1$-parameter family of diffeomorphisms
  \begin{align*}
     H_t:(\mathbb{C}^*)^2&\longrightarrow (\mathbb{C}^*)^2 \\
        (x,y)&\mapsto   \left(|x|^{\frac{1}{\log{t}}}\frac{x}{|x|},|y|^{\frac{1}{\log{t}}}\frac{y}{|y|} \right),
  \end{align*} 
 which induces a $1$-parameter family of almost complex structures $J_t$ via pulling back the standard complex structure by $H_t$. Fixing the suitable torus invariant K\"ahler form, the volume of the torus fiber approaches zero  as $t\rightarrow \infty$. This is coherent with the SYZ interpretation of the large complex structure limit \cite{GW}\cite{KS4}. 
 
 One says that a Riemann surface $V_t$ is $J_t$-holomorphic if and only if $V_t=H_t(V)$ for some holomorphic curve $V\subseteq (\mathbb{C}^*)^2$ with respect to the standard complex structure. We denote by $Log$ the torus fibration, which is the same as the toric moment map up to a Legendre transformation, given by
     \begin{align*}
         Log:(\mathbb{C}^*)^2&\longrightarrow \mathbb{R}^2 \\ 
              (x,y)&\mapsto (\log{|x|},\log{|y|}).
     \end{align*} 
One of the rudimental results in tropical geometry by Mikhalkin \cite{M2} states that there is a bijection between the set of tropical curves in the adiabatic limit and that of sequences of $J_t$-holomorphic curves for $t \gg 0.$
  \begin{prop}\cite{M2}
     Let $V_{\infty}$ be a Gromov-Hausdorff limit of a sequence of $J_t$-holomorphic curve. Then $Log(V_{\infty})$ is a tropical curve. Conversely, a tropical curve can be realized as the $Log$-image of a Gromov-Hausdorff limit of a sequence of $J_t$-holomorphic curve. 
  \end{prop}

\subsection{Tropical discs and tropical descendant invariants} \label{sec:trop_discs}
We next recall the analogous notions of tropical disc counting following \cite{G7} and \cite{N2}. Such disc counting invariants give rise to the notation of the tropical superpotential defined later in this subsection. 
 
%
\begin{defn}
   A tropical disc of $X$ is a triple $(h,T,w)$ with the following properties. 
    \begin{enumerate}
       \item[(i)] $T$ is a rooted tree with a unique root $x$ which may contain unbounded edges. The set of vertices  and the set of edges are denoted by $T^{[0]}$ and $T^{[1]}$, respectively. 
       \item[(ii)] Every vertex other than $x$ is trivalent. 
       \item[(iii)] $h:T\rightarrow M_{\mathbb{R}}$ restricted to an edge $e$ is an embedding of affine line segment or ray if $e$ is bounded or unbounded respectively. 
       \item[(iv)] $w:T^{[1]}\rightarrow \mathbb{N}$ be the weights on edges such that the balancing condition holds: for every vertex $v\neq x$ with adjacent edges $e_1,e_2,e_3,$ one has that
           \begin{align*}
             w(e_1)v(e_1)+w(e_2)v(e_2)+w(e_3)v(e_3)=0,
           \end{align*}
       where $v(e_i)$ is the primitive vector tangent to $h(e_i)$ and pointed away from the vertex $v$.
       \item[(v)] $u=h(x)\in \mathbb{R}^2$ is called the end of the tropical disc. 
       \item[(vi)] If $e$ is an unbounded edge, then $h(e)$is parallel to a $1$-cone of $\Sigma$. 
    \end{enumerate}
 \end{defn}
 There are different notions of equivalences for tropical discs.
 \begin{defn}   
 Two tropical discs $(h_i,T_i, w_i)$ are called 
   \begin{enumerate}
     \item isomorphic if there exists a diffeomorphism $f:T_1\rightarrow T_2$ such that $h_2\circ f=h_1$ and $w_2(f(e))=w_1(e)$ for each $e\in T_1^{[1]}$. 
     \item of the same type if there exists a diffeomorphism $f:T_1\rightarrow T_2$ such that $h_2(f(e))$ is parallel to $h_1(e)$ and $w_2(f(e))=w_1(e)$ for each $e\in T_1^{[1]}$.
   \end{enumerate}   
\end{defn}
To define the analogous notion of Maslov index $2$ disc in tropical geometry, one considers a tropical disc $(h,T,w)$ ending on $u \in \R^2$ in a toric Fano manifold which has a unique unbounded edge.  As the direction of the unbounded edge determines a toric divisor in $X$, one can define the relative class $[h]\in H_2(X,L_{u})$ of $h$ by that of the corresponding Maslov two holomorphic disc considered in \cite{CO}. 

For an arbitrary tropical disc $(h,T,w)$ ending on $u$, one can define the class of $h$, inductively as follows. Let $e$ be the edge adjacent to the root $x$ such that $h(x)=u$, and $x'$ be the other vertex adjacent to $e$. 
Note that erasing $e$ from $T$ gives several sub-tropical disc with end on $L_{h(x')}$, say $\{(h_i,T_i,w_i)\}_i$.
Each sub-tropical disc $(h_i,T_i,w_i)$ has an associated relative class $H_2(X,L_{h(x')})$ by induction on the number of vertices of a tropical discs. Then we define 
$$[h]\in H_2(X,L_u)$$ 
to be the parallel transport of $\sum_i [h_i]$ along $h(e)$. Maslov indices for tropical discs can be defined similarly as in the holomorphic case below.
\begin{defn}\label{def:tropMI}
   Given a tropical disc $(h,T,w)$, its Maslov index $\mbox{MI}(h)$ is defined to be the twice of the sum of weights of unbounded edges.
\end{defn}

\begin{defn}\label{def:mikmult}
   Given a tropical disc $(h,T,w)$ with end at $u\in \mathbb{R}^2$ and only trivalent vertices except the root. Let $v\in T^{[0]}$ be a trivalent vertex with adjacent edges $e_1,e_2,e_3$. Then the (Mikhalkin) weight at $v$ denoted as $\mbox{Mult}_v$ is given by 
     \begin{align*}
    \mbox{Mult}_v:=    |w(e_1)v(e_1)\wedge w(e_2)v(e_2)|\in \wedge^2 T_{\mathbb{Z}}\mathbb{R}^2\cong \mathbb{Z}.
     \end{align*} 
     The weight $\mbox{Mult}(h)$ of the tropical disc $(h,T,w)$ is defined to be 
     \begin{align*}
       \mbox{Mult}(h)= \prod_{v\in T^{[0]}, val(v)=3}\mbox{Mult}_v.
     \end{align*} 
\end{defn}
 
 Fo $u$ in the log base $\R^2,$ we write $L_u$ for the Lagrangian torus fiber $\pi^{-1}(u)$.  Since our Lagrangian fibration is topologically trivial, we can choose a basis $e_1,e_2$ in $H^1(L_u,\mathbb{Z})$ compatibly with all $u$. Given $[\partial \beta] \in H_1(L_u,\mathbb{Z})$, set 
   \begin{align*}
     z^{[\partial \beta]}=z_1^{\langle[\partial\beta],e_1\rangle}z_2^{\langle[\partial\beta],e_2\rangle},
   \end{align*} where $\langle \,\, , \,\, \rangle$ is the natural pairing between $H_1(L_u, \Z)$ and $H^1(L_u, \Z).$ Now we are ready to define the tropical analogue of the Hori-Vafa potential as follows.

 \begin{defn} For any $u\in \mathbb{R}^2$,
     the tropical superpotential is defined to be 
     \begin{align} \label{1077}
        W(u):=\sum_{h} \mbox{Mult}(h)z^{\partial [h]},
     \end{align} where the summation is over all tropical discs of Maslov index two with end on $u$. 
      \end{defn}
 Cho-Oh \cite{CO} proved that (\ref{1077}) indeed coincide with the superpotential that counts Maslov index two holomorphic discs in Floer theory. 
 Gross \cite{G7} further considered the bulk deformation of the tropical superpotential, which is a tropical analogue of the bulk-deformed potential that we defined in Section \ref{sec:pseudo_iso}.
 \begin{defn}
    Let $p_1,\cdots,p_k\in\mathbb{R}^2$ be $k$ generic points.
    For a tropical disc $(h,T,w)$ with end $u$ generic, we define the following:
     \begin{enumerate}
       \item $p_h:=\left\{i\in \{1,\cdots,k\} \mid p_i\in \{h(v)\mid v\in T^{[0]}, val(v)=2\} \right\}$.
       \item  The generalized Maslov index $\mbox{MI}'(h)$ of a tropical disc $h$ is defined to be 
           \begin{align*}
               \mbox{MI}'(h)=\mbox{MI}(h)-2|p_h|.
           \end{align*}
     \end{enumerate}
 \end{defn}
 To avoid contributions from complicated tropical discs to the wall-crossing, Gross  \cite{G7} restricted the coefficients ring to
   \begin{align*}
     R_k=\mathbb{C}[t_1,\cdots, t_k]/(t_i^2=0).
   \end{align*}
   The use of $R_k$, especially the condition $t_i^2=0$, is crucial in our computations of the bulk-deformed potential, since it only records the \textit{first order information} with respect to each of $t_i$'s in the wall-crossing map defined in Definition \ref{1001} below. We now define the bulk-deformed potential $W_{k}(u)$ tropically.
 
 \begin{defn}\label{1001}
    Let $p_1,\cdots,p_k,u\in\mathbb{R}^2$ be in generic positions. The bulk deformed tropical superpotential $W_k(u)$ is defined by 
      \begin{align*}
         W_k^{trop}(u)=\sum_{h} \mbox{Mult}(h)z^{\partial[h]}t_h\in \mathbb{Q}[z_1,z_2]\otimes R_k,
      \end{align*} where $t_h=\prod_{i\in p_h}t_i$ and the summation is over all rigid generalized Maslov index two tropical discs with end on $u$.      
 \end{defn}

By the work of Nishinou \cite{N2} on tropical disc counting, we can conclude that the above expression is well-defined and it defines a Laurent polynomial in $R_k[z_1^{\pm}, z_2^{\pm}]$ for each $k.$

 \begin{prop} \label{1035}
   (\cite{N2} Proposition 3.5)
    There are only finitely many tropical discs with end on $u$ of generalized Maslov index $2$ with respect to a given configuration $p_1,\cdots,p_k$. Therefore, $W_k(u)$ is a Laurent polynomial. Moreover, all the generalized Maslov index two tropical discs are trivalent. 
 \end{prop}
 

The tropical bulk-deformed potential $W_k^{trop}(u)$ shows discontinuity when $u$ goes across a certain 1-dimensional skeleton in $\R^2$ called the tropical wall defined below.  In this case, two potentials on the other sides of the wall differ by a transformation, which is referred as a slab function in the Gross-Siebert program \cite{GS1}. These functions are tropical analogues of the wall-crossing maps in Floer theory.

\begin{defn} \label{1033}
Let $(h,T,w)$ be a tropical disc ending at $u\in \mathbb{R}^2$ with generalized Maslov index $2$. Let $e$ be the edge adjacent to the root. A tropical wall refers to the extension of $f(e)$ in the direction of $u$ to infinity. To each tropical wall $\mathfrak{d}$, we associate a slab function 
 \begin{align*}
 f_{\mathfrak{d}}=1+\mbox{Mult}(h)z^{\partial[h]}t_h \in \mathbb{Q}[z_1,z_2]\otimes R_k,
 \end{align*} and an symplectomorphism 
 \begin{align*}
   \mathcal{K}_{\mathfrak{d}}:\mathbb{Q}[z_1,z_2]&\rightarrow \mathbb{Q}[z_1,z_2]\\
    z_i&\mapsto z_i f_{\mathfrak{d}}^{\langle e_i,\partial [h]\rangle},
 \end{align*} which preserves the standard holomorphic symplectic $2$-form $\frac{dz_1}{z_1}\wedge \frac{dz_2}{z_2}$.
The tropical wall structure $\mathcal{W}^{trop}$ associated to $p_1,\cdots, p_k$ is the union of all tropical walls.    \end{defn}

As a consequence, the tropical bulked potential $W_k^{trop}(u)$ in Definition \ref{1001} is well-defined only when $u$ is in the complement of the wall structure.

   Given two tropical walls $\mathfrak{d}_1,\mathfrak{d}_2,$ it can be verified that $\mathcal{K}_{\mathfrak{d}_1}\mathcal{K}_{\mathfrak{d}_2}=\mathcal{K}_{\mathfrak{d}_2}\mathcal{K}_{\mathfrak{d}_1}$ if and only if $\mathfrak{d}_1,\mathfrak{d}_2$ are parallel. 
   To study the analogue of the Fukaya's trick that we proved in Proposition \ref{1086}, we will consider a generic tropical loop on the log base $\R^2$ defined as follows.
   
   \begin{defn}
   We say a loop $\phi$ in $\mathbb{R}^2$ is generic if 
 \begin{enumerate}
   \item it avoids all the intersections of the tropical walls, and
   \item  every intersection of $\phi$ with the tropical wall $\mathfrak{d}$ is transversal.
 \end{enumerate}
 \end{defn}
  Such a generic loop $\phi$ induces a symplectomorphism 
   \begin{align*}
      F_{\phi}:=\prod_{\mathfrak{d}:\mathfrak{d}\cap \phi\neq \phi} \mathcal{K}_{\mathfrak{d}}^{\epsilon_{\mathfrak{d}}},
   \end{align*} where the product order is with respect to the path and $\epsilon_{\mathfrak{d}}=\mbox{sgn}\langle \phi'(t_{\mathfrak{d}}) , \partial[h] \rangle$. 
   The following illustrates the wall-crossing phenomenon of the tropical discs. 
 \begin{lemma} \label{1078} \cite{GPS}
   Assume that there are two tropical walls $\mathfrak{d}_1,\mathfrak{d}_2$ intersect at $p$ transversely, and no other walls passing through $p$ (which can be always achieved by requiring the point constraints configuration $p_1,\cdots, p_k$ to be generic). Let $\phi$ be a small loop around $p$, then $F_{\phi}=\mbox{id}$ if and only if there exists exactly one tropical wall $\mathfrak{d}$ emanating from $p$ with $f_{\mathfrak{d}}=f_{\mathfrak{d}_1}f_{\mathfrak{d}_2}$.
 \end{lemma}
%
Since $t_i^2=0$, there are only finitely many tropical walls to be added. 
Therefore, one has the following theorem, which is the tropical analogue of the Fukaya's trick that we proved in Proposition \ref{1086}.
      \begin{theorem}[\cite{G7} Proposition 4.7, Theorem 4.15]
      \label{1002}
      $F_{\phi}=\mbox{id}$ for any generic loop $\phi$ in $\mathbb{R}^2$. Moreover, given $u,u'\in \mathbb{R}^2$ in adjacent chambers bounded by a tropical wall $\mathfrak{d}$, we have 
         \begin{align*}
            W_k^{trop}(u)=\mathcal{K}_{\mathfrak{d}}W_k^{trop}(u').    
         \end{align*}
         \end{theorem}
With help of the above theorem, we can inductively compute the tropical bulk-deformed potential $W_k^{trop}(u)$ for any given configuration of $k$ generic points and Lagrangian torus fiber $L_{u}$ in the next subsection. 

%
%

   \subsection{Algorithmic computations of $\mathcal{W}^{trop}$ and $W_k^{trop}(u)$} \label{1084} 
We provide the algorithm of computing the tropical walls $\mathcal{W}^{trop}$ as well as the $k$th-ordered bulk-deformed potential $W_k^{trop}$ tropically. This inductive algorithm should be well-known to experts, but we include it here for completeness. 

Let $p_1, \cdots ,p_k$ be $k$ generic points in $(\C^*)^2.$ By induction, suppose that for any arbitrary generic $k$ points configuration, we are able to locate the corresponding walls, and to compute the tropical bulk-deformed potential $W_k^{trop}(u)$ in every chambers. The strategy to construct the tropical wall structure for $p_1,\cdots, p_{k+1}$ is to use tropical wall structure associated to the point constraints $$p_1,\cdots,\hat{p}_i,\cdots, p_{k+1}$$
and the tropical bulk-deformed potential in each of the chambers, which are known to us by the induction hypothesis. Here, the notation $\hat{p}_i$ indicates that we omit the $i$th point constraint.

In particular, the expression of $W_k^{trop}(u)$ at the point $u=p_i$ can be computed. Every term $n_{\beta}z^{\partial \beta}$ in $W_k^{trop}(p_i)$ implies that there are $n_{\beta}$ tropical discs counted with weights which end on $p_i$. Therefore, extending the edges of those tropical discs with end on $p_i$ gives the corresponding family of tropical discs of generalized Maslov index zero.  Conversely, any tropical discs of generalized Maslov index zero with $p_i$ on the edge with the end arise in this way. 

As a result, we have shown that the tropical walls emanating from $p_i$ are in one-to-one correspondence with the monomials in $W_k^{trop}(p_k)$.
We apply the procedure to each $p_i$, and in addition, we add one more wall for each intersection of the existing walls by Lemma \ref{1078}. There can only be finitely many walls added in this step, thanks to the relations $t_i^2=0$ in the coefficient ring $R_k$. 

Now given a tropical disc $(h,T,w)$ of generalized Maslov index zero. Let $v$ is the other vertex of the edge $e$ adjacent to the end. Deleting the edge $e$ of $T$ gives rise to two tropical discs $(h_1,T_1,w_1)$ and $(h_2,T_2,w_2)$ with ends on $h(v)$ (since $v$ is trivalent), where $T_1 \cup T_2=T\backslash e$ and $h_i=h|_{T_i}$, $w_i=w|_{T_i}$. 
Since generalized Maslov index tropical discs are rigid up to elongation of the ends,  generalized Maslov indices satisfy $\mbox{MI}'(h_1)+\mbox{MI}'(h_2)=\mbox{MI}'(h)=0$. We see that both $h_1$ and $h_2$ are of generalized Maslov index zero. Then by induction on the number of edges, the above algorithm produces all possible tropical discs of generalized Maslov index zero.

\subsection{Tropical descendant invariants} \label{sec:descendant_trop}
For later use in Section \ref{sec:period_thm}, we recall the tropical descendant curve counting. 
It is almost the same as the tropical curve above except that we allow valency to be higher than $3$ subject to a certain balancing condition.This was first studied by Gross \cite{G7} and later by Mandel-Ruddat \cite{MR16}. 
Here, we review the basic definitions of tropical descendant invariants following the exposition in \cite[Section 2]{MR16}.
 
\begin{defn}
  Let $M \cong \Z^2$ be a rank $2$ lattice and let $M_{\R}:=M \otimes_{\Z} \R$. Given generic points $p_1, \cdots, p_k,u\in M_{\mathbb{R}}$, we denote by
\begin{equation}
\mathcal{M}_{\Delta,n}^{trop}(X, p_1, \cdots, p_k, \psi^{\nu}u) \nonumber
\end{equation}  
the moduli space of tropical rational curves of degree $\Delta$ satisfying the following conditions:
\begin{enumerate}
\item[(1)] $v\in T^{[0]}$ can be non-trivalent only if 
\begin{eqnarray}
h(v)\in \{p_1,\cdots, p_k,u\} \text{ and } \left|h^{-1}(\{p_1,\cdots, p_k\})\cap T^{[0]} \right|=n, \nonumber
\end{eqnarray}
\item[(2)] $v\in T^{[0]}$ and $h(v)=p_i$ then $val(v)=2$,
\item[(3)] $x \in T^{[0]}$ is a distinguished vertex with valency $val(x)=\nu+2$ and $h(x)=u$.

\end{enumerate}
\end{defn}

The virtual dimension of the moduli space $\mathcal{M}_{\Delta,n}^{trop}(X, p_1, \cdots, p_k, \psi^{\nu}u)$ is given by
$$|\Delta|-n-\nu-2.$$
To see this, one first observes that the moduli space of trees $T$ (all possible domains of $h$) with $|\Delta|+n+1$ unbounded edges and one vertex $x$ with valency $\nu+2$ is parametrized by $(\R_{\geq 0})^{|\Delta|+n-\nu-2}\otimes M_{\R}.$ There is an evaluation map at the distinguished vertex $x \in T^{[0]}$ given as
$$ev_x \colon \mathcal{M}_{\Delta,n}^{trop}(X, p_1, \cdots, p_k, \psi^{m}u) \rightarrow M_{\R}, \ \ ev_x(h):=h(x).$$
Requiring $h(x)=u$ further reduces the dimension by $2n$, and we have $|\Delta|-n-\nu-2$.

Next we prove that every such tropical descendant curve 
$$h \in \mathcal{M}_{\Delta,n}^{trop}(X, p_1, \cdots, p_k, \psi^{m}u)$$
can be splitted into finitely many generalized Maslov two tropical discs in the complement of $h(T)-\{u\}.$ This is an analogue of the Lemma 3.3 in \cite{G7} which will be one of key ingredients in the proof of Theorem \ref{thm:period_thm} in Section \ref{sec:period_thm}.
\begin{lemma} \label{lem:split_MI2}
Let $\Sigma\subseteq M_{\mathbb{R}}$ be a fan and $p_1, \cdots , p_k, u \in M_{\R}$ be in generic position. Let $T_i$ denote the closures of the connected components of $T-\{x\}$ for $i=1,2 \cdots, \nu+2$. We restrict $h$ to each $T_i$ to obtain $h_i \colon T_i \rightarrow M_{\R}$. Then when $n=|\Delta|-\nu-2,$ one has that $MI'(h_i)=2$ for all $i$. Equivalently, there are exactly $\nu+2$ generalized Maslov two tropical discs in the complement of $h(T)-\{u\}.$
\end{lemma}
\begin{proof}
The proof is an index calculation. Since the valency of non-contracted edges at $x$ is $\nu+2$ for curves $h \in \mathcal{M}_{\Delta,n}^{trop}(X, p_1, \cdots, p_k, \psi^{\nu}S).$ For generic point constraints $p_1, \cdots ,p_k$ and $S=\{u\}$, suppose by contradiction that the generalized Maslov index of $h_i$ satisfies $MI'(h_i) > 2$ for some $i,$ then as a tropical disc with boundary on $L_u:=\pi^{-1}(u)$ the map $h_i$ is not rigid. This implies that $h_i$ can be deformed and keeping its boundary on $L_u,$ which gives rise to a non-trivial deformation of the closed tropical descendant invariants $h \colon \Gamma \rightarrow \R^2$ as well. This contradicts the virtual dimension zero assumption and hence $MI'(h_i) \leq 2$ for all $i.$ 

Since $|\Delta|-n =\nu+2$, one has from Definition \ref{def:tropMI},
\begin{equation}
 \sum_{i=1}^{\nu+2} MI'(h_i)= \sum_{i=1}^{\nu+2}2(|\Delta_i(h_i)|-k_i)=2(\nu+2),
 \nonumber
\end{equation}
where $k_i$ is the number of domain marked points for the tropical disc $h_i.$ This implies $MI'(h_i)=2$ for all $i$.
\end{proof}

We end this subsection by introducing the tropical descendants following \cite{G7} and \cite{MR16}. For this, we first need to define relevant weights for the tropical counting.

\begin{defn} \label{defn:tropical_descedant}
   Let $(h,T,w)\in \mathcal{M}^{trop}_{\Delta,n}(X,p_1,\cdots,p_k,\psi^{\nu} u)$. Let $x\in T^{[0]}$ be the distinguished vertex with valency $\nu+2$. Let $n_{\rho}$ be the number of unbounded edges adjacent to $x$ and parallel to $\rho\in \Sigma(1)$. Then we define 
      \begin{align*}
         \mbox{Mult}'(h):=\prod_{\rho\in \Sigma(1)}\frac{1}{n_{\rho}!}\prod_{v\in T^{[0]},val(v)=3}\mbox{Mult}_v.
      \end{align*} Then the tropical descendant is defined to be 
      \begin{align*}
        \langle p_1,\cdots, p_{|\Delta|-\nu-2},\psi^{\nu}u\rangle^{trop}_{\Delta,n}=\sum_{h}\mbox{Mult}'(h),
      \end{align*} where $h$ sum over all tropical curves in $\mathcal{M}^{trop}_{\Delta,n}(X,p_1,\cdots,p_{|\Delta|-\nu-2},\psi^{\nu} u)$.
\end{defn} 
 The tropical descendant Gromov-Witten invariants are defined in the way to match the enumerative geometry. 
  \begin{theorem} \cite{G7}
     For $X=\mathbb{P}^2$, the tropical descendant invariants coincide with the ordinary descendant Gromov-Witten invariants 
       \begin{align*}
          \sum_{\Delta:|\Delta|=dH}\langle p_1,\cdots, p_{3d-\nu-2},\psi^{\nu}u\rangle^{trop}_{\Delta,n}=\langle p_1,\cdots, p_{3d-\nu-2},\psi^{\nu}u\rangle_{dH},
       \end{align*} where $H$ is hyperplane class and $\langle p_1,\cdots, p_{3d-\nu-2},\psi^{\nu}u\rangle_{dH}$ denotes the ordinary descendant Gromov-Witten invariants in class $dH$. 
  \end{theorem}
  \begin{remark}
     It is proved tropical-holomorphic correspondence of descendant Gromov-Witten invariants with more general $\psi$-class insertions for $\mathbb{P}^2$. However, the weights for tropical descendant curves are still mysterious to the authors. 
  \end{remark}
  Tropical geometry and log geometry are naturally arising together. It is well-known that the tropical rational curves are non-superabundant. As a consequence of the main result of Mandel-Ruddat, for non-superabundant curves one has the following correspondence between tropical descendant invariants and descendant log Gromov-Witten invariants.
  \begin{theorem}\cite{MR16} \label{thm:MR_main}
   Let $\langle q_1,\cdots, q_k,\psi^{\nu}q_u\rangle^{X(log D)}_{\Delta,n}$ be the log Gromov-Witten descendant invariant of degree $\Delta$, then 
    \begin{align*}
    \langle p_1,\cdots, p_k,\psi^{\nu}u\rangle^{trop}_{\Delta,n}=\frac{1}{|\mbox{Aut}(\Delta)|}\langle q_1,\cdots, q_k,\psi^{\nu}q_u\rangle^{X(log D)}_{\Delta,n}
   \end{align*} 
  \end{theorem}
Here we omit the definitions of descendant log Gromov-Witten invariants for readability and refer to \cite[Section 3]{MR16} for details.

\section{Lower order holomorphic bulk-deformed potential for $\mathbb{P}^2$}\label{sec:wall_crossing_Wk}
In this section, we will compute $W_k$ of $\mathbb{P}^2$ for $k=1,2$ explicitly via two different methods. One can find the tropical analogue in \cite{G2}. 
We first provide the initial wall structure for the holomorphic bulk-deformed potential $W_k(u)$. We illustrate some of the difficulties lying ahead of determining the holomorphic wall structure on the log base without a tropical-holomorphic correspondence Theorem \ref{thm:correspondence} in the case of $X=\mathbb{P}^2.$
The arguments in this section are also valid for all toric Fano surfaces, but we focus on $\mathbb{P}^2$ in this section to make the exposition more explicit.

For the first-order potential, all the relevant walls are induced by Maslov index 2 discs which are known to be regular by Cho-Oh \cite{CO}, and hence we can illustrate the wall structure more concretely. For the second order, we can still compute the loci of corresponding walls by direct computation, but we will see that they experience some nonlinear behaviors unlike the first order ones, which motivates us to seek for another method for higher orders in the next section. 

Our setting in this section is as follows.
We fix a generic point $q_1$ in $X=\mathbb{P}^2$, and bulk-deform the Lagrangian Floer complex of toric fibers $L$ by $\mathfrak{b} = t_1 q_1$ where the parameter $t_1$ is taken from $R_1=\Lambda[ t_1 ] / t_1^2$.  The condition $t_1^2 =0$ automatically discard the contributions by holomorphic discs with more than one interior markings, and the resulting bulk-deformed Floer theory can be interpreted as the first-order deformation.

\subsection{Wall-crossing for first-order bulk-deformed potentials} 
The first-order bulk-deformed potential $W_1 (u)$ for a torus fiber $L_u$ is given by the coefficient of the unit class appearing in
$$\sum m_{k,\beta}^{\mathfrak{b}} (b, \cdots, b), $$
where $\mathfrak{b}=t_1 q_1$ as above, and $b= x_1 e_1 + x_2 e_2 \in H^1 (L_u)$ for a toric fiber $L_u$ in $\mathbb{P}^2$. We will write $W_{k,u}$ for $W_k(u)$ in this section.

Due to the relation $t_1^2=0$, the first order bulk-deformed potential $W_{1,u}$ has only two components
$$ W_{1,u} (z_1,z_2) = W_{0,u} (z_1,z_2) + t_1 F_u (z_1,z_2),$$
where $z_i = e^{x_i}$ and $W_{0,u} (z_1,z_2) = T^{u_1} z_1 + T^{u_2} z_2 + \frac{T^{1-u_1-u_2}}{z_1 z_2}$ is the Hori-Vafa potential studied in \cite{CO}, and $F$ counts holomorphic discs with one interior insertion constrained by $q_1$. Let $\beta_0,\beta_1,\beta_2$ be the generators of $\pi_2 (X,L_u)$ represented by the standard Maslov 2 discs that are responsible for the terms $\frac{1}{z_1 z_2}$ and $z_1,z_2$ in $W_{0,u}$. Then $\beta$ can be written as a linear combination of $\beta_i$'s. By degree reason, the term $F_u$ is contributed by Maslov 4 discs which are in the classes $\beta_i + \beta_j$ with $i \neq j$. By Lemma \ref{799}, we do not have contributions from classes of the type $2 \beta_i$ by generic choice of the point constraint $q_1$.

\begin{figure}[htb!]
    \includegraphics[scale=0.55]{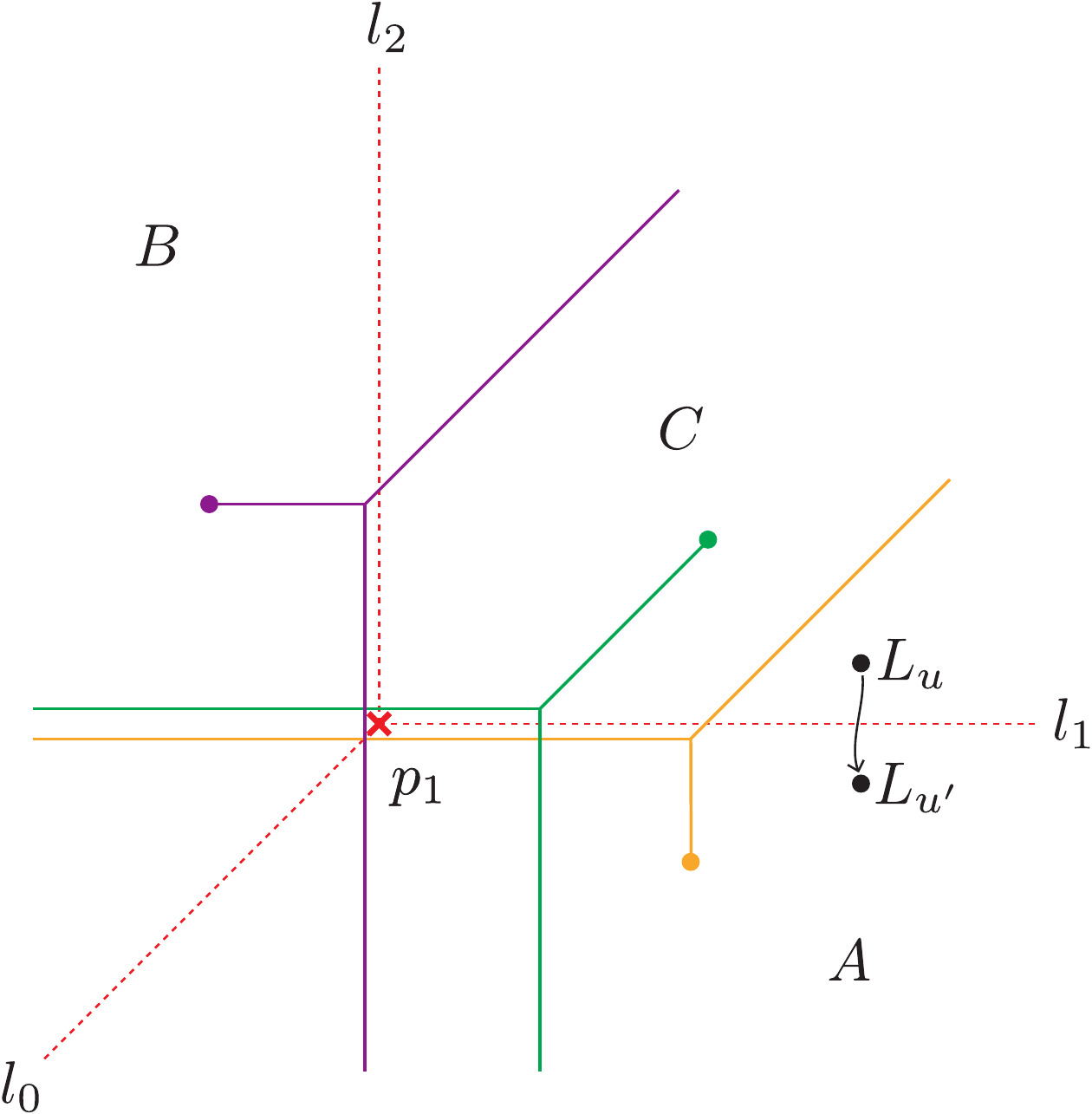}
    \caption{First order walls drawn on the Log base. 
    }
		\label{fig:obulkP2}
\end{figure}


Therefore, 1we are interested in the behavior of the moduli space
$$\mathcal{M}_{k,1} (L_u,\beta;q_1) = \mathcal{M}_{k,1} (L_u,\beta) _{ev^{int}_1} \times \{ q_1\}$$
for $\beta=\beta_i + \beta_j$ when $L_u$ varies among toric fibers. We will show that this moduli space jumps when we cross one of rays depicted in Figure \ref{fig:obulkP2}, or in other words, there is a wall-crossing. These rays are indeed the images of boundaries of Maslov 2 discs that pass through $q_i$ under the fibration map $log$. As drawn in the Figure \ref{fig:obulkP2}, there is an obvious change in countings of \emph{tropical discs} which pass through $p_1 :=\mu (q_1)$ for different chambers $A$, $B$ and $C$. One can consult Section \ref{sec:trop_sec} for details on the analogous tropical disc counting.

In what follows, we shall provide precise Floer theoretical analysis for this wall-crossing phenomenon. The wall-crossing occurs since the $A_{\infty}$ homomorphism defined in \eqref{eqn:induced_on_MC} induces a non-trivial map between the weak Maurer-Cartan spaces associated to a torus fiber in one chamber and that of the another chamber. As explained in Section \ref{subsec:canmcelt}, this map is contributed by the following moduli space of holomorphic discs
\begin{equation}\label{eqn:ftwcmod}
 \mathcal{M}_{k,1 } (L_u,\beta_i, J_t; q_1) := \mathcal{M}_{k,1} (L_u,\beta_i, J_t) _{ev^{int}_1} \times \{ q_1\}
\end{equation}
associated to the $1$-parameter family of almost complex structure $J_t$ considered in \eqref{eqn:J_t}.
It can observed that the virtual dimension of \eqref{eqn:ftwcmod} becomes $-1$  after forgetting the boundary marking. Hence for generic $L_u$, this moduli is empty. 


By the classification of Maslov 2 discs in \cite{CO}, this moduli space becomes nonempty precisely when $L_u$ lies over the three rays $l_0, l_1, l_2$ in Figure \ref{fig:bulkP2} that emanates from $p_1$. Fukaya's trick gives an identity unless two Lagrangians are in different chambers separated by these rays. Thus, one can conclude the following Lemma.

\begin{lemma}
The moduli space $\mathcal{M}_{k,0} (L_u ,\beta_i ; q_1) \neq \phi$ if and only if $L_u$ lies over the rays $l_i$ in the $log$ base.
\end{lemma}

Moreover, the wall-crossing maps are given explicitly as follows. 

\begin{prop}\label{prop:wcmap1p2}
When one travels in clockwise direction around $p$, the wall-crossing maps for three rays $l_0,l_1$ and $l_2$ are given as follows. Namely, along the directions $A \to B$, $C \to A$ and $B \to C$ in Figure \ref{fig:bulkP2}, we have
\begin{equation}\label{eqn:wcformulae1p2}
\begin{array}{lclcl}
(z_1,z_2) &\mapsto&  \left(z_1( 1 + \frac{t_1}{z_1 z_2}  T^{\epsilon_0} ) , z_2 ( 1 + \frac{t_1}{z_1 z_2}  T^{\epsilon_0} )^{-1} \right) &  \mbox{for}& l_0 \\
(z_1,z_2) &\mapsto& \left( z_1, z_2( 1 + t_1 z_1 T^{\epsilon_1} ) \right) & \mbox{for}& l_1 \\
(z_1,z_2) &\mapsto& \left( z_1 ( 1+ t_1 z_2 T^{\epsilon_2})^{-1} ,z_2 \right) & \mbox{for}& l_2
\end{array}
\end{equation}
for some $\epsilon_i>0$.
\end{prop}

If we introduce new variable $z_0 = \frac{1}{z_1 z_2}$, then the first map in \eqref{eqn:wcformulae1p2}  also simplifies into a standard form of the cluster transformation.

\begin{proof}
We will prove the formula for $l_1$ as the rest is similar. As in Figure \ref{fig:bulkP2}, one chooses two Lagrangian torus fibers $L_u$ and $L_{u'}$ lying on different sides of $l_1,$ and a smooth isotopy $\phi_{t}$ for $0 \leq t \leq q$ between them along the paths drawn in Figure \ref{fig:obulkP2}. For $J_t=(\phi_t)_*J,$ we recall that the moduli space
\begin{equation}\label{eqn:cupsmc}
 \mathcal{M}_{k,1 } (L_u,\beta_i, \mathcal{J}; q_1) :=\bigcup_{t \in [0,1]} \{t\} \times \mathcal{M}_{k,1} (L_u,\beta_i, J_t) _{ev^{int}_1} \times \{ q_1\} 
\end{equation}
induces a map on the weak Maurer-Cartan spaces of $L_u$ and $L_{u'}$ is given by
$$f_\ast : H_1 (L_u) \to H_1 (L_{u}) \quad x \mapsto f_0 (1) + f_1(b) + f_2(b,b) + \cdots,$$
defined explicitly in equation \eqref{eqn:comptwf}.

One observes that the only nonempty slice of \eqref{eqn:cupsmc} is when $\phi_t (L_u)$ precisely sits over $l_1$. Since the corresponding disc has class $\beta_1$, the divisor axiom proved in Lemma \ref{lem:divisor} implies that
$$ f_\ast (x_1 e_1 + x_2 e_2) =  x_1 e_1+  (x_2 + t_1 e^{x_1} T^{\epsilon_1}) e_2,$$
where $\epsilon_1=\omega(\beta_1)$, which is equivalent to the desired statement in exponential coordinates. One notices that
$$ e^{x_2 + t_1 e^{x_1} T^{\epsilon_1}} = e^{x_2} ( 1+ t_1 e^{x_1} T^{\epsilon_1} ) = z_2 ( 1+ t_1 z_1 T^{\epsilon_1})$$
since $t_1^2=0$. This completes the proof.
\end{proof}

\begin{remark}
Notice that the composition of three wall crossing maps produces a nontrivial monodromy while the assumption in Proposition \ref{1086} fails.
\end{remark}

Once we fix the order $k$, Maslov indices of holomorphic discs contributing to both the wall-crossing maps and the bulk-deformed potentials are bounded from above. Hence there are finitely many terms with respect to exponential variables $z_1$ and $z_2$ due to the Fano condition and divisor axioms \ref{lem:divisor}. In the following discussions, we will set $T=1$ to simplify the expressions of potentials.

\subsection{Computation of $W_1$}\label{subsec:W1comp}
Recall that we have defined $W_{1,u} = W_{0,u} + t_1 F_u,$ where $F_u$ counts Maslov $4$ discs passing through $q_1$ which boundary on $L_u,$ or a generalized Maslov two disc contributing to $W_{1,u}$. Due to wall-crossing, the expressions of $F_u$ differ in each chamber. We choose generic points $u_A,$ $u_B$ and $u_C$ in the chambers $A$,$B$ and $C$ respectively as in Figure \ref{fig:bulkP2}, and we compute $F_{u_A}$, $F_{u_B}$ and $F_{u_C}$ using wall-crossing maps as follows.

\begin{figure}[htb!]
    \includegraphics[scale=0.55]{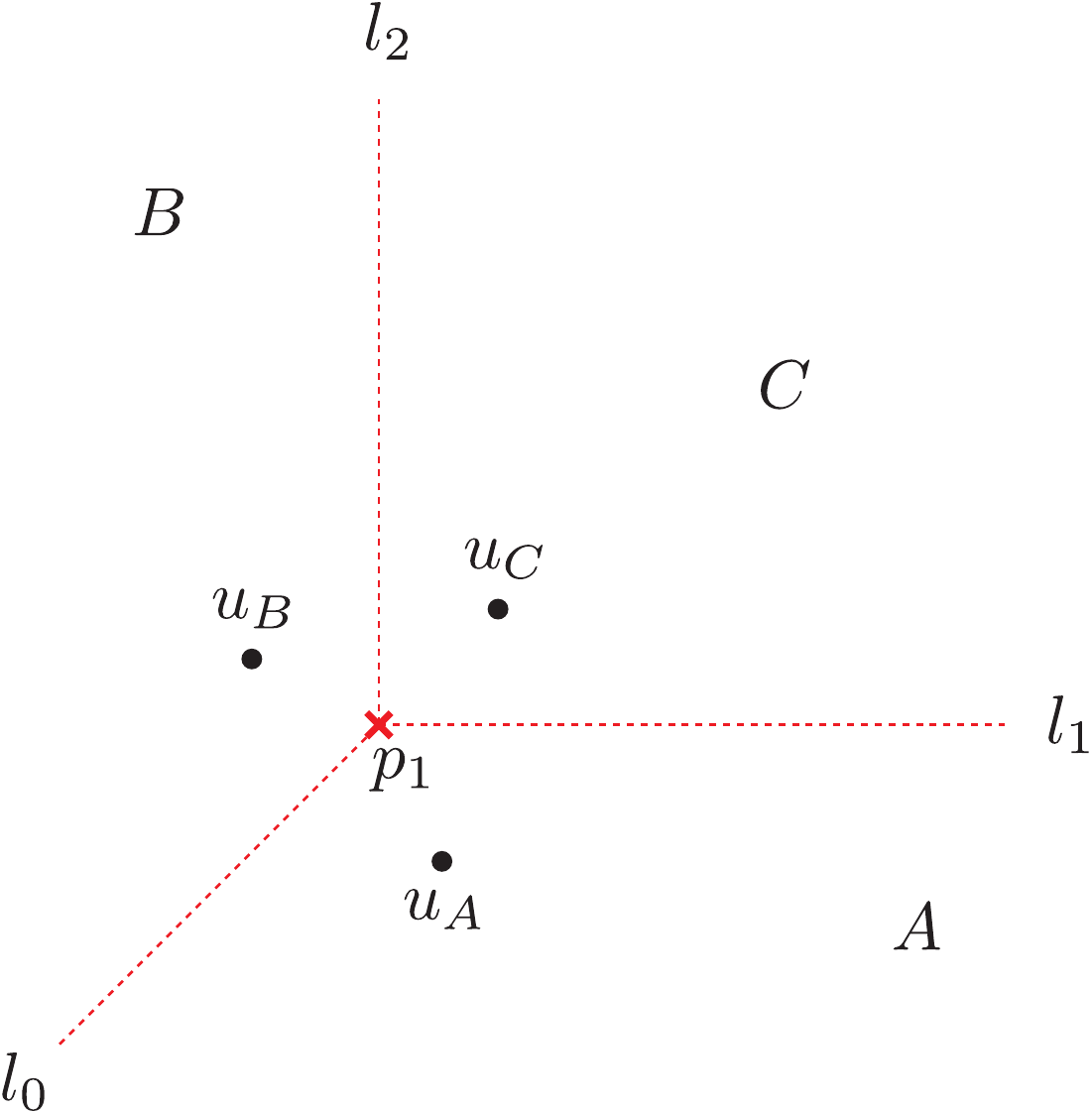}
    \caption{Three chambers induced by the point constraint $q_1$ ($p_1=log (q_1)$).}
		\label{fig:bulkP2}
\end{figure}

Since $W_1$ is compatible with wall-crossing maps,
one observes that $F_{u_A}$ does not depend on torus fibers $L_{u_A}$ as long as $u_A$ stays in $A$ given the condition that $T=1$, and similarly for $F_{u_B}$ and $F_{u_C}$. By the classification theorem \cite{CO} and by Lemma \ref{799}, the only possible contributions to $F$ are from primitive classes $\beta_0 + \beta_1$, $\beta_1 + \beta_2$ and $\beta_2 + \beta_0,$ which gives the terms
$$ \frac{1}{z_2}, \quad z_1 z_2, \quad \frac{1}{z_1}$$
in the bulk-deformed potential, respectively. 

Next let us investigate the term $z_1 z_2$. If there exists a holomorphic disc $w: (D^2, \partial D^2) \to (X,L_u)$ of the class $\beta_1 + \beta_2$ passing through $q_1$, then its entire image should lie in $X \setminus \{ z_0 = 0\}$ by the classification of such discs, as otherwise it would have $\beta_0$ component in its class. Therefore, one can compose $|z_1| : X \setminus \{ z_0 = 0\} \to \R$ with $w$ to get a well-defined function on $D^2$. Since this function achieves the value $|q_1|^2$ in the interior of $D^2$, the $x$-coordinate of $w$ must be bigger than that of $p_1$ by maximum principle.  We can apply the same argument to $|z_2|$, and see that $w$ should lie in the chamber $C$. The same is true for other classes $\beta_0 + \beta_1$ and $\beta_2 + \beta_0$. Therefore we have
$$F_{u_A} (z_1,z_2) =   a \frac{1}{z_2}, \quad F_{u_B}(z_1,z_2) = b \frac{1}{z_1}, \quad F_{u_C}(z_1,z_2) =  c z_1 z_2$$
for some constants $a,b,c$. One can determine $a,b,c$ using the wall-crossing maps \eqref{eqn:wcformulae1p2} as follows.

\begin{prop}\label{prop:firstorderWp2}
The first order bulk-deformed potential on the chambers $A$,$B$ and $C$ are given by
\begin{equation*}
\begin{array}{l}
W_{1,{u_A}} (z_1,z_2)= z_1 +  z_2 + \frac{1}{z_1 z_2} + t_1 a  \frac{1}{z_2} \\
W_{1,{u_B}} (z_1,z_2)=   z_1 +   z_2 + \frac{1}{z_1 z_2} + t_1 b   \frac{1}{z_1} \\
W_{1,{u_C}} (z_1,z_2)=  z_1 +   z_2 + \frac{1}{z_1 z_2} + t_1 c  z_1 z_2
\end{array}
\end{equation*}
with $a=b=c=1$.
\end{prop}

\begin{proof}
By \eqref{eqn:comptwf}, $W_1$ in different chambers have compatibility with the wall-crossing maps given in Proposition \ref{prop:wcmap1p2}. For instance, the expression
\begin{equation}\label{eqn:WABcoord}
W_{1, u_A} \left( z_1, z_2( 1 + t_1 z_1   ) \right)
\end{equation}
should agree with $W_{1, u_C} (z_1,z_2)$. Then $W_{1, u_A} \left( z_1, z_2( 1 + t_1 z_1   ) \right)$ becomes
\begin{eqnarray*}
 &&  \ \ z_1   +   z_2 \left(  1 + t_1 z_1    \right)  + \frac{1}{z_1 z_2  \left(  1 + t_1 z_1   \right) } + a  \frac{1}{z_2  ( 1 + t_1 z_1  ) } \\
&&= z_1   +    z_2 + t_1  z_1 z_2 + \frac{  (1 + t_1 a  z_1)  }{z_1 z_2  \left(  1 + t_1 z_1    \right) }.
\end{eqnarray*}
%
Using $\frac{1}{1+t_1 z_1} = 1- t_1 z_1$ due to $t_1^2=0$, one obtains that
$$   z_1   +  z_2 + t_1   z_1 z_2 +   \frac{1}{z_1z_2} + (a-1) t_1 \frac{1}{z_2}. $$
Comparing with the formula of $W_{1,u_C}$, we conclude that $a=c=1$. Similarly, one can show that $b=1$.
 \end{proof}


\subsection{The second order bulk-deformed potential} \label{sec:9999}

We next  consider the case of two point constraints $q_1$, $q_2$ in $\mathbb{P}^2$ and study the behavior of the second order potential $W_2$. We will only give a partial computation in this subsection since we will provide a different approach to compute $W_k$ explicitly for higher-order potentials later in Section \ref{sec:HTcorres}.

Similarly as before, we set our bulk parameters to be
$$ t_1 q_1 + t_2 q_2$$
where $t_i$'s are taken from $\Lambda [ t_1, t_2 ] / t_1^2= t_2^2 = 0 $ as before. As functions in $t_1$ and $t_2$, the second-order potential $W_2$ decomposes into
$$ W_2(z_1,z_2) = W_0 (z_1,z_2) + t_1 F_1^1 (z_1,z_2) + t_2 F_1^2  (z_1,z_2)+ t_1 t_2 F_2 (z_1,z_2),$$
where $F_1^1$ and $F_1^2$ are computed in Proposition \ref{prop:firstorderWp2} above. However, their wall-crossing phenomena depend on different chamber structures. We conclude that the wall structure for $W_2$ should at least contain two trivalent graphs with vertices $p_1= \mu (q_1)$ and $p_2  = \mu (q_2)$. Recall that these are the loci of Lagrangian torus fibers bounding generalized Maslov zero discs passing through $q_1$ and $q_2$, respectively. Thus the only missing part is to count discs that pass through both $q_1$ and $q_2$. Such discs should have Maslov index $4$ in order for the corresponding moduli space to be isolated. By the classification, their relative classes should fall into one of primitive classes
$$\beta_0 + \beta_1 , \beta_1 + \beta_2, \beta_2 + \beta_0$$
as we exclude multiple cover contributions by imposing generic point constraints.

Let us first find all possible contributions in the class of $\beta_1 + \beta_2$.
Fix two points 
$$q_1:=(a,b), \ q_2:=(c,d)\in \C^2 = \mathbb{P}^2 \setminus \{z_0 =0\}. $$ For convenience, we assume that $a,b,c,d \in \R$ as an example to demonstrate the non-linearity of the second-order wall. Any holomorphic discs with Maslov index $4$ in class $\beta_1 +\beta_2$ can be written as
$$u(z): \{ |z| \leq 1\} \to X, \qquad z \mapsto  \left(\lambda \dfrac{z-\alpha}{1- \bar{\alpha}z}, \rho \dfrac{z-\beta}{1-\bar{\beta}z} \right).$$
If we choose two interior marked points to be at $0$ and $r \in (0,1)$ inside the unit disc, respectively, then the incident condition with $q_1$ and $q_2$ gives
$$ u(0) = q_1=(a,b) ,\quad u(r) = q_2 = (c,d).$$
Hence,
$$ (-\lambda \alpha, -\beta \rho) = (a,b),\quad \left(\lambda \dfrac{r-\alpha}{1-\bar{\alpha} r},  \rho\dfrac{r-\beta}{1-\bar{\beta} r}\right) = (c,d).$$
From the first equation, one concludes that $\alpha= - a / \lambda$ and $\beta = -b / \rho$. Plugging it to the second equation gives
\begin{equation*}
\begin{array}{l}
\lambda \dfrac{r+ \frac{a}{\lambda}}{1+ r \frac{a}{\bar{\lambda}}} = c \\
\rho \dfrac{r+ \frac{b}{\rho}}{1+ r \frac{b}{\bar{\rho}}} = d
\end{array},
\end{equation*}
or equivalently,
\begin{equation}\label{eqn:quadeq2order}
\begin{array}{l}
|\lambda|^2 r + a \bar{\lambda} = c \bar{\lambda} + r ac \\
| \rho|^2 r + b \bar{\rho} = d \bar{\rho} + bd r 
\end{array}
\end{equation}
which forces $\lambda$ and $\rho$ to be real numbers.

We are interested in the locus that the tours fibers bounding such discs draw in the base of the torus fibration
$$(\log |z_1|, \log |z_2|) : \C^2 = \mathbb{P}^2 \setminus \{z_0 = 0\} \to \R^2$$
on $\C^2 = \mathbb{P}^2 \setminus \{ z_0 = 0\}$. 
Note that the image of the boundary of $u$ under the moment map is precisely $(x,y):=(\log |\lambda|,\log |\rho|)$. 
%
We first parametrize $\lambda$ and $\rho$ in terms of $r$ by solving \eqref{eqn:quadeq2order}.
$$\lambda (r) =  \dfrac{(c-a) + \sqrt{(c-a)^2 + 4r^2ac}}{2r} >0$$
$$\rho (r) = \dfrac{(d-b) + \sqrt{(d-b)^2 + 4r^2bd}}{2r} >0.$$
Since $r \in (0,1),$ we  can set $s =\frac{1}{r}$ for $1 \leq s < \infty$ and obtain the explicit formulae
\begin{equation}\label{eqn:estimatess}
\begin{array}{lcl}
\lambda (s) =  \dfrac{(c-a)s + \sqrt{(c-a)^2 s^2 + 4ac}}{2} &\sim& (c-a)s \quad \mbox{for} \quad s \gg 1,\\
\rho (s) = \dfrac{(d-b)s + \sqrt{(d-b)^2 s^2 + 4bd}}{2} &\sim& (d-b)s \quad \mbox{for} \quad s \gg 1 
\end{array}
\end{equation}
Notice that $(\lambda(1),\rho(1)) = ( \max \{a,c \}, \max \{b,d\})$. Namely, the new wall due to the discs in the class $\beta_1 +\beta_2$ emanates from the intersection point between the horizontal wall from $p_1$ and the vertical wall from $p_2$ (see Figure \ref{fig:2ndordern}), which can be interpreted as a scattering phenomenon.


\begin{figure}[htb!]
    \includegraphics[scale=0.45]{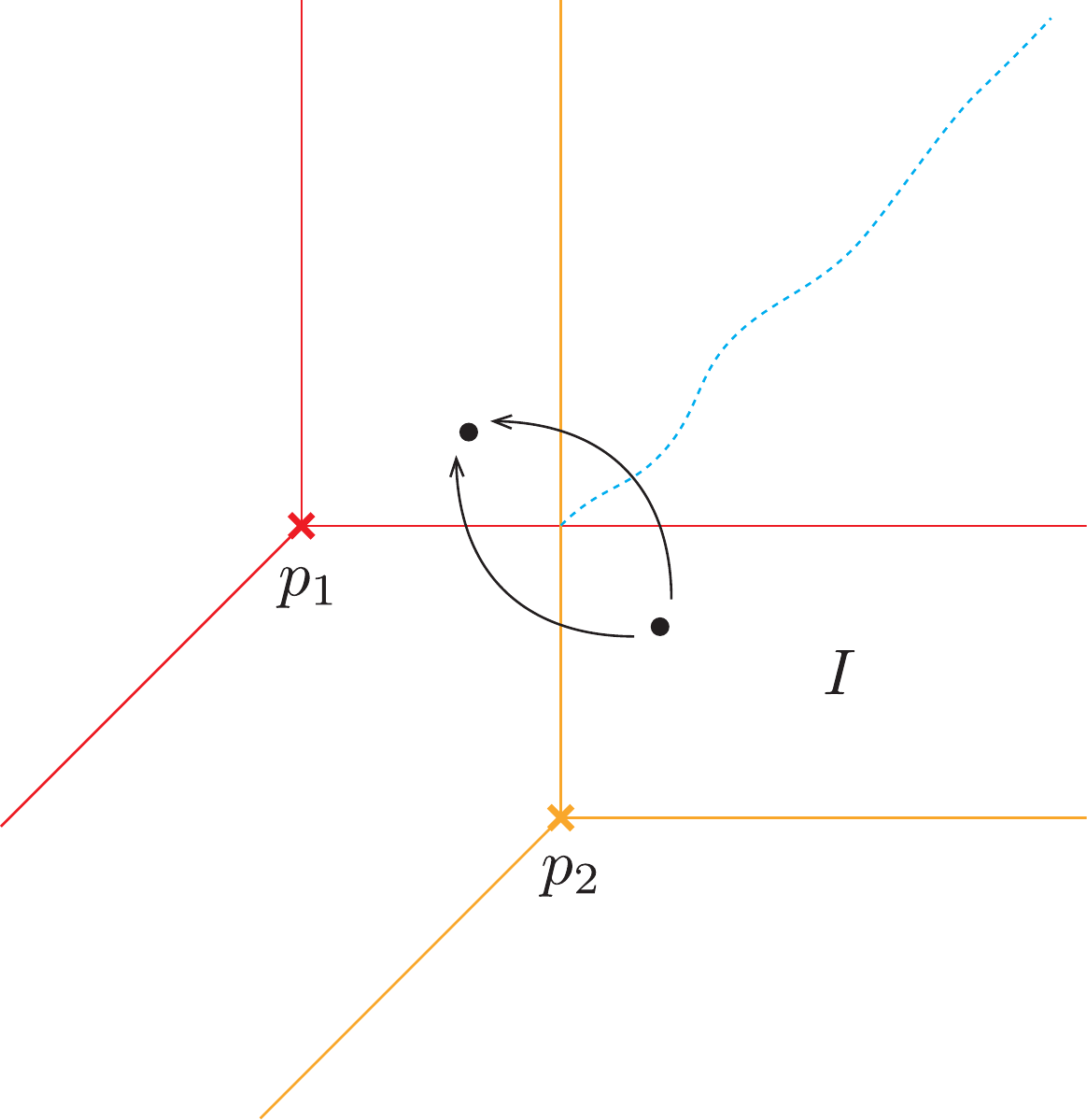}
    \caption{The new wall from scattering}
		\label{fig:2ndordern}
\end{figure}

%
%
%
%

The new wall described above can be also detected by Fukaya's trick \ref{1086}. Consider the two different compositions of the wall-crossing maps from $J$ to $I$ as in Figure \ref{fig:2ndordern}. Since $q_i$ are away from the region enclosed by these two paths, they would produce the same map if there was not an additional wall. However, one can check that these two maps do not agree with each other, which implies that there is an extra contribution to the wall-crossing map. We will use this to compute the precise wall-crossing map for the new wall in the proof of Proposition \ref{prop:wcforw2} below.

One can perform similar computations for other classes of Maslov 4 discs, and the full wall structure is drawn in Figure \ref{fig:2ndorderfull}.  
The additional two walls induced by the classes $\beta_0 + \beta_1$ and $\beta_2 + \beta_0$ can be also understood by the following disc degenerations. From the computation in Section \ref{subsec:W1comp}, we know that there is a Maslov 4 disc in the class $\beta_0 + \beta_1$ which passes through $q_1$ and is bounded by $L_{p_1}$. This holomorphic disc contributes to $W_1$ at $p_2$ by $t_1 \frac{1}{z_2}$. Then it glues with a constant disc at $q_2,$ where $\pi(q_2) = p_2$, and the disc continues downward from $p_2$ to create the vertical wall as in Figure \ref{fig:2ndorderfull}. We will give more detailed analysis on the moduli of such glued discs in Proposition \ref{1301}. This type of walls will be called \emph{extended walls}.

We remark that these new walls together with two trivalent graphs make  Figure \ref{fig:2ndorderfull} into a scattering diagram after straightening the new walls.

\begin{remark}
It is an important observation that the additional walls are not straight lines. In particular, this makes it hard to compute the bulked superpotential by direct tropical counting. 
\end{remark}

\begin{figure}[htb!]
    \includegraphics[scale=0.45]{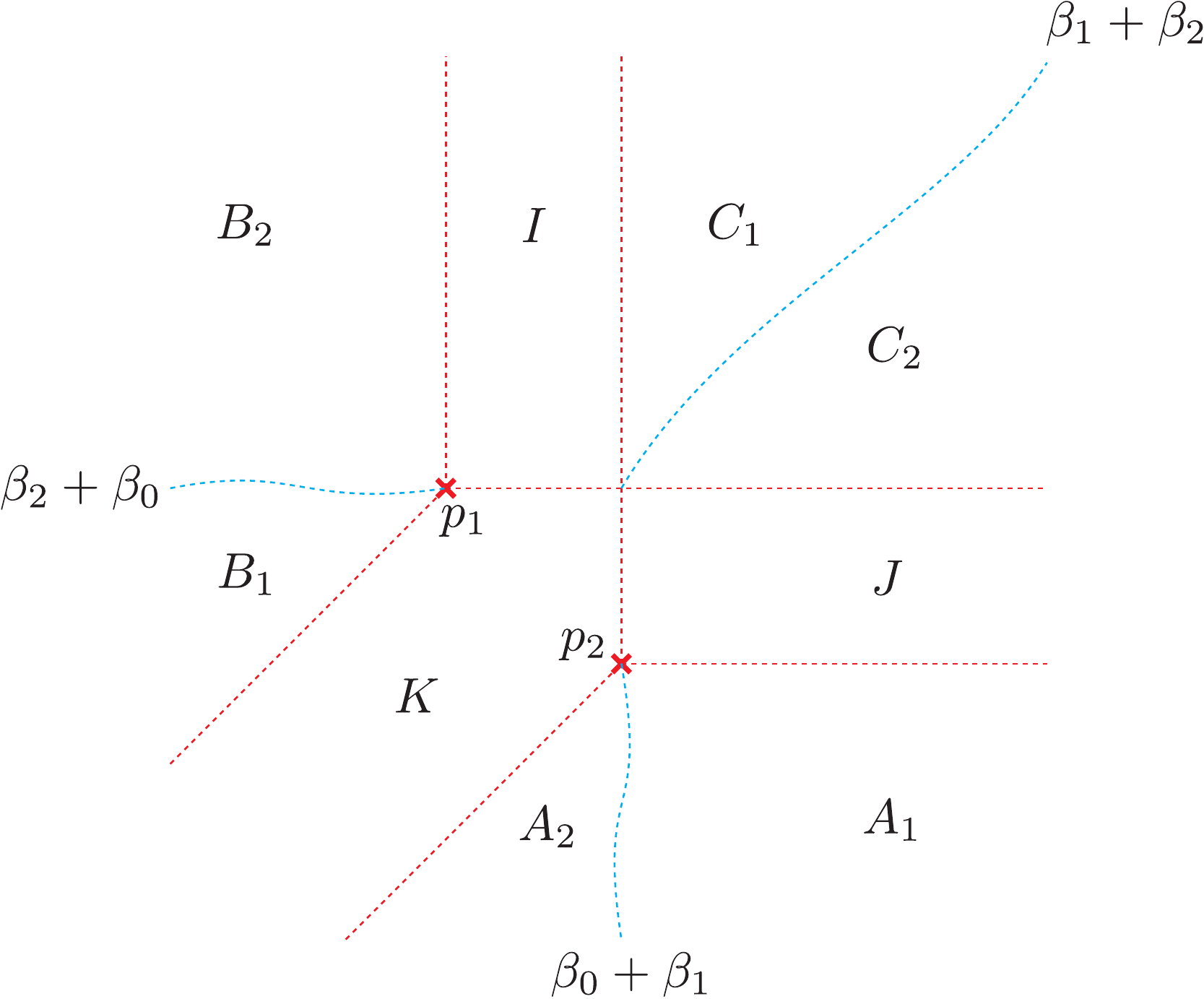}
    \caption{Full wall structure for the second order bulk-deformation}
		\label{fig:2ndorderfull}
\end{figure}

We first show that three chambers $I,J$ and $K$ in Figure \ref{fig:2ndorderfull} do not admit higher order contribution to $W_2$. Such chambers will be called \emph{initial chambers}, and serve as starting regions for computing $W_2$ in other chambers making use of wall-crossing maps.

\begin{lemma}\label{lem:w2ijk}
The second-order potential $W_2$ in chambers $I,J$ and $K$ does not involve terms with $t_1 t_2$, and
\begin{equation*}
\begin{array}{lcl}
W_2|_{I} &=& z_1 + z_2 + \dfrac{1}{z_1 z_2}  + t_1 z_1 z_2 +t_2 \dfrac{1}{z_1}, \\
W_2|_{J} &=&z_1 + z_2 + \dfrac{1}{z_1 z_2}  + t_1 \dfrac{1}{z_2} + t_2 z_1 z_2, \\
W_2|_{K} &=& z_1 + z_2 + \dfrac{1}{z_1 z_2} + t_1 \dfrac{1}{z_2} + t_2 \dfrac{ 1}{z_1}.
\end{array}
\end{equation*} 
\end{lemma}

\begin{proof}
The additional terms in $W_2$ that do not appear in $W_1$ are contributed by Maslov index $6$ discs passing through $q_1$ and $q_2$. By the classification result, these discs are in the classes $a_0 \beta_0 + a_1 \beta_1 + a_2 \beta_2$ with $a_0 + a_1 + a_2 = 3$, where $(a_0, a_1, a_2)$ can takes values in
$$(2,1,0), (2,0,1), (1,2,0), (0,2,1), (1,0,2), (0,1,2),$$ 
as multiple covers are excluded by Lemma \ref{799} again.
 

Let us first look at the chamber $I$. By applying maximum principle to the function $|z_0|$, we concludes that the discs in classes involving $\beta_0$ of the form
$$(2,0,1), (0,2,1), (1,0,2), (0,1,2)$$ 
can not contribute to $W_2|_{I}$. This is because the function $|z_0|$ is well-defined whenever the discs avoid either of the divisors $\{|z_1| = 0\}$ and $\{ |z_2| = 0\}$. Moreover, the other two classes $(2,1,0), (1,2,0)$ are excluded by applying the maximum principle to the function $|z_1|$ and considering its value at $p_2$. Therefore $W_2|_I$ is only contributed by the Maslov $4$ discs passing through one of $p_1$ and $p_2$, which we have computed already. \\
\indent The same argument applies to the chambers $J$ and $K$.
\end{proof}

We next compute the wall-crossing map for each of walls drawn in Figure \ref{fig:2ndorderfull}.

\begin{prop}\label{prop:wcforw2}
Wall crossing maps are given as follows.
\begin{equation*}
\begin{array}{rlcl}
A_2 \to A_1: & (z_1,z_2) & \mapsto & (z_1(1+  t_1 t_2 \frac{1}{z_2}), z_2 ) \\
B_2 \to B_1: &(z_1,z_2) & \mapsto &  (z_1, z_2 (1+t_1 t_2 \frac{1}{z_1})^{-1} ) \\
C_2 \to C_1: & (z_1,z_2)& \mapsto & (z_1(1+ t_1 t_2 z_1 z_2), z_2 (1+ t_1 t_2 z_1 z_2)^{-1})  \\
A_1 \to J: & (z_1,z_2)& \mapsto & (z_1,z_2( 1+ t_2 z_1)) \\
K \to A_2: & (z_1,z_2) & \mapsto & (z_1 (1+ t_2 \frac{1}{z_1 z_2}), z_2 (1+ t_2 \frac{1}{z_1 z_2})^{-1}) \\
B_1 \to K: & (z_1,z_2)& \mapsto & (z_1 (1+ t_1 \frac{1}{z_1 z_2}), z_2 (1+ t_1 \frac{1}{z_1 z_2})^{-1}) \\
I \to B_2: & (z_1,z_2) & \mapsto & (z_1(1+ t_1 z_2)^{-1}, z_2) \\
J \to K \,\, \mbox{and} \,\, C_1 \to I: &(z_1,z_2) & \mapsto & (z_1(1+ t_2 z_2)^{-1}, z_2) \\
K \to I\,\, \mbox{and} \,\, J \to C_2: &(z_1,z_2)  & \mapsto & (z_1, z_2(1+ t_1 z_1))
\end{array}.
\end{equation*}
 
\end{prop}


\begin{proof}
Wall-crossing maps for the initial walls emanating from $p_1$ and $p_2$ are the same as before. We now consider the wall induced by the discs in the class $\beta_1 + \beta_2$. Let us denote this wall by $l_{12}$, and the corresponding wall-crossing map by $f_{12}$. 

We compare two different ways of composing the wall-crossing maps, one along the path (i) and the other along (ii) in Figure \ref{fig:2ndorderfull}. Since we know that there is no wall in the region $K$ and the composition of the wall-crossing maps around the intersection point of two walls from $p_1$ and $p_2$ should give the identity, we conclude that
$$f_{12} = f_{I C_1}\circ f_{KI} \circ  f_{JK} \circ f_{C_2 J} ( = f_{C_1 I}^{-1} \circ f_{KI} \circ f_{JK} \circ f_{J C_2}^{-1} ),$$ 
where $f_{AB}$ denotes the wall-crossing map from a chamber $A$ to $B$. This is because there are no singular fibers or point constraints in the region enclosed by this loop. Hence Fukaya's trick (Proposition \ref{1086}) can be applied. 

Note that the right hand side only consists of first-order wall-crossing maps which we have already computed. Direct computation of this composition of four maps gives the desired formula in the statement.

We next compute the wall-crossing map for the extended walls. 
Let us consider the contributions of discs in the class $\beta_0 + \beta_1$ that pass through $q_1$ and $q_2$. As explained above, this disc defines a wall going downward from $p_2$.
Recall that the potentials $W_J$ and $W_K$ on the chambers $J$ and $K$ do not involve the second order terms as in Lemma \ref{lem:w2ijk}. By applying the wall-crossing maps to $W_J$ and $W_K$ for the initial walls emanating from $p_2$, we obtain two Laurent polynomial on the chamber $A$, which are
\begin{equation*}
\begin{array}{lcl}
W_{A_1} &=& z_1 + z_2 + \dfrac{1}{z_1 z_2} + t_1 \dfrac{1}{z_2} + t_2 \dfrac{1}{z_2} + t_1 t_2 \dfrac{z_1}{z_2},\\
W_{A_2} &=& z_1 + z_2 + \dfrac{1}{z_1 z_2} +t_1 \dfrac{1}{z_2} + t_2 \dfrac{1}{z_2} + t_1 t_2 \dfrac{1}{z_1 z_2^2}.
\end{array}
\end{equation*}
We remark that as these two expressions do not agree, there should be additional walls in $A$, which is one of indirect ways to show the existence of this extended wall.



First one notices that the boundary of $\beta_0 + \beta_1$ does not intersect the cycle in $e_2$-direction generically. 
Therefore, we see that the wall-crossing map $\varphi$ should be of the form
$$ \varphi: (z_1, z_2) \mapsto (z_1 ( 1+ t_1 t_2 f), z_2)$$
for some Laurent polynomial $f$. Also, the expressions $W_{A_1}$ and $W_{A_2}$ should be compatible with $\varphi$ due to $A_\infty$-relations \eqref{eqn:comptwf}. Namely, we have
$$W_{A_2} ( z_1 (1+t_1 t_2 f), z_2) = W_{A_1} (z_1,z_2).$$
Direct computation of the left hand side shows
$$W_{A_2} ( z_1 (1+t_1 t_2 f), z_2) = z_1 + z_2 + \frac{1}{z_1 z_2} + t_1 \frac{1}{z_2} + t_2 \frac{1}{z_2} + t_1 t_2 \left(  z_1 f - \frac{1}{z_1 z_2} f + \frac{1}{z_1 z_2^2} \right).$$
Using the relation $t_i^2=0$, this implies
$$z_1 f - \frac{1}{z_1 z_2} f + \frac{1}{z_1 z_2^2}= \frac{z_1}{z_2}$$
by comparing with $A_1$. Therefore $f(z_1,z_2) = \frac{1}{z_2}$, and 
$$ \varphi: (z_1, z_2) \mapsto \left(z_1 \left( 1+ t_1 t_2 \frac{1}{z_2} \right), z_2 \right).$$

The other wall-crossing maps  can be deduced similarly.
\end{proof}

We are now ready to compute $W_2$ in other chambers besides chambers $I,J,K$. 

\begin{prop}
Indexed by the labeling of chambers as in Figure \ref{fig:2ndorderfull}, the second-order potential $W_2$ can be written as follows.
\begin{equation*}
\begin{array}{lcl}
W_2|_{A_1} &=& z_1 + z_2 + \dfrac{1}{z_1 z_2} + t_1 \dfrac{1}{z_2} + t_2 \dfrac{1}{z_2} + t_1 t_2 \dfrac{z_1}{z_2} \\
W_2|_{A_2} &=& z_1 + z_2 + \dfrac{1}{z_1 z_2} +t_1 \dfrac{1}{z_2} + t_2 \dfrac{1}{z_2} + t_1 t_2 \dfrac{1}{z_1 z_2^2} \\
W_2|_{B_1} &=& z_1 + z_2 + \dfrac{1}{z_1 z_2}+ t_1 \dfrac{1}{z_1} + t_2 \dfrac{1}{z_1} + t_1 t_2  \dfrac{1}{z_1^2 z_2}\\
W_2|_{B_2} &=& z_1 + z_2 + \dfrac{1}{z_1 z_2} + t_1 \dfrac{1}{z_1} + t_2 \dfrac{1}{z_1} + t_1 t_2 \dfrac{z_2}{z_1} \\
W_2|_{C_1} &=& z_1 + z_2 + \dfrac{1}{z_1 z_2}+ t_1 z_1 z_2 + t_2 z_1 z_2 + t_1 t_2 z_1 z_2^2 \\
W_2|_{C_2} &=& z_1 + z_2 + \dfrac{1}{z_1 z_2} + t_1 z_1 z_2 + t_2 z_1 z_2 + t_1 t_2 z_1^2 z_2
\end{array}
\end{equation*}
\end{prop}

\begin{proof}

Recall that $W_2$ in different chambers are related compatibly with the wall-crossing maps.
Hence, having the formulae for $W_2|_{I}$, $W_2|_{J}$ and $W_2|_{K}$, the 2nd order potential in the other six chambers are simply obtained by applying the wall-crossing formulae in Proposition \ref{prop:wcforw2}. 
\end{proof}

The first and second-order potentials $W_1, W_2$ obtained from holomorphic disc counting precisely matches the tropical counting in \cite{G7}, apart from the fact that the new wall appearing in the second stage are not straight lines with respect to the standard affine structure on $\R^2$. 
\subsection{Wall-crossing via Floer theoretical isomorphism}
There is another way to capture the wall-crossing for the bulk-deformed potential using isomorphisms between  Lagrangian tori as objects of the Fukaya category. 

Let us consider a wall-crossing locally around the vertical wall $l_2$ in $\mathbb{P}^2$ induced by a generic point-constraint $q$ as in (a) of Figure \ref{fig:localconic}. 
Since two different torus fibers do not intersect, it is hard to make sense of isomorphisms between them. For this reason, we perturb torus fibers with help of the trivial conic fibration defined in (b) of Figure \ref{fig:localconic}. Such an idea was first appeared in \cite{Seinote}.

Two Lagrangians $\mathbb{L}_1$ and $\mathbb{L}_2$ that intersect are analogues of the torus fibers below/above the wall. We equip these two Lagrangians with flat connections $\nabla_1$ and $\nabla_2$ whose holonomies are parametrized by $(z_1,z_2)$ and $(z_1',z_2')$, respectively. More precisely, $z_1,z_1'$ are holonomies along the direction of the fiber circles of the conic fibration in the Lagrangians, and $z_2,z_2'$ are along that of the base circles.

\begin{figure}[htb!]
    \includegraphics[scale=0.45]{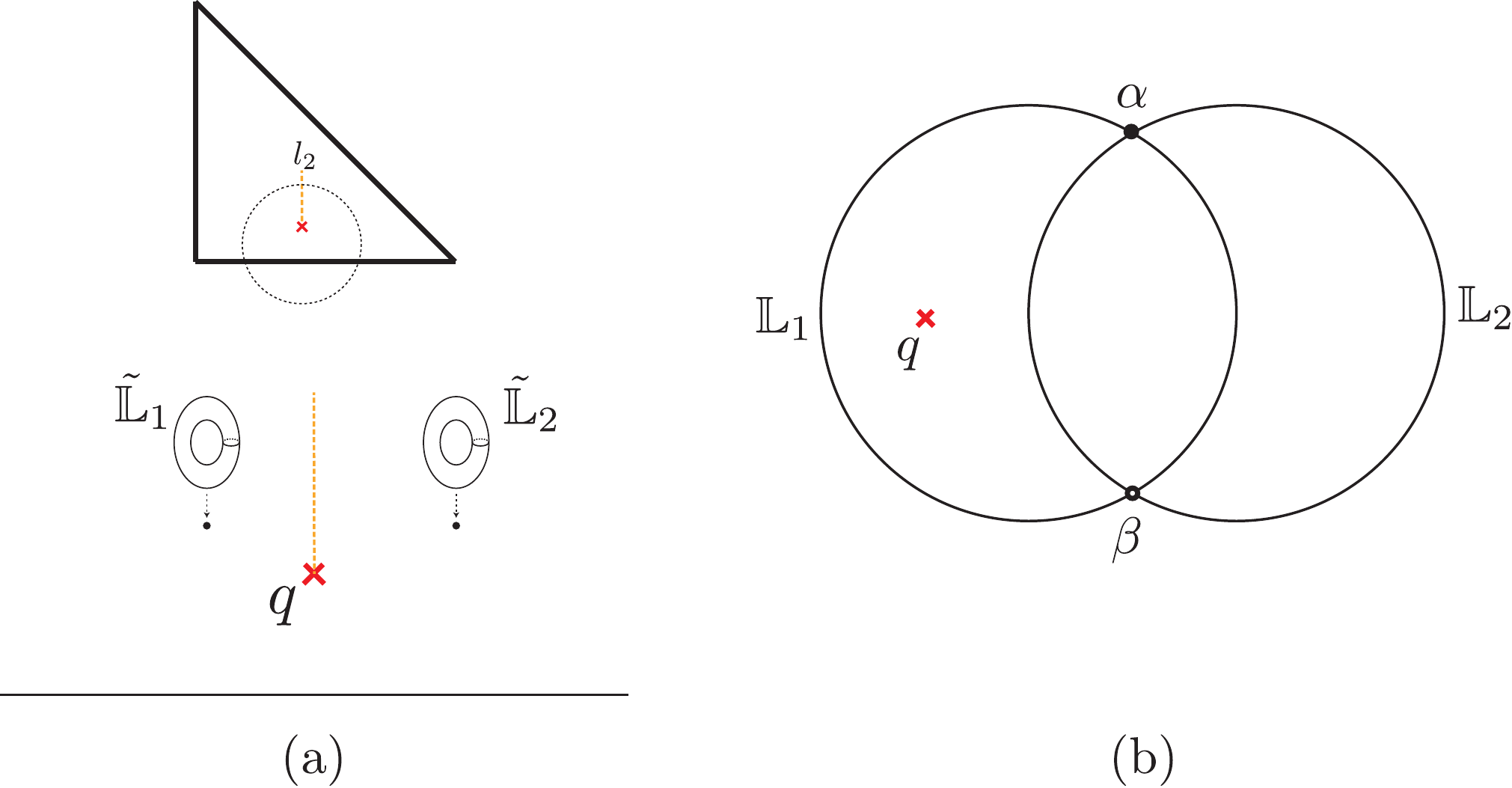}
    \caption{Two disjoint fibers $\tilde{\mathbb{L}}_1$ and $\tilde{\mathbb{L}}_2$ in (a) are deformed into $\mathbb{L}_1$ and $\mathbb{L}_2$ in (b).}
		\label{fig:localconic}
\end{figure}

Let $CF((\mathbb{L}_1, \nabla_1),(\mathbb{L}_2, \nabla_2) )$ be the Floer complex. In what follows, we will find the conditions under which the two objects $(\mathbb{L}_1, \nabla_1)$ and $ (\mathbb{L}_2, \nabla_2)$
in the Fukaya category are isomorphic to each other. 
%
As the two Lagrangians intersect along two disjoint circles, one chooses generic Morse functions on these circle to perturb the intersections. The direct sum of two Morse complexes will serve as a model for $CF((\mathbb{L}_1, \nabla_1),(\mathbb{L}_2, \nabla_2) )$. We denote the corresponding Floer differential by $d$ which counts pearl trajectories, which can possibly pass through $q$.

Notice that the degrees are shifted by $1$ for the Morse complex of one of the circles due to Floer theoretic intersection of two base circles in (b) of Figure \ref{fig:localconic}. There exists a unique degree 0 element, say $\alpha_{12}$, 
and there are two degree 1 elements, one lying in the same circle with $\alpha_{12}$ and the other lying in the other circle. We denote them by $\overline{\alpha_{12}}$ and $\beta_{12}$, respectively. Under the conic fibration, $\alpha_{12}$ and $\overline{\alpha_{12}}$ project to $\alpha$ in Figure \ref{fig:localconic}, where $\beta_{12}$ project to $\beta$.

We first count the strips contributing to coefficients of $\beta_{12}$ in $d(\alpha_{12})$. In Figure \ref{fig:L1L2}, we observe that there are two such strips, but they are weighted by different holonomies. Their contributions sum up to
$$ z_2 \beta_{12} - z_2' \beta_{12},$$
where the signs can be chosen in suitable way by adjusting the definitions of $\alpha,\beta,z$'s.

The coefficients of $\overline{\alpha_{12}}$ in $d(\alpha_{12})$ involves more complicated contributions. Two Morse flows with weighted by different holonomies give rise to $(z_1 - z_1') \overline{\alpha_{12}}$. Furthermore, there is a U-shape strip drawn in (b) of Figure \ref{fig:L1L2}. Originally, the moduli space of such strips is not isolated, but we pick precisely one in the moduli by requiring the strip to pass through $q$. 
This strip induces the terms $t z_2' z_1 T^\epsilon$ in $d(\alpha_{12})$ where $\epsilon$ is the symplectic area of the strip. $t$ appears in the term since this strip admits a nontrivial bulk-deformation by $t q$. Here, $t^2=0$ is crucial, as otherwise one would have to count more complicated strips of higher indices that wrap around the same region several times. 

\begin{remark}
This indeed depends on the position of $q$ along the conic fiber. If the point $q$ is away from certain range, all such $U$-shaped strips can avoid $q$, and hence do not contribute to $d$. This is analogous to perform Fukaya's trick for the path passing the region below $q$ in (a) of Figure \ref{fig:localconic} where there is no wall.
\end{remark}

\begin{figure}[htb!]
    \includegraphics[scale=0.4]{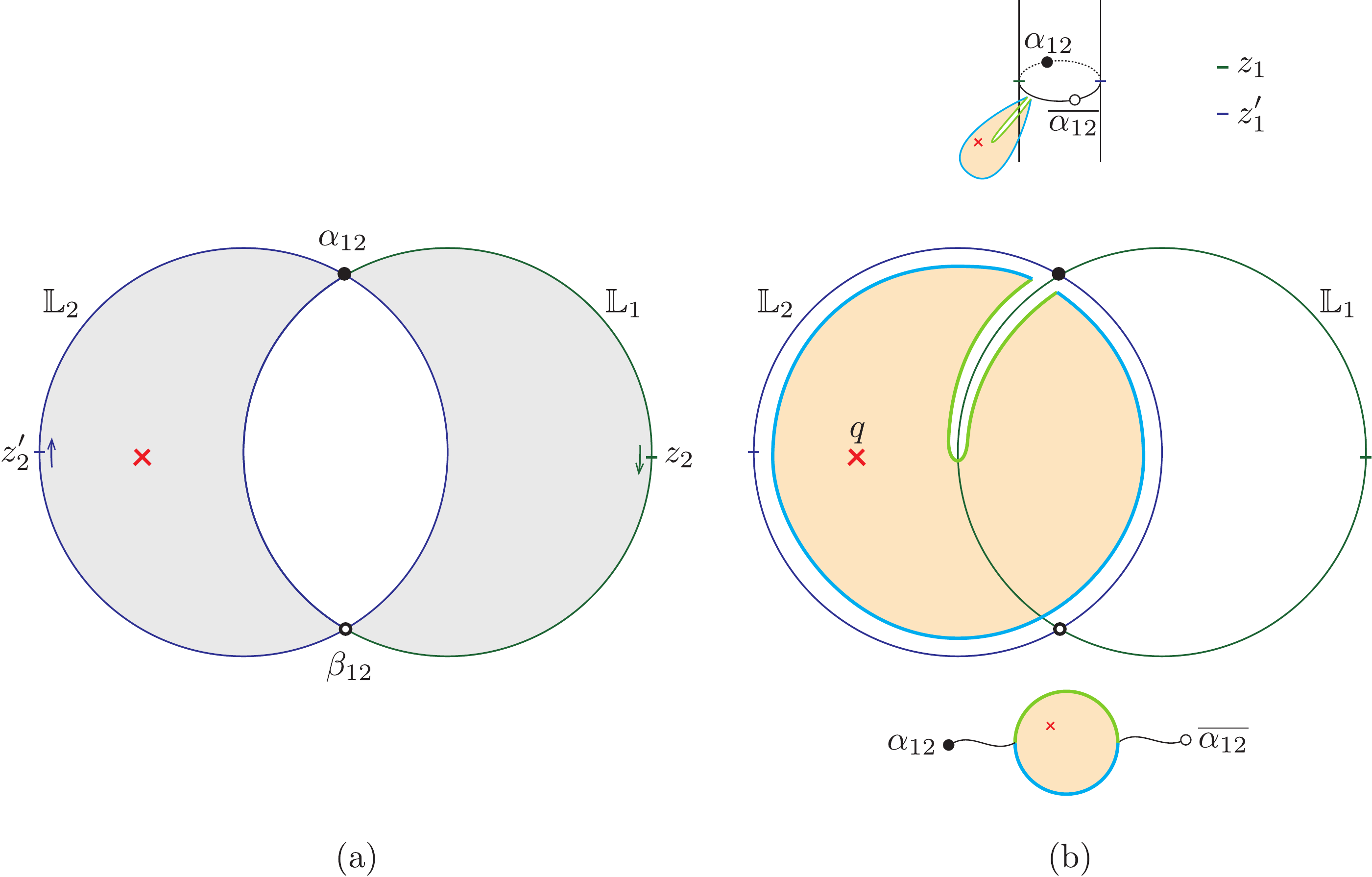}
    \caption{Holomorphic strips counted for the Floer differential of $\alpha_{12}$}
		\label{fig:L1L2}
\end{figure}

In total, we have
$$d(\alpha_{12}) = (z_2 - z_2')\beta_{12} + (z_1 - z_1' + t z_2' z_1 T^{\epsilon} )\overline{\alpha_{12}},$$
and $\alpha_{12}$ is $d$-closed if and only if
\begin{equation*}
\left\{
\begin{array}{l}
z_2' = z_2 \\
z_1' = z_1 ( 1+ t z_2 T^{\epsilon}).
\end{array}\right.
\end{equation*}
One can also check that $(\mathbb{L}_1, (z_1,z_2))$ and $(\mathbb{L}_2, (z_1',z_2'))$ are isomorphic under this condition. In fact, the element $\overline{\beta_{12}},$ viewed as a morphism in the opposite direction, gives an inverse of $\alpha_{12}$. Notice that the formula obtained this way coincides with our earlier computation in Proposition \ref{prop:wcmap1p2}.

\section{Tropical-Holomorphic Correspondence}\label{sec:HTcorres}
\indent Based on the computations of initial walls for $W_0(u)$ and $W_1(u),$ one observes that the holomorphic walls are in general \textit{nonlinear}, or not straight lines in the affine base. Due to this ambiguity of the wall structures of $W_k(u)$ for $k \geq 2,$ it is difficult to calculate $W_k(u)$ for all $k$ directly in holomorphic setting, i.e., by the wall-crossing techniques given in Proposition \ref{1086}. One solution to resolve this issue is to consider the \textit{rescaled} bulk-deformed potential defined as follows. 
\begin{defn}\label{def:holo_Wk}
Given $k$ generic marked points $q_1, \cdots, q_k$ in $(\C^*)^2$ and fixed Lagrangian torus fiber $L_u$ for each $t \gg 1,$
\begin{enumerate}
\item[(1)] the $k$th order rescaled bulk-deformed potential is defined as
$$W_k^t (u):=W^{H_t^{-1}(q_1),\cdots H_t^{-1}(q_k)}_k(H^{-1}_t(u)),$$
where $W^{H_t^{-1}(q_1),\cdots H_t^{-1}(q_k)}_k(H^{-1}_t(u))$ is the counts of generalized Maslov two discs subject to the $k$-point constraints given by ${\textbf{q}}_t$ and the condition that $\mu(\beta)=2k+2$ as in Definition \ref{defn:bulk_potential}.
\item[(2)] The \textit{the holomorphic walls} $\mathcal{W}^{holo}_t$ is defined to be the loci in the Log base $\R^2$ over which the Lagrangian torus fiber $L_u$ bounds generalized Maslov zero $J$-holomorphic discs subject to the constraint $\tilde{\textbf{q}}_t:=(H_t^{-1}(q_1), \cdots, H_t^{-1}(q_k))$.
\end{enumerate}
\end{defn}


In the following section, we will prove that as $t \rightarrow \infty,$ the holomorphic walls defined by this rescaled bulk-deformed potential $W_k^t (u)$ converges to that of the tropical walls given by straight line in the affine base.
(Recall from \eqref{eqn:estimatess} that we have seen such a feature in the computation of $W_2$ for $\mathbb{P}^2$.)
Hence at least for $t \gg 1,$ one still explicitly determine the wall structures and compute $W_k^t (u)$ for all $k$ via tropical geometry. In the subsequent discussions, all the $k$th bulk-deformed potentials are taken to be the \textit{rescaled} bulk-deformed potentials $W_k^t (u)$. If there are no further confusions, we will denote the rescaled bulk-deformed potentials as $W_k(u)$ as well.


\subsection{An anti-symplectic involution of $ (\C^*)^2$} 
 It is an important observation of Nishinou \cite{N2} that every holomorphic disc with boundary on a moment map torus fiber can be extended to a rational curve. 
 We observe that such an extension can be realized making use of an anti-symplectic involution of $X\backslash D \cong (\C^*)^2$, which does not extend to the entire toric Fano surface $X$.
After a suitable translation, one can identify the given moment torus fiber $L$ with $S^1(1)\times S^1(1)\subseteq (\mathbb{C}^*)^2$ and $T^*L\cong (\mathbb{C}^*)^2$. Notice that $L$ is the fixed locus of the following anti-holomorphic involution 
   \begin{align} \label{1005}
      \tilde{\sigma}:(z_1,z_2)\mapsto \left(\frac{1}{\bar{z}_1},\frac{1}{\bar{z}_2}\right).
   \end{align} 
  
Making use of \eqref{1005}, any holomorphic disc $f:(D^2,\partial D^2)\rightarrow (X,L)$ can be doubled to give a rational curve. More precisely, one can double the open part of the holomorphic disc in $(\mathbb{C}^*)^2$ and extend it to a rational curve via valuation criterion of properness. Although the map (\ref{1005}) does not extend to $X$, we will still call a rational curve fixed by $(\ref{1005})$ a \textit{real rational curve} throughout the section.
   
On the other hand, one can obtain a tropical curve from a given tropical disc by extending its affine edge that is adjacent to the end. Readers are warned that it is \textit{not} analogous to the above procedure of extending the holomorphic disc to a rational curve. Instead, one notices that there is the identity 
       \begin{align*}
          Log\circ \tilde{\sigma}=-Log.
       \end{align*}  
The corresponding operation in tropical geometry is to double the tropical disc via $\mathbb{Z}_2$ reflection symmetry with respect to the end.
  \begin{remark}
     The involution $\tilde{\sigma}$ is the composition of the complex conjugation and a Cremona transformation. Note that the latter is only a birational map. In particular, $\tilde{\sigma}$ cannot be extended globally to $X$ and is not a restriction of a global anti-symplectic involution on $X$. It is also worth keeping in mind that the other half of the doubled disc can have different Maslov index from the original disc, and is always torically non-transverse. 
  \end{remark}     
        
\begin{example}
    Given a Maslov index four disc in $\mathbb{P}^2$. Without loss of generality, we may assume that the image is contained in $\mathbb{C}\subset \mathbb{P}^2$ after a suitable coordinate change. By the classification theorem in \cite{CO}, one can parametrize the disc as 
       \begin{align*}
          \{z| |z|\leq 1\}\mapsto \left(\rho_1z,\rho_2 \frac{z-\alpha}{1-\bar{\alpha}z} \right)\in \mathbb{C}^2\subseteq \mathbb{P}^2,
       \end{align*} 
       up to domain automorphisms, and where $|\rho_1|=|\rho_2|=1$ and $|\alpha|\leq 1$. The image of the disc satisfies an algebraic equation 
          \begin{align}\label{eqn:rho1rho2ext}
             \frac{y}{\rho_2} \left(1-\bar{\alpha}\frac{x}{\rho_1} \right)=\frac{x}{\rho_1}-\alpha.
          \end{align} 
                    \begin{figure}[htb!]
    \includegraphics[scale=0.3]{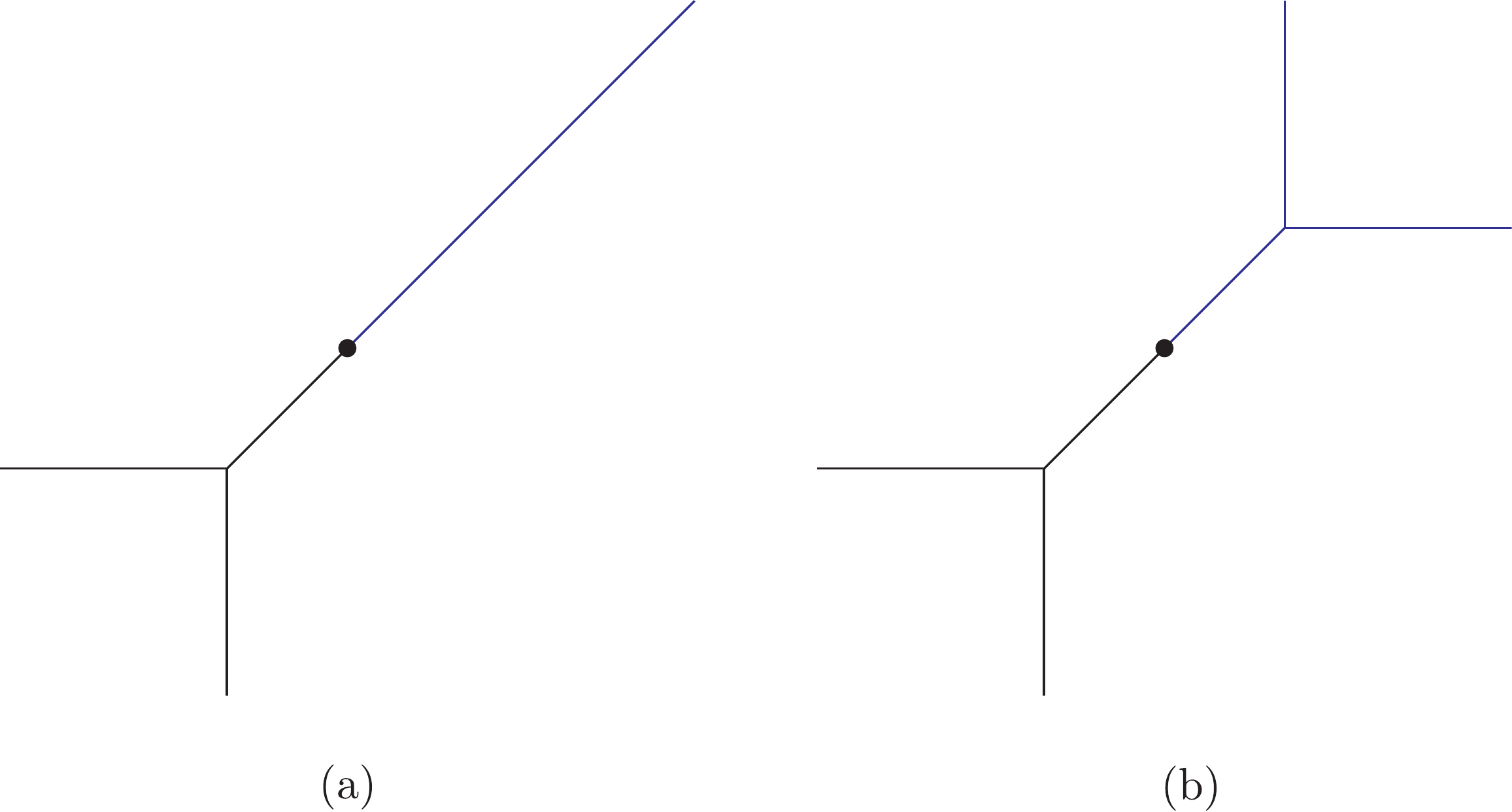}
    \caption{The tropicalization of the extended rational curve}
		\label{fig:extendedcurve}
\end{figure}
          This is a degree two curve in $\mathbb{P}^2$ which is set-theoretically fixed by $\tilde{\sigma}$. On the other hand, the naive extension of the tropical disc gives a degree one tropical rational curve of $\mathbb{P}^2$ (see (a) of Figure \ref{fig:extendedcurve}). 
          
          The tropicalization of the extended rational curve \eqref{eqn:rho1rho2ext} is depicted in (b) of Figure \ref{fig:extendedcurve}. 
          Two unbounded edges are not in the directions in the fan of $\mathbb{P}^2,$ because the extended rational curves always hit the $0$-dimensional toric stratum and are never torically transverse.
          In general, a holomorphic disc intersecting the coordinate lines of $\mathbb{P}^2$ $a,b,c$ times will be extended to a rational curve of degree $a+b+c$.  
\end{example}

 A priori, there might be real rational curves without real loci. 
However, the following lemma excludes this possibility.
Although we will not use this fact, we present it here, as it gives a good way to understand real curves as doublings of holomorphic discs.

 \begin{lemma}\label{1300}
   Any real irreducible rational curve admits non-trivial real locus. In other words, every real rational curve is a doubling of a holomorphic disc with boundary on a moment torus fiber.  
 \end{lemma}
 \begin{proof}
   Given an irreducible rational real curve $C\subseteq (\mathbb{C}^*)^2$ with respect to $\sigma$, we claim that there exists an involution $\sigma$ on the normalization $\tilde{C}$ of $C$ such that $p \circ \sigma =\tilde{\sigma}\circ p$, where $p:\tilde{C}\rightarrow C$ is the normalization. Let $\bar{C}$ be the conjugate holomorphic curve of $C$. Then there exists morphisms $\tau:C\rightarrow \bar{C}$ and $\bar{\tau}:\bar{C}\rightarrow C$. From the universal property of the normalization, there exists a natural lifting of $\tau\times \bar{\tau}$ which makes the following diagram commutative
   \begin{align*}
      \xymatrix{
          \tilde{C}\times \tilde{\bar{C}} \ar[d]\ar[r] & \bar{\tilde{C}}\times C \ar[d] \\
          C\times \bar{C}\ar[r]_{\tau\times \bar{\tau}}& \bar{C}\times C
                   } .
   \end{align*} 
Here, we used the fact that $\tilde{\bar{C}}=\bar{\tilde{C}}$ and the product of normalization is the normalization of the product. Then the lifting of $\tau\times \bar{\tau}$ is independent of the second factor. Indeed, this is true on a Zariski open subset of $\tilde{C}$ and thus holds in general. The claim is proved. 
    
 Then it suffices to prove that $\mbox{Log}(C)$ contains $(0,0)$. Otherwise, let $\tilde{C}$ be the normalization of $C$. There exists $x\in \tilde{C}$ such that $\mbox{Log}(\tilde{\sigma}(x))\neq \mbox{Log}(x)$. Since $C$ is path connected, there exists a path $\phi$ connecting $x$ and $\tilde{\sigma}(x)$. Then $\tilde{\sigma}(\phi)\cup \phi$ gives a loop in $\tilde{C}$. The loop under Log map is non-contractible in $\mathbb{R}\backslash \{(0,0)\}$ and hence, so is the loop in $\tilde{C}$. This contradicts the fact that $\tilde{C}$ is a rational curve. 
 \end{proof}
 Similarly, one has a tropical version of Lemma \ref{1300}, whose proof is elementary.
\begin{lemma} \label{1036}
   Let $(h,T,w)$ be a real tropical curve with respect to $\sigma$, then $(0,0)$ is in the image $h(T)$. 
\end{lemma}
  
\subsection{Tropical-holomorphic correspondence theorem}

We now establish the correspondence between holomorphic bulk-deformed potentials and tropical ones as well as their associated wall structures, $\mathcal{W}^{holo}$ (\ref{def:holo_Wk}) and $\mathcal{W}^{trop}$ (Definition \ref{1033}).
First of all, the following describes how the two walls are related.

\begin{theorem}\label{thm:wall}
  Given $q_1,\cdots, q_k\in (\mathbb{C}^*)^2\subset X$ in generic position and set $p_i=\mbox{Log}(q_i) \in \mathbb{R}^2$. Then there exists a neighborhood $\mathcal{U}_t$ of $\mathcal{W}^{trop}$ such that 
    \begin{enumerate}
       \item $\mathcal{U}_t$ deformation retracts to tropical wall structure $\mathcal{W}^{trop}$ associated to the $k$ point constraints $p_1,\cdots, p_k$. 
       \item The wall structure associated to $H_t^{-1}(q_1)\cdots, H_t^{-1}(q_k)$ contains in $\mathcal{U}_t$. 
       \item $\cap_{t\rightarrow \infty}\mathcal{U}_t$ is the tropical wall structure. 
    \end{enumerate}
\end{theorem}

We begin by showing that the walls of holomorphic discs converge to walls of tropical discs.
The following estimate due to Mikhalkin \cite[Lemma 8.5]{M2} will be used crucially in the proof. 
\begin{lemma}\label{9998}
For a rational curve $C$ in $(\mathbb{C}^\ast)^2$,  $\mbox{Log}_t(C)$ lies inside a $\frac{d}{\log{t}}$-neighborhood of $C^{trop}$, where $d$ is a constant only depending on the Newton polytope $\Delta$ of $C^{trop}$.
\end{lemma}

\begin{prop} \label{1081}
 Given $q_1,\cdots, q_k\in (\mathbb{C}^*)^2$ in generic positions, $p_i=\mbox{Log}{(q_i)}\in \mathbb{R}^2$ and $\epsilon>0$, then there exists $T>0$ such that  $\mathcal{W}_t^{holo}$ is contained inside an $\epsilon$-tubular neighborhood $N_{\epsilon}\big(\mathcal{W}^{trop}\big)$ of the tropical wall structure. 
\end{prop}
\begin{proof}
 From Remark \ref{1035}, there are only finitely many tropical discs of generalized Maslov index zero, up to elongation of the ends. We denote them by $h_1,\cdots, h_l$. Let $\tilde{h}_1,\cdots, \tilde{h}_l$ be the doubling tropical curves. Moreover, all the above tropical discs $h_1,\cdots, h_k$ are trivalent except the vertices mapped to $p_i$. Also, there exists a constant $r>0$, depending on the position of $p_i$, such that for any $p'_i\in B(p_i,r)$, there exist exactly $l$ tropical discs $h'_1,\cdots, h'_l$ of generalized Maslov index two passing through $p'_1,\cdots, p'_k$ and $h_i$ are of the same type with $h'_i$ and $\mbox{Mult}(h_i)=\mbox{Mult}(h_i')$.

Now suppose that there exists a holomorphic disc $h:(D^2,\partial D^2)\rightarrow (X,L_u)$ passing through $H^{-1}_t(q_i)$ of generalized Maslov zero and with boundary on $L_u$. The Maslov index of the disc is at most $2k$. We give a more concrete description of the doubling $C$ of $h$. $h$ can be written in terms of toric homogeneous coordinates as
\begin{equation}\label{eqn:hblaschke}
\tilde{h} :  (D^2,\partial D^2) \to (\C^N \setminus Z(\Sigma), \tilde{L_u}) \qquad z_j (\tilde{h}) = c_j \cdot \prod_{k=1}^{\mu_j} \frac{z- \alpha_{j,k}}{1- \bar{\alpha_{j,k}} z}
\end{equation}
for $c_j \in \C^\ast$ and non-negative integers $\mu_j$ for each $j=1, \cdots N$. 
%
where $Z(\Sigma)$ is the fixed loci of the action of the algebraic torus $D(\Sigma)$ on $\C^N$ such that $X$ is the quotient of the complement. We refer readers to \cite[Section 2]{CO} for more details. Then the parametrization of $C$ is locally given by getting rid of denominators using the torus action if necessary.

If two discs have the same intersection pattern with the boundary divisors, they share the same types of factors in \eqref{eqn:hblaschke}, and hence, one can choose $1$-parameter families of the coefficients $c_j$ and $\alpha_{j,k}$ in \eqref{eqn:hblaschke}  connecting the corresponding expressions of the two discs. Therefore their doublings are also isotopic to each other. Since there are only finitely many possible intersection patterns of a holomorphic disc once we fix Maslov index $2k$, finitely many homology classes 
can appear when taking the doubling such discs. In particular, the degree of the doubling rational curve is bounded.


%

In conclusion, there are only finitely many possible intersection pairings of the curve $C$ with each of the toric boundary divisor. If real rational curve $C$ is defined by $\sum_{i,j}a_{ij}x^iy^j=0$ in $(\mathbb{C}^*)^2$ and $\Delta$ is the associate Newton polytope, there exists only finitely many such $\Delta$ up to translations. 

The tropicalization $C^{trop}\subseteq \mathbb{R}^2$ is defined as the corner locus of \begin{align*} \mbox{max}_{i,j\in \Delta} ix+jy+\log{|a_{ij}|}. \end{align*}
Since the tropicalization $C^{trop}$ of $C$ passes through some $p'_i\in B(p_i,r)$ when $\frac{d}{\log{t}} <r$, $C^{trop}$ must be of the same type of one of $h_i$. Therefore, the curve $C^{trop}$ falls in a $\frac{d'}{\log{t}}$-neighborhood of $\tilde{h}_i$ by Lemma \ref{9998}, and hence $\mbox{Log}_t(h(D^2))$ is contained in a $\frac{d'}{\log{t}}$-neighborhood of $h_i$. We remark that $d'$ is some constant depending only on $\Delta$.
In particular,  one has that $\frac{u}{\log{t}}$ lands in a $\frac{d'}{\log{t}}$ neighborhood of a wall of tropical discs, and one can take $\epsilon=\frac{d'}{\log{t}}$.

\end{proof}

We next discuss the computation of the wall-crossing of the extended wall. 
\begin{prop}\label{1301}
    There is a one-to-one correspondence 
      \begin{align*}
           \{\text{terms in $W_{k}(p_{k+1})$}\} \longleftrightarrow \{\text{the walls of $W_{k+1}$ passing through $p_{k+1}$} \}.
      \end{align*}
      Moreover, if $W_k(p_{k+1})=\sum_{\beta}N_{\beta}z^{\partial \beta}$. The wall passing through $p_{k+1}$ responsible for $N_{\beta}z^{\partial \beta}$ induces the wall-crossing  transformation
         \begin{align} \label{eqn:cluster_trans}
             z^{\partial \alpha}\mapsto z^{\partial \alpha}(1+N_{\beta}z^{\partial\beta})^{\langle \alpha,\beta\rangle}. 
         \end{align} 
 \end{prop}     
 \begin{proof} 
 Let $t\gg 1$, and take $u_+$ and $u_-$ in a  complement of $N_{\epsilon}(W^{trop})$ intersected with a small neighborhood of $q_{k+1}$. Assume that both $W_k(u_+)$ and  $W_k(u_-)$ are well-defined. 
Choose a path $\phi(t), t\in [0,1]$ from $u_-$ to $u_+$ crossing the wall as in Figure \ref{fig:path} and a family of diffeomorphisms $\phi_t$ such that $\phi_1(L_{u_-})=L_{u_+}$. Let  
    $$\mathcal{M}_{k,1,\beta}^{\phi}=\bigcup_{t\in [0,1]}\mathcal{M}_{k,1}(L_{u_-},\beta, J_t)$$
 be the one-parameter family of holomorphic discs with respect to the almost complex structure $J_t=(\phi_t^{-1})_*J$.
    
  \begin{figure} 
  \includegraphics[scale=0.5]{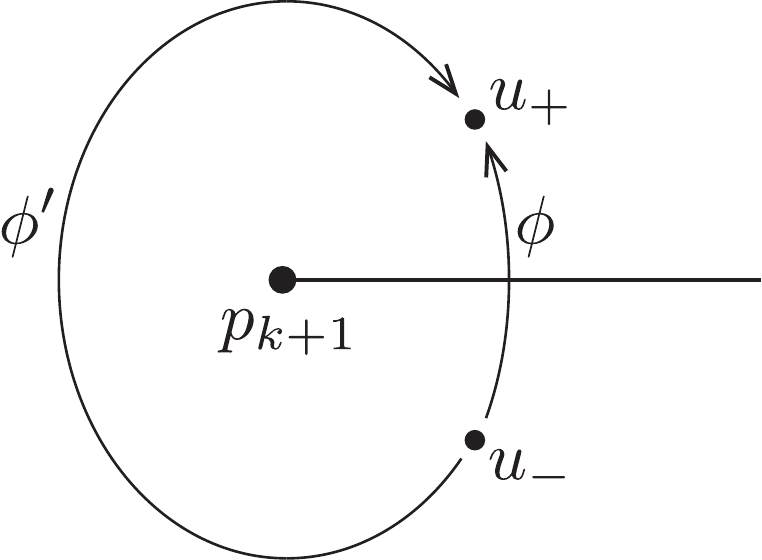}
  \caption{The path $\phi$}\label{fig:path}
  \end{figure}

     Then from \eqref{eqn:pseudo-iso} and Lemma \ref{lem:divisor}, we have $W(u_+)=TW(u_-)$, where $T$ is the transformation 
       \begin{align*}
           z^{\partial \alpha}\mapsto z^{\partial\alpha} (1+n_{\beta}z^{\partial \beta}t_{k+1})^{\langle \alpha,\beta\rangle}.
       \end{align*}
Here,  $n_{\beta}$ is given by      \begin{align*}
         n_{\beta}=Corr_*(\mathcal{M}_{k,1,\beta}^{\phi},pt;tri,tri)(1),
      \end{align*} 
      from \eqref{eqn:harmonic1}. We refer the reader to \cite[Section 4]{F1} for the definition of the correspondences.

    Let $\phi'$ by another path with the same end points such that $\phi^{-1}\circ \phi'$ gives a loop containing $q_{k+1}$. We denote the correspondence for $\phi'$ by $n_{\beta}'$. One chooses an isotopy of paths $\phi^s$ between $\phi$ and $\phi'$ and a two parameter family of diffeomorphism $\phi_{s,t}$. We denote by $$\mathcal{M}_{k,1,s,t}=\bigcup_{s,t}\mathcal{M}_{k,1}(L_{u_-},\beta, J_{s,t})$$ 
the union of moduli spaces of holomorphic discs with respective to the family of almost complex structure $J_{s,t}=(\phi_{s,t})^{-1}_*J$. There exist a Kuranishi structure on $\mathcal{M}$ compatible with that on $\partial \mathcal{M}_{k,1,s,t}$. By Stokes' theorem of the correspondences proved in \cite[Proposition 4.2]{F1}, we have
      \begin{align} \label{1079}
         n_{\beta}-n_{\beta}'= Corr_*(\partial_1 \mathcal{M}_{k,1,s,t},pt;tri,tri)(1),
      \end{align} where $\partial_1\mathcal{M}$ denotes the real codimension one boundary of the moduli space $\mathcal{M}.$

     One notices that $\partial_1 \mathcal{M}_{k,1,s,t}$ occurs due to the degeneration of the holomorphic discs when the interior marked point $q_{k+1}$ moves to the boundary of the generalized Maslov index zero holomorphic disc to result in the constant disc bubble. In other words, $\partial_1 \mathcal{M}_{k,1,s,t}$ consists of the holomorphic discs with a component of generalized Maslov index $2$ disc passing through a subset of $\{p_1,\cdots, p_k\}$ and a constant disc bubble on the boundary with an interior marking mapped to $q_{k+1}$. The compatibility of Kuranishi structures implies that 
         \begin{align} \label{1080}
            n_{\beta}-n_{\beta}'=Corr_*(\mathcal{M}'_{k,1,\beta},pt:tri,tri)(1),
        \end{align} where $\mathcal{M}'_{k,1,\beta}$ is the fiber product of $\mathcal{M}_{k,1,\beta}$ and $q_{k+1}$ with respect to the evaluation map and the inclusion of $\{q_{k+1}\}$ into $X$. Then the composition formula in \cite[Proposition 4.3]{F1} computes the right hand side of (\ref{1080}) to give 
         \begin{align}\label{eqn:aftercompF}
            n_{\beta}-n_{\beta}'= \int_L Corr_*(\mathcal{M}_{k,1,\beta}(L),L;tri, ev^{int})(1).
         \end{align} 
The right hand side \eqref{eqn:aftercompF} is the coefficient $N_{\beta}$ of $z^{\partial \beta}$ in $W_k(p_{k+1})$ and in particular, $n_{\beta}-n_{\beta}'\neq 0$. Thus there exists holomorphic disc of generalized Maslov index zero passing through $q_{k+1}$ with boundary on some moment torus fiber near $q_{k+1}$. 

We conclude that the wall of holomorphic discs of the class $\beta$ is non-empty near $q_{k+1}$. From Proposition \ref{1081}, such holomorphic discs project to the $\epsilon$-neighborhood of the image of the tropical disc $h.$ In particular, there is no such holomorphic disc with boundary on $L_{\phi'(t)}$ and thus $n_{\beta}'=0, n_{\beta}=N_{\beta}$, which finishes the proof.   
       
 \end{proof}
 
 \begin{proof}[Proof of Theorem \ref{thm:wall}]
   The wall of tropical discs consists of the extended walls and the ones added due to  scattering described in Section \ref{1084}. Having Proposition \ref{1301} in mind, it suffices to show that there are wall of holomorphic discs converge to the latter ones. 
   
   Assume that there exist two walls of tropical discs $\mathfrak{d}_1,\mathfrak{d}_2$ intersect at $p$. Let $u_+,u_-$ be two points on the complement of the tubular neighborhood of the tropical wall structure near $p$ and on the different sides of the wall of tropical discs $\mathfrak{d}$ emanating from $p$. Choose a path $\phi$ connecting $u_+,u_-$ and intersecting $\mathfrak{d}$. Choose another path $\phi'$ connecting $u_+,u_-$ such that $\phi^{-1}\circ \phi'$ is a simple loop around $p$. Then from Proposition \ref{1086}, we have $F_{\phi}=F_{\phi'}$. By induction on the number of  edges of the tropical discs, we have  
     \begin{align*}
        F_{\phi}=\mathcal{K}_{\beta_1}\mathcal{K}_{\beta_2}\mathcal{K}_{\beta_1}^{-1}\mathcal{K}_{\beta}^{-1} \neq \mbox{Id}. 
     \end{align*} 
 Then Theorem \ref{thm:wall} is proved by induction on the number of the edges of tropical discs. 
\end{proof}

Now we move onto studying relation between holomorphic and tropical bulk-deformed potentials. The following statement says that there is a tropical-holomorphic correspondence between the rescaled holomorphic bulk potential $W_k(u)$ and tropical bulk potential $W_k^{trop}(u)$.

\begin{theorem} \label{9999}
   Let $X$ be a toric Fano surface. Let $q_1,\cdots, q_k\in (\mathbb{C}^*)^2\subseteq X$ be generic points and $p_i=Log(q_i)$. Then one has that $W_k(u)=W_k^{trop}(u)$ if $u\in \mathcal{U}_t$ and $t\gg 1$. 
\end{theorem}

\begin{proof}
Fix $q_1,\cdots, q_k\in (\mathbb{C}^*)^2$ in general position. It is well-known that $W_0(u)=W^{trop}_0(u)$ for every $u\in \Int(P)$ \cite{CO}. The holomorprhic wall and tropical wall structures are identified in Theorem \ref{thm:wall} for $t\gg 1$. 

Along the same line as Section \ref{1084}, it suffices to prove that the transformation associated to each wall shows the same behavior as tropical counting of open Gromov-Witten invariants. We will prove the statement by induction on $k$.  If the wall is an extended wall, then the claim is true 
by the induction hypothesis and Proposition \ref{1301}. If the wall is produced by scattering, it follows from Proposition \ref{1086} and Lemma \ref{1078}.
\end{proof}


Finally, we study  the limiting behavior of $W_k^{\epsilon} (u)$ for an 1-parameter family of  $k$ points $  q_1^{\epsilon}, \cdots,   q_k^{\epsilon}$ satisfying $\lim_{\epsilon \to 0} \, \mathrm{Log} (q_i^\epsilon) =0$. From tropical-holomorphic correspondence \ref{thm:correspondence}, it is enough to consider the tropical counter part. Let $\epsilon p_1,\cdots,\epsilon p_k\in \R^2$ be the point constraints under the log map, which converge to the origin as $\epsilon\rightarrow 0^+$, and denote by $W_k^{trop,\epsilon}$ the corresponding tropical potential. The tropical wall structures converge to rays emanating from the origin. 

There is a wall $\mathfrak{d}$ in the direction  of $\partial \beta$  and the corresponding slab function has a term $n_{\beta}z^{\partial \beta}\prod_{j=1}^l t_{i_j}$ if and only if there exist rigid tropical curves passing through $p_{i_1}\cdots, p_{i_l}$ when $\epsilon\ll 1$ with an unbounded edge in the direction of  $ \partial \beta$ (see (a) of Figure \ref{fig:extendedcurve}). Let $X'$ be the toric surface obtained by blowing up $X$ in such a way that the holomorphic curves corresponding to these tropical ones become torically transverse.
These tropical curves are all in the same homology class, say $C_{\beta}\in H_2(X';\mathbb{Z})$. Then one can think of the coefficient $n_{\beta}$ as a weighted count of tropical curves in  $X'$ passing through generic $l$ points in the homology class ${C_\beta}$, the proper transformation of $C_{\beta}$. By the correspondence theorem of Nishinou-Siebert \cite{NS}, the tropical counting of curves in the class ${C_\beta}$ is independent of the generic position of such $l$ points. 

Therefore  if one base changes from $R_k$ to $\C[t]/t^{k+1}$ via 
 \begin{equation}
  \C[t]/t^{k+1} \rightarrow R_k, \ \  t\mapsto \sum_{i=1}^k t_i. \nonumber
  \end{equation}
 it is natural to define the limiting tropical superpotential $W^{\lim}_k(u)$ as   \begin{equation}
      W_k^{trop, \lim}:=\lim_{\epsilon\rightarrow 0} W_k^{trop, \epsilon}(u) \otimes_{R_k} \C[t]/t^{k+1}, \nonumber
   \end{equation} 
   where $t$ is a formal variable with $t^{k+1}=0$. Finally, by Theorem \ref{9999}, there is also a well-defined limiting scaled holomorphic bulk-potential
      \begin{equation}\label{eqn:Wlim}
      W_k^{\lim}:=\lim_{\epsilon \rightarrow 0} W_k^{\epsilon} (u) \otimes_{R_k} \C[t]/t^{k+1}.
   \end{equation} 
Thus, we obtain the following. 

\begin{prop}
   Given $\epsilon p_1,\cdots, \epsilon p_k \in \mathbb{R}^2$ generic and $\epsilon \to 0$,  the limiting wall structures and bulk superpotential $W^{trop, \lim}_k(u)$ are independent of $\{p_1, \cdots, p_k\}$. In particular,   $W^{\lim}_k(u)$ becomes a well-defined function in $\Q[t]/(t^{k+1})\otimes \mathbb{Q}[z_1^{\pm},z_2^{\pm}]$ under the base change map.

\end{prop}

\subsection{Relation to other works}\label{subsec:reltother}

We finish the section with a brief explanation on the relation between the result here and other existing ones. 

\subsubsection{Bulk-deformation by torus-invariant insertions}
Recall that Fukaya-Oh-Ohta-Ono defined the bulk-deformed superpotential with torus-invariant insertions for general toric manifolds \cite{FOOO_bulk}. Since all the relevant moduli spaces are torus-invariant, such bulk-deformed superpotential is independent of the Lagrangian boundary condition, and in particular, it does not experience a wall-crossing. Fukaya has pointed out, however, that when we change one tours-invariant bulk-insertion to another, a homotopy between them can not be torus-invariant in general. We believe that such a homotopy will also produce some transformation between the associated superpotentials which is similar to our wall-crossing formula.

Consider the case of toric Fano surface and denote $W^{q_1,\cdots, q_k}(u)$ and $W^{q_1',\cdots,q_k'}_{FOOO}$ be the bulk-deformed potentials in this article and in \cite{FOOO_bulk} respectively. We conjecture that 
      \begin{align}\label{979}
        \lim_{t\rightarrow \infty} W^{q^t_1,\cdots,q^t_k}(u)=W_{FOOO}^{q_1',\cdots, q_k'},
      \end{align} if $q_i'$ are distinct and $\lim_{t\rightarrow \infty} q^t_i=q_i'$. 
 For instance, assume that there is only one bulk point constraint $q_1^t \in \mathbb{P}^2$ and $\lim q_1^t=[0:0:1]$. Then the projection $p_1^t$ moves away to infinity along the negative diagonal direction as $t\to \infty$, and hence the region $C$ in Figure \ref{fig:bulkP2} will exhaust the whole plane at the limit. It is straight forward to check that $W_{FOOO}^{[0:0:1]}$ coincide with $W_{1,u_C}$ (in the first order) given in Proposition \ref{prop:firstorderWp2}. 
 
 It is worth noticing that when there are torus-invariant bulk insertion $q_i'=q_j'$, the relevant moduli spaces would have excess dimension, which makes the right hand side of (\ref{979}) harder to compute. While the left hand side of (\ref{979}) may not be well-defined due to the relative positions when taking limit. Thus, it still remain interesting to ask how to compute the bulk-deformed potential defined in \cite{FOOO_bulk} from the one in this paper.

\subsubsection{The dgla structure on the polyvector fields with a descendant variable}
Interestingly, the recent work of Chan-Ma \cite{Chan_Ma} seems to suggest an interpretation of our tropical counting for generalized Maslov zero and two discs in terms of the asymptotic analysis for the Maurer-Cartan equations defined on the mirror dga of polyvector fields as follows.

Given a Landau-Ginzburg model $W \colon X^{\vee}:=(\C^*)^n \rightarrow \C$ for any toric Fano variety $X,$ there is a dgla structure defined on its polyvector fields $PV_{X^{\vee}}^{*,*}$, which governs the deformation of complex structures on $X^{\vee}$ together with the choice of a holomorphic volume form $\Omega.$ We set $\Omega:=\frac{dx_1}{x_1}\wedge \cdots \wedge\frac{dx_n}{x_n}$ and define the polyvector fields cochain complex to be
\begin{equation}
PV^{*,*}_{X^{\vee}} \cong \Omega^{0, *}(X^{\vee}, T^{*,0}). \nonumber
\end{equation}
There is an isomorphism $PV^{p,q}(X^{\vee}) \cong \Omega^{n-p,q}(X^{\vee})$ depending on the choice of $\Omega,$ which locally is defined as
\begin{equation}
( \partial_I \vdash \Omega) := \iota_{\partial_I} \Omega, \ \ \partial_I:=\frac{\partial}{\partial_{z_{i_1}}} \wedge \cdots \wedge \frac{\partial}{\partial_{z_{i_n}}} \text{ for some }n. \nonumber
\end{equation}
Under such an isomorphism, the differentials $\bar{\partial}+dW\wedge$ and $\partial$ in the twisted de Rham complex become derivations of degree $1$ and $-1$ respectively if we grade the elements in $PV^{i,j}_{X^{\vee}}$ by $i+j$,
\begin{equation}
\bar{\partial}+\iota_{dW} \colon PV^{*,*}_{X^{\vee}}  \rightarrow PV^{*,*+1}_{X^{\vee}} \text{ and } \partial:=\Delta=(\vdash \Omega)^{-1}\circ \partial \circ (\vdash \Omega) \colon  PV^{*,*}_{X^{\vee}}  \rightarrow PV^{*-1,*}_{X^{\vee}} \nonumber.
\end{equation}
To obtain a well-defined differential of degree $1$, one can consider the polyvector fields with descendants
\begin{equation}
PV^{*,*}_{X^{\vee}}[[\hbar]], \ \ Q:=\bar{\partial}_W+\hbar \partial \nonumber,
\end{equation}
where $\bar{\partial}_W:=\bar{\partial}+\{W, \cdot\}$, the bracket $\{ \cdot, \cdot \}$ denotes the Schouten-Nijenhius bracket and $\hbar$ is a formal variable of degree $2$. There is a dgla structure on $PV^{*,*}_{X^{\vee}}[[\hbar]]$ whose Maurer-Cartan equation is written as
\begin{equation} \label{eqn:PV_MC}
Q(\Xi^{1,1}+\hbar\Xi^{0,0})+\frac{1}{2}\{\Xi^{1,1}+\hbar\Xi^{0,0}, \Xi^{1,1}+\hbar\Xi^{0,0}\}=0. 
\end{equation}
One verifies that this Maurer-Cartan equation is equivalent to the following two equations,
\begin{eqnarray}
&& \bar{\partial}_W(\Xi^{1,1})+\frac{1}{2}\{ \Xi^{1,1}, \Xi^{1,1}\}=0, \nonumber\\
&&  \bar{\partial}\Xi^{0,0} +\{W,\Xi^{0,0}\}+\{\Xi^{0,0}, \Xi^{1,1}\}=0, \nonumber
\end{eqnarray}
where the first says that $\Xi^{1,1} \in H^1(X^{\vee}, T_{X^{\vee}
})$ is the Kodaira-Spencer class associated to a deformation of the complex structure on $X^{\vee}$ and the second equation means that the top form $e^{\Xi^{0,0}}\Omega$ is a holomorphic volume form in the new complex structure. We remark that this dgla on polyvector fields with descendants on the B-side have been considered in \cite{LLS} and \cite{Costello_Li}. It is recently proved by Chan-Ma \cite[Theorem 1.1]{Chan_Ma} that the asymptotic solutions of the Maurer-Cartan equations for this dgla can be obtained order by order locally in $PV^{*,*}_{X^{\vee}}\otimes_{\C}R_k$ for all $k$. Moreover, they also define a \textit{tropical dgla} on the A-side, which is mirror to the polyvector field with descendants. In terms of the solutions $\Xi^{0,0}$ and $\Xi^{1,1}$ to \eqref{eqn:PV_MC} in each order $k,$ one has that $\Xi^{0,0}$ and $\Xi^{1,1}$ corresponds to the generalized Maslov zero and the generalized Maslov two discs that we define. From a symplectic geometric point of view, such dgla or $L_{\infty}$ structure can be conceivably obtained from the TQFT structures on the symplectic/logarithmic cohomology of the log Calabi-Yau pair $(X, D)$ considered by Ganatra-Pomerleano in \cite{GP1, GP2}.

\section{The big quantum period theorem}\label{sec:period_thm}
As the corollary of the tropical/holomorphic correspondence of holomorphic discs, we follow Gross's argument to achieve following the  big quantum period theorem which provides a relationship between log Gromov-Witten correlation function and oscillatory integrals of the $k$th order bulk-deformed potential $W_k$ computed in Section \ref{sec:HTcorres}.
\begin{theorem} \label{thm:period_thm}
Let $X$ be a toric Fano surface and let $L$ be any Lagrangian torus fiber of the moment map $\pi.$ The bulk-deformed potential $W_k$ associated to $L$ satisfies  for any $k \in \N$
\begin{eqnarray} \label{eqn:period_thm}
&& \ \  \frac{1}{(2\pi i)^2}\int_{T^2}e^{W_k/\hbar}\frac{dz_1}{z_1}\wedge \frac{dz_2}{z_2} \\
&&  =1+\sum_{I}\sum_{m \geq 2}\sum_{\Delta: |\Delta|=m+n}  \frac{1}{|\mbox{Aut}(\Delta)|}\langle p_1,\cdots,p_n, \psi^{m-2}u\rangle^{X(\log{D})}_{\Delta,n}\hbar^{-m}t_I \nonumber
\end{eqnarray}
where $I \subset \{1,2, \cdots, k\}$ is an ordered subset such that $|I|=n$ and $\langle \cdot , \cdots, \cdot \rangle^{X(\log{D})}_{\Delta,n}$ denotes the log Gromov-Witten invariants. Similarly, one also has the following statement after base change from $R_k$ to $\C[t]/t^{k+1}$ via the ring map $t \mapsto \sum_{i=1}^k t_i,$
\begin{eqnarray} \label{eqn:period_Wlim}
&& \ \  \frac{1}{(2\pi i)^2}\int_{T^2}e^{W^{lim}_k(u)/\hbar}\frac{dz_1}{z_1}\wedge \frac{dz_2}{z_2} \nonumber \\
&&  =1+\sum_{0 \leq n \leq k}\sum_{m \geq 2}\sum_{\Delta: |\Delta|=m+n}  \frac{1}{|\mbox{Aut}(\Delta)|}\langle p_1,\cdots,p_n, \psi^{m-2}u\rangle^{X(\log{D})}_{\Delta,n}\hbar^{-m}\frac{t^n}{n!}, \nonumber
\end{eqnarray}
where $\displaystyle W_k(u)^{lim}$ is the limiting bulk-deformed potential defined in equation \eqref{eqn:Wlim}.
\end{theorem}
\begin{remark} \label{rmk:Gross}
In the case of $X=\mathbb{P}^2,$ Gross \cite{G7} proved a big quantum period theorem which relates tropical descendant invariants and oscillatory integrals of the form
\begin{equation} \label{eqn:Gross}
1+\sum_{I=\{i_1, \cdots, i_n\}}\sum_{m \geq 2}\langle q_{i_1}, \cdots ,q_{i_n}, \psi^{m-2}(\alpha_i)\rangle_{X,d}^{trop} \hbar^{-m} t_I=\int_{\Xi_i} e^{W_k/\hbar} \frac{dz_1}{z_1} \wedge \frac{dz_2}{z_2}, 
\end{equation}
where $\alpha_0=1, \alpha_1=a$ and $\alpha_2=a^2$ are generators of $H^*(\mathbb{P}^2)$. The cycles $\Xi_i$ that he considered are given by $\Xi_0=\frac{1}{(2\pi i)^2}T^2$ and some suitable relative classes $\Xi_1$ and $\Xi_2 \in H_2(X, \mathrm{Re}(W/\hbar) \ll 0; \Z)$ whose homology classes are generated by Lefschetz thimbles of superpotential $W=z_1+z_2 +\frac{1}{z_1z_2}$ (See \cite[Remark 2.42]{Gross_book} for details).\\
\indent For the first period integral $\frac{1}{(2\pi i)^2}\int_{T^2} e^{W_k/\hbar} \frac{dz_1}{z_1} \wedge \frac{dz_2}{z_2}$ for $\mathbb{P}^2,$ Theorem \ref{thm:period_thm} can also be obtained as a consequence of our Theorem \ref{thm:correspondence} and M. Gross's big quantum period theorem which relates tropical descendant invariants with oscillatory integrals as in \eqref{eqn:Gross}.  
\end{remark}

\indent A priori, the expressions of the bulk-deformed potential $W_k$ that we computed in Section \ref{sec:HTcorres} depend on the choice of chambers in the base of the SYZ fibration. However, between any two different chambers, the $k$th-order bulk-deformed potential $W_k$ differ by compositions of cluster transformations of the form \eqref{eqn:cluster_trans}. Direct computation shows that the resulting oscillatory integrals are invariant under such transformations. Hence the left hand side of equation \eqref{eqn:period_thm} is well-defined and independent of the choices of Lagrangian torus fiber $L_u=\pi^{-1}(u)$.

  \subsection*{The proof of the big quantum period theorem} 
    The proof is analogous to that of Gross \cite[Section 3]{G7} when $X=\mathbb{P}^2$.
   By the Cauchy residue formula, one first observes that the left hand side in equation \eqref{eqn:period_thm} can be interpreted as
   \begin{equation}\label{eqn:residueconst}
   \frac{1}{(2\pi i)^2}\int_{T^2}e^{W_k/\hbar}\frac{dz_1}{z_1}\wedge \frac{dz_2}{z_2} =1+ \sum_{m=2}^{\infty}\sum_{I} \frac{1}{m!} \textnormal{coefficients of } (W_k^m) \,\,  \textnormal{in front of }\hbar^{-m}t_I, \nonumber
   \end{equation}
   where $I=\{i_1, \cdots, i_n\}$ is a subset of $\{1, 2, \cdots k\}$ for some $n \leq k$ and $t_I:=t_{i_1}\cdots t_{i_n}$. 
   By our tropical-holomorphic correspondence theorem \ref{thm:correspondence}, one notices that the above equation is equivalent to the following statements,
   \begin{equation} 
   \resizebox{1\hsize}{!}{$\textnormal{coefficients of } \left( \frac{1}{m!}W_k^m \right) \textnormal{ in front of }t_I \text{ in } \C  =  \sum_{n \leq k}\langle  p_{i_{1}}, \cdots, p_{i_{n}}, \psi^{m-2}u \rangle_{\Delta,n}^{trop},$}\nonumber
   \end{equation}
   \begin{equation} 
   \resizebox{1\hsize}{!}{$\textnormal{or constant terms of } \left( \frac{1}{m!}W_k^m \right) \textnormal{ in } R_k=\sum_{n \leq k} \langle p_{i_{1}}, \cdots, p_{i_{n}}, \psi^{m-2}u  \rangle_{\Delta, n}^{trop},$}\nonumber
   \end{equation} 
   where $|\Delta|=m+n$. 
   
   \indent
    Let $(h_i, T_i,w_i)$, $i=1,\cdots, d$, be the tropical discs in the expansion of the Hori-Vafa potential $W_0$  and $(h_j', T_j', w_j')$, $j=1,\cdots, \tilde{d}$,  be the generalized Maslov index two tropical discs which contribute to the bulk-deformed potential $W_k$ but not $W_0$. Given a non-zero constant term in the expansion of $\frac{1}{m!}W_k^m$ is equivalent to a collection of $m$ tropical discs with boundary classes sum up to zero, which is responsible for the constant term in $W_k^m$. Since all such tropical discs share the same end, say $u \in \R^2$, the collection will glue to a rigid tropical descendant rational curve $(h,T,w)$ (Figure \ref{fig:highval}) which has exactly one vertex $u \in T^{[0]}$ of higher valency and 
    \begin{align*}
          (h,T,w)\in \mathcal{M}_{\Delta,n}^{trop}(X,p_{i_1},\cdots,p_{i_n},\psi^{m-2}u),
     \end{align*} for some $\Delta$ and some index set $I\subseteq \{1,\cdots, k\}$ with $n=|I|=|\Delta|-m$. 
   Conversely, given any descendant tropical curve in $\mathcal{M}^{trop} _{\Delta,n}(p_{i_1},\cdots, p_{i_k},\psi^{m-2}u)$, it is rigid if $n=|I|=|\Delta|-m$. One can also chop off at $u$ and each component is a generalized Maslov index two tropics disc due to Lemma \ref{lem:split_MI2}. \\
   \indent 
   Since $t_i^2=0$ in $R_k,$ one concludes that $(h_j')^2=0$. Let $m_i$ and $m_j'$ be the number of copies of $(h_i,T_i,w_i)$ and $(h_j', T_j', w_j')$ respectively in the given constant term. Then $m_j'$ are at most $1$. The given term appears 
   $$\binom{m}{m_1 \, \cdots \,  m_d \, \, m_1' \, \cdots \, m'_{\tilde{d}}} = \dfrac{m!}{ m_1 ! \cdots m_d! m_1' ! \cdots m'_{\tilde{d}}! } = \dfrac{m!}{m_1! \cdots m_d !}$$
   many times in $W_k^m$ and thus
    contributes the constant term of $\frac{1}{m!}W_k^m$ by
   \begin{eqnarray} \label{eqn:compconstp}
   \dfrac{1}{m!}  \cdot\dfrac{  m!}{ m_1 ! \cdots m_{d} !} \prod Mult (h_i)^{m_i}  \prod Mult ({\tilde{h}_j}) =Mult(h)\nonumber 
   \end{eqnarray}
   
   \begin{figure}[htb!]
       \includegraphics[scale=0.5]{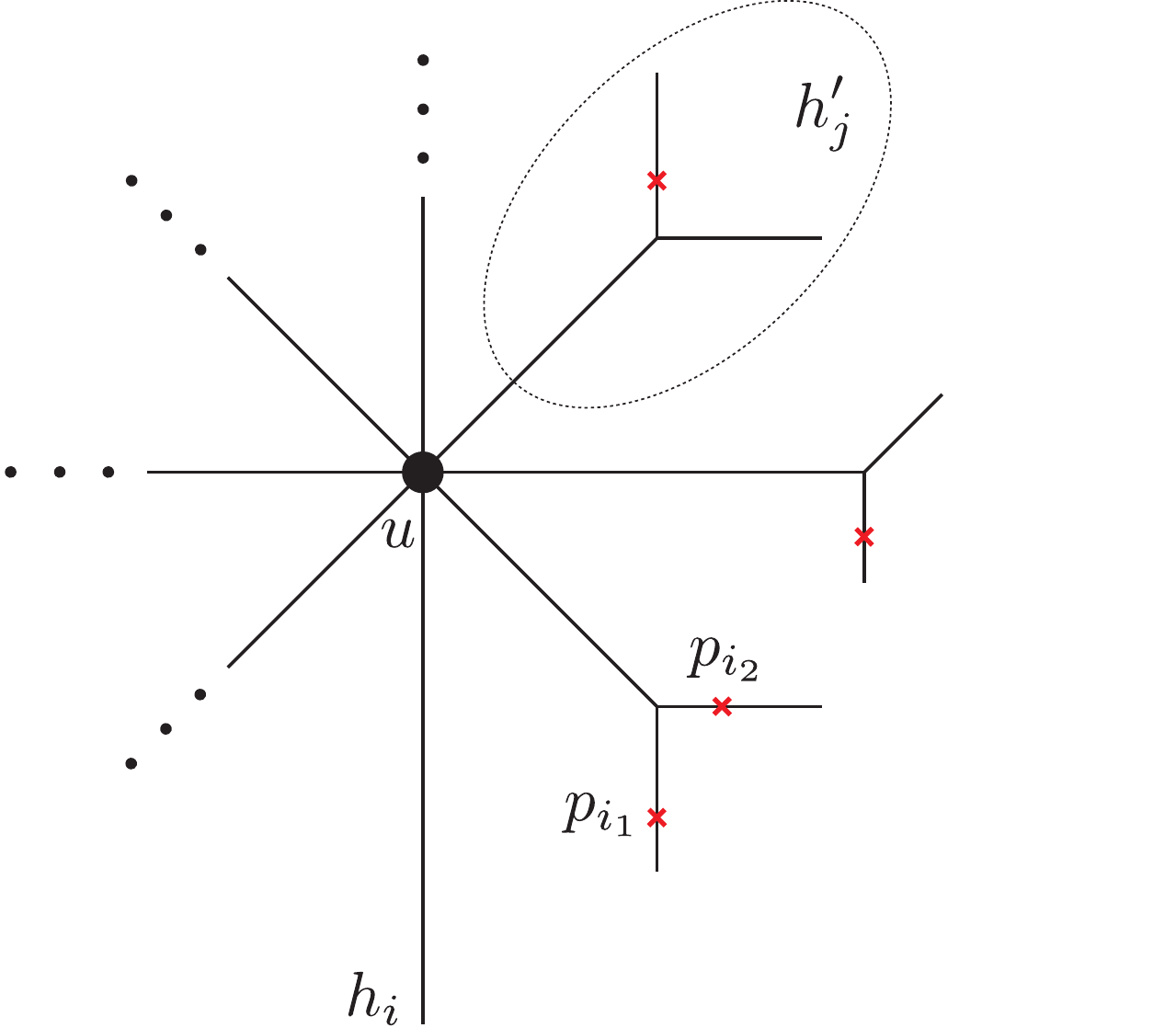}
       \caption{A descendant tropical curve $(h,T,w)$}
   		\label{fig:highval}
   \end{figure}
 Therefore, the quantum period is 
   \begin{align*}
      \frac{1}{(2\pi i)^2}\int_{T^2}e^{W_k/\hbar}\frac{dz_1}{z_1}\wedge \frac{dz_2}{z_2}=1+\sum_{I,m,\Delta} \sum_h  Mult(h) \hbar^{-m}t_I,
   \end{align*}
   where the first summation is taking over $I=\{i_1, \cdots, i_n \}\subseteq \{1,\cdots,k\},$ all $\Delta$ such that $|\Delta|=m+n$ and $m \geq 2$, and the second summation is over all $h\in \mathcal{M}^{trop}_{\Delta,n}(p_{i_1},\cdots,p_{i_n},\psi^{m-2}u)$ and $\Delta$ such that $|\Delta|=m+n$. By Definition \ref{defn:tropical_descedant}, one has that
   \begin{align*}
       & \ \ \ \ \ \frac{1}{(2\pi i)^2}\int_{T^2}e^{W_k/\hbar}\frac{dz_1}{z_1}\wedge \frac{dz_2}{z_2}\\
       &=1+\sum_{0 \leq n \leq k}\sum_{m \geq 2} \sum_{\Delta: |\Delta|=m+n} \langle p_{i_1},\cdots,p_{i_n},\psi^{m-2}u\rangle^{trop}_{\Delta,n} \hbar^{-m}t_I\\
        &=1+\sum_{0 \leq n \leq k}\sum_{m \geq 2} \sum_{\Delta: |\Delta|=m+n}  \langle p_{i_1},\cdots,p_{i_n}, \psi^{m-2}u\rangle^{trop}_{\Delta,n} \hbar^{-m} t_I\\
        &=1+\sum_{0 \leq n \leq k} \sum_{m \geq 2} \sum_{\Delta: |\Delta|=m+n} \frac{1}{|\mbox{Aut}(\Delta)|}\langle p_{i_1},\cdots,p_{i_n}, \psi^{m-2}u\rangle^{X(\log{D})}_{\Delta,n}\hbar^{-m}t_I,
   \end{align*} 
   where $t_I=t_{i_1} \cdots t_{i_n}$ and we have applied Theorem \ref{thm:MR_main} in the last equation.
   If we furthermore set $t=\sum t_i$, then we obtain that 
   \begin{eqnarray}
   && \ \ \ \frac{1}{(2\pi i)^2}\int_{T^2}e^{W_k/\hbar}\frac{dz_1}{z_1}\wedge \frac{dz_2}{z_2}\nonumber \\
   && =1+\sum_{0 \leq n \leq k}\sum_{m \geq 2}\sum_{\Delta: |\Delta|=m+n}  \frac{1}{|\mbox{Aut}(\Delta)|}\langle p_{i_1},\cdots,p_{i_n}, \psi^{m-2}u\rangle^{X(\log{D})}_{\Delta,n}\hbar^{-m}\frac{t^n}{n!}, \nonumber
   \end{eqnarray}
where we used the fact that $\langle p_{i_1},\cdots, p_{i_n},\psi^{m-2}u\rangle^{trop}_{\Delta,n}$ is independent of the positions of $p_{i_1},\cdots,p_{i_n}$.

\end{document}